\documentclass[10pt]{amsart}
\usepackage[english]{babel}
\usepackage{amssymb}
\usepackage[initials]{amsrefs}
\usepackage{mathrsfs} 			
\usepackage{enumitem}
\usepackage{color}
\usepackage[pdftex,colorlinks,linkcolor=blue,citecolor=blue,urlcolor=blue,
hypertexnames=true,plainpages=false]{hyperref}
    \hypersetup{pdfauthor={Uri Bader, Alex Furman, Roman Sauer},
                pdftitle ={Integrable measure equivalence rigidity of hyperbolic lattices},}
\usepackage[all]{xy}
\SelectTips{cm}{}

\newcommand{\bbN}{{\mathbb N}}
\newcommand{\bbQ}{{\mathbb Q}}
\newcommand{\bbR}{{\mathbb R}}

\newcommand{\bbC}{{\mathbb C}}

\newcommand{\bbF}{{\mathbb F}}
\newcommand{\bbH}{{\mathbb H}}
\newcommand{\bbO}{{\mathbb O}}

\newcommand{\calB}{\mathcal{B}}
\newcommand{\calG}{\mathcal{G}}
\newcommand{\calS}{\mathcal{S}}

\newcommand{\Hsp}{\mathbf{H}}
\newcommand{\Euc}{\mathbf{E}}

\newcommand{\rmF}{\mathrm{F}}
\newcommand{\rmI}{\mathrm{I}}

\newcommand{\rmL}{\mathrm{L}}
\newcommand{\rmH}{\mathrm{H}}
\newcommand{\rmHb}{\mathrm{H}_\mathrm{b}}

\newcommand{\rmLweak}{\rmL_{\mathrm{w}\ast}}

\newcommand{\rmC}{\mathrm{C}}
\newcommand{\rmP}{\mathrm{P}}
\newcommand{\rmR}{\mathrm{R}}
\newcommand{\rmD}{\mathrm{D}}
\newcommand{\rmCb}{\mathrm{C}_\mathrm{b}}
\newcommand{\rmCm}{\mathrm{C}^\mathrm{m}}

\newcommand{\rmPT}{\mathrm{PT}}

\newcommand{\rmSO}{\mathrm{SO}}
\newcommand{\rmSL}{\mathrm{SL}}
\newcommand{\rmPSL}{\mathrm{PSL}}

\newcommand{\rmPGL}{\mathrm{PGL}}
\newcommand{\rmSp}{\mathrm{Sp}}
\newcommand{\rmSU}{\mathrm{SU}}

\newcommand{\bs}{\backslash}
\newcommand{\comp}{\operatorname{comp}}

\newcommand{\dvol}{\operatorname{dvol}}

\newcommand{\id}{\operatorname{id}}

\newcommand{\isom}{\operatorname{Isom}}
\newcommand{\map}{\operatorname{Map}}
\newcommand{\pr}{\operatorname{pr}}
\newcommand{\supp}{\operatorname{supp}}
\newcommand{\smear}{\operatorname{sm}}
\newcommand{\vol}{\operatorname{vol}}

\newcommand{\Ker}{\operatorname{Ker}}
\newcommand{\Prob}{\operatorname{Prob}}
\newcommand{\Homeo}{\operatorname{Homeo}}
\newcommand{\rmTpl}{\operatorname{Trp}}

\newcommand{\euler}{\operatorname{eu}}
\newcommand{\Psp}{\bbR\operatorname{P}}

\newcommand{\acts}{\curvearrowright}
\newcommand{\abs}[1]{{\left\lvert #1\right\rvert}}
\newcommand{\dual}[1]{{#1}^*}

\newcommand{\overto}[1]{{\buildrel{#1}\over\longrightarrow}}

\newtheorem{mthm}{Theorem}

\newtheorem{theorem}{Theorem}[section]
\newtheorem{lemma}[theorem]{Lemma}

\newtheorem{corollary}[theorem]{Corollary}
\newtheorem{cor}[theorem]{Corollary}
\newtheorem{proposition}[theorem]{Proposition}
\newtheorem{prop}[theorem]{Proposition}
\theoremstyle{definition}
\newtheorem{definition}[theorem]{Definition}
\newtheorem{defn}[theorem]{Definition}

\newtheorem{example}[theorem]{Example}

\newtheorem{remark}[theorem]{Remark}
\newtheorem{remarks}[theorem]{Remarks}

\newtheorem{setup}[theorem]{Setup}

\numberwithin{equation}{section}

\begin{document}
\title[Integrable ME-rigidity of hyperbolic lattices]{
Integrable measure equivalence and rigidity of hyperbolic lattices}

\author{Uri Bader}
\address{The Technion, Haifa}
\email{bader@tx.technion.ac.il}
\author{Alex Furman}
\address{University of Illinois at Chicago, Chicago}
\email{furman@math.uic.edu}
\author{Roman Sauer}
\address{University of Chicago, Chicago}
\curraddr{KIT, Karlsruhe}
\email{roman.sauer@kit.edu}
\begin{abstract}
	We study rigidity properties of lattices in
	$\isom(\Hsp^n)\simeq \rmSO_{n,1}(\bbR)$, $n\ge 3$, and of surface groups
	in $\isom(\Hsp^2)\simeq \rmSL_{2}(\bbR)$ in the context of
	\emph{integrable measure equivalence}.
	The results for lattices in $\isom(\Hsp^n)$, $n\ge 3$, are generalizations of
	Mostow rigidity;
	they include a \emph{cocycle} version of strong rigidity and an
	integrable measure equivalence classification. Despite the lack 
	of Mostow rigidity for $n=2$ we show that cocompact lattices in $\isom(\Hsp^2)$ 
	allow a similar integrable measure equivalence classification. 
\end{abstract}
\maketitle

\enlargethispage{\baselineskip}
\vspace{-0.5cm}
{\tableofcontents}

\section{Introduction and statement of the main results} % (fold)
\label{sec:introduction}

\subsection{Introduction}
\label{sub:discussion_of_results}%\hfill\\ %(fold)
\emph{Measure equivalence} is an equivalence relation on groups,
introduced by Gromov~\cite{gromov-invariants} as a measure-theoretic counterpart
to quasi-isometry of finitely generated groups.
It is intimately related to \emph{orbit equivalence} in ergodic theory, to the theory of von Neumann algebras,
and to questions in descriptive set theory.
The study of rigidity in measure equivalence or orbit equivalence goes back to
Zimmer's paper~\cite{zimmer-csr},
which extended Margulis' superrigidity of higher rank
lattices~\cite{Margulis:1974:ICM} to the context
of \emph{measurable cocycles} and applied it to prove \emph{strong rigidity} 
phenomena in \emph{orbit equivalence} setting.
In the same paper \cite[\S 6]{zimmer-csr} Zimmer poses the question of whether these 
strong rigidity for orbit equivalence
results extend to lattices in rank one groups $G\not\simeq \rmPSL_2(\bbR)$;
and in \cite{zimmer-foliatedmostow} and later in a joint paper with Pansu \cite{pansu+zimmer} 
obtained some results under some restrictive geometric condition.

The study of measure equivalence and related problems has recently experienced a rapid growth,
with~\cites{FurmanOEw,FurmanME, Gaboriau-cost,gaboriau-l2,MonodShalom,hjorth,
Ioana+Peterson+Popa,popa-betti,popa:wI,popa:wII, popa-cocycle,
popa-spectralgap,Kida:thesis,Kida:ME,ioana} being only a partial list of important advances.
We refer to~\cites{Shalom:2005ECM,Popa:ICM,Furman:MGT} for surveys and further references.
One particularly fruitful direction of research in this area has been in obtaining the complete
description of groups that are
measure equivalent to a given one from a well understood class
of groups.
This has been achieved for lattices in simple Lie groups of higher
rank~\cite{FurmanME}, products of hyperbolic-like groups~\cite{MonodShalom},
mapping class groups~\cite{Kida:thesis,Kida:ME, Kida:OE}, and
certain amalgams of groups as above~\cite{Kida:amalgamated}.
In all these results, the measure equivalence class of one of such groups turns out to be small and to consist of
a list of "obvious" examples obtained by simple modifications of the original group.
This phenomenon is referred to as~\emph{measure equivalence rigidity}.
On the other hand, the class of groups measure equivalent to lattices in $\rmSL_2(\bbR)$
is very rich: it is uncountable, includes wide classes of groups and does not seem to have an explicit description
(cf.~\cites{Gaboriau:2005exmps,bridson}).

\medskip

In the present paper we obtain measure equivalence rigidity results for lattices in the least rigid
family of simple Lie groups  $\isom(\Hsp^n)\simeq \rmSO_{n,1}(\bbR)$ for
$n\ge 2$, including surface groups, albeit
within a more restricted category of \emph{integrable measure equivalence}, hereafter also called \emph{$\rmL^1$-measure equivalence} or just \emph{$\rmL^1$-ME}.
Let us briefly state the classification result, before giving the precise definitions and stating more detailed results.
\begin{mthm}\label{T:ME-rigidity2and3}
Let $\Gamma$ be a lattice in $G=\isom(\Hsp^n)$, $n\ge 2$; in the case $n=2$ assume that $\Gamma$ is cocompact.
Then the class of all finitely generated groups that are $\rmL^1$-measure equivalent
to $\Gamma$ consists of those $\Lambda$,
which admit a short exact sequence $\{1\}\to F\to\Lambda\to\bar\Lambda\to\{1\}$ where
$F$ is finite and $\bar\Lambda$ is a lattice in $G$; in the case $n=2$,  
$\bar\Lambda$ is also cocompact in $G=\isom(\Hsp^2)$.
\end{mthm}
The integrability assumption is necessary for the validity of the rigidity results for cocompact lattices in
$\isom(\Hsp^2)\cong\rmPGL_2(\bbR)$.
It remains possible, however, that the $\rmL^1$-integrability assumption is superfluous for lattices in
$\isom(\Hsp^n)$, $n\ge 3$.
We also note that a result of Fisher and Hitchman \cite{fisher}
can be used to obtain $\rmL^2$-ME rigidity results similar to
Theorem~\ref{T:ME-rigidity2and3} for the
family of rank one Lie groups $\isom(\Hsp_{\bbH}^n)\simeq{\rm Sp}_{n,1}(\bbR)$
and $\isom(\Hsp_{\bbO}^2)\simeq F_{4(-20)}$\footnote{Here $\simeq$ means \emph{locally isomorphic}.}; 
it is possible that this $\rmL^2$-integrability assumption can be relaxed or
removed altogether.

\smallskip

The proof of Theorem~\ref{T:ME-rigidity2and3} for the case $n\ge 3$ proceeds through a cocycle version of Mostow's strong rigidity theorem
stated in Theorems~\ref{T:SOn1-1-taut} and \ref{T:SOn1-L1-cocycle}.
This cocycle version relates to the original Mostow's strong rigidity theorem
in the same way in which Zimmer's cocycle superrigidity
theorem relates to the original Margulis' superrigidity for higher rank lattices.
Our proof of the cocycle version of Mostow rigidity, which is
inspired by Gromov-Thurston's proof of Mostow rigidity using simplicial volume \cite{thurston} and Burger-Iozzi's proof for dimension $3$~\cite{burger+iozzi2}, heavily uses bounded cohomology and other homological methods.
A major part of the relevant homological technique, which applies to
general Gromov hyperbolic groups,
is developed in the companion paper~\cite{sobolev};
\emph{in fact, Theorem~\ref{thm:main result about induction in cohomology} taken
from~\cite{sobolev} is the only place in this paper where we require the integrability assumption.}

Theorem~\ref{T:ME-rigidity2and3} and the
more detailed Theorem~\ref{T:ME-rigidity}
are deduced from the strong rigidity for integrable cocycles (Theorem~\ref{T:SOn1-1-taut}) using a general
method described in Theorem~\ref{T:reconstruction}.
The latter extends and streamlines the approach developed in \cite{FurmanME},
and further used in \cite{MonodShalom} and in \cite{Kida:ME}.

The proof of Theorem~\ref{T:ME-rigidity2and3} for surfaces uses a cocycle version of the fact that an abstract
isomorphism between uniform lattices in $\rmPGL_2(\bbR)$ is realized by conjugation in $\Homeo(S^1)$.
The proof of this generalization uses homological methods mentioned above and a cocycle version of the
Milnor-Wood-Ghys phenomenon (Theorem~\ref{thm:taut relative homeo}), in which an integrable ME-cocycle
between surface groups is conjugate to the identity map in $\Homeo(S^1)$.
In the case of surfaces in Theorem~\ref{T:ME-rigidity2and3}, this result is used
together with Theorem~\ref{T:reconstruction} to construct a representation $\rho:\Lambda\to \Homeo(S^1)$.
Additional arguments (Lemma~\ref{L:furstenberg} and~Theorem~\ref{T:intoPGL2}) are then needed
to deduce that $\rho(\Lambda)$ is a uniform lattice in $\rmPGL_2(\bbR)$.

Let us now make precise definitions and describe in more detail the main results.
% subsection Introduction (end)

\subsection{Basic notions}  % (fold)
\subsubsection{\textbf{Measure equivalence of locally compact groups}} % (fold)
\label{ssub:measure_equivalence_of_locally_compact_groups}
We recall the central notion of \emph{measure equivalence} which
was suggested by Gromov~\cite{gromov-invariants}*{0.5.E}.
It will be convenient to work with general unimodular,
locally compact second countable (lcsc) groups rather than just countable ones.

\begin{definition}\label{D:ME}
	Let $G$, $H$ be unimodular lcsc groups with Haar measures
	$m_G$ and $m_H$. A $(G,H)$-\emph{coupling} is a Lebesgue measure space
	$(\Omega,m)$ with a measurable, measure-preserving action of $G\times H$ such that
	there exist finite measure spaces $(X,\mu)$, $(Y,\nu)$ and
	measure space isomorphisms
\begin{equation} \label{e:ij-fd}
 i:(G,m_G)\times (Y,\nu)\xrightarrow{\cong} (\Omega,m)\quad \mbox{and} \quad j:(H,m_H)\times (X,\mu)\xrightarrow{\cong} (\Omega,m)
\end{equation}
such that $i$ is $G$ equivariant and $j$ is $H$ equivariant, that is
\[ i(gg',y)=gi(g',y) \quad \mbox{and} \quad j(hh',x)=hj(h',x) \]
	for every $g\in G$ and $h\in H$ and almost every $g'\in G$, $h'\in H$, $y\in Y$ and $x\in X$.
	Groups which admit such a coupling are said to be \emph{measure equivalent} (abbreviated ME).
\end{definition}

In the case where $G$ and $H$ are countable groups, the condition
on the commuting actions $G\acts (\Omega,m)$
and $H\acts (\Omega,m)$ is that they admit finite $m$-measure
Borel fundamental domains $X,Y\subset \Omega$
with $\mu=m|_X$ and $\nu=m|_Y$ being the restrictions.

As the name suggest, measure equivalence
is an equivalence relation between unimodular lcsc groups.
For reflexivity, consider the $G\times G$-action on $(G,m_G)$, $(g_1,g_2):g\mapsto g_1 g g_2^{-1}$.
We refer to this as the \emph{tautological self coupling} of $G$.
The symmetry of the equivalence relation is obvious. For transitivity and more details we refer
to Appendix~\ref{subs:composition}.

\begin{example}
	\label{exmp:lattices}
Let $\Gamma_1,\Gamma_2$ be lattices in a lcsc group $G$\footnote{Any lcsc group containing a lattice is necessarily unimodular.}.
Then $\Gamma_1$ and $\Gamma_2$ are measure equivalent, with $(G,m_G)$ serving as a natural
$(\Gamma_1,\Gamma_2)$-coupling when equipped with the action
$(\gamma_1,\gamma_2):g\mapsto \gamma_1 g \gamma_2^{-1}$ for $\gamma_i\in\Gamma_i$.
In fact, any lattice $\Gamma<G$ is measure equivalent to $G$, with $(G,m_G)$ serving as a natural $(G,\Gamma)$-coupling when equipped with the action
$(g,\gamma):g'\mapsto g g' \gamma^{-1}$.
\end{example}

\subsubsection{\textbf{Taut groups}} % (fold)
\label{ssub:subsubsection_name}
We now introduce the following key notion of taut couplings and taut groups.

\begin{definition}[Taut couplings, taut groups]
\label{D:M-rigidity}
A $(G,G)$-coupling $(\Omega,m)$ is \emph{taut}
if it has the tautological coupling as a factor uniquely; in other words if it
admits a unique, up to null sets, measurable map $\Phi:\Omega\to G$
so that for $m$-a.e. $\omega\in\Omega$ and all
$g_1,g_2\in G$\footnote{If one only requires equivariance for almost all $g_1,g_2\in G$ one
can always modify $\Phi$ on a null set to get an everywhere equivariant map~\cite{zimmer-book}*{Appendix B}.}
\[
	\Phi((g_1,g_2)\omega)=g_1\Phi(\omega)g_2^{-1}.
\]
Such a $G\times G$-equivariant map $\Omega\to G$ will be called a \emph{tautening map}.
A unimodular lcsc group $G$ is \emph{taut} if every $(G,G)$-coupling
is taut.
\end{definition}
The requirement of uniqueness for tautening maps in the definition of taut groups
is equivalent to the \emph{strongly ICC} property for the group in question 
(see Definition~\ref{def:strongy ICC definition} and Lemmas~\ref{L:uniq2sICC} 
and \ref{L:sICC2uniq}(1) in the Appendix~\ref{sub:strong_icc_property} for a proof of this claim).
This property is rather common; in particular it is satisfied by all center-free semi-simple Lie groups
and all ICC countable groups, i.e.~countable groups with infinite conjugacy classes.
On the other hand the existence of tautening maps for $(G,G)$-coupling is hard to obtain;
in particular taut groups necessarily satisfy Mostow's strong rigidity property.

\begin{lemma}[Taut groups satisfy Mostow rigidity]
Let $G$ be a taut unimodular lcsc group. If $\tau:\Gamma_1\xrightarrow{\cong}\Gamma_2$
is an isomorphism of two lattices $\Gamma_1$ and $\Gamma_2$ in $G$, then there exists a unique $g\in G$ so that
$\Gamma_2=g^{-1} \Gamma_1 g$ and $\tau(\gamma_1)=g^{-1}\gamma g$ for $\gamma\in\Gamma_1$.
\end{lemma}
The lemma follows from considering the tautness of the measure equivalence $(G,G)$-coupling given
by the $(G\times G)$-homogeneous space $(G\times G)/\Delta_\tau$, where
$\Delta_\tau$ is the graph
of the isomorphism $\tau:\Gamma_1\to\Gamma_2$; see Lemma~\ref{L:taut-MR} for details.

\medskip

The phenomenon, that any isomorphism between lattices in $G$ is realized
by an inner conjugation in $G$, known as \emph{strong rigidity} or \emph{Mostow rigidity},
holds for all simple Lie groups\footnote{For the formulation of Mostow rigidity above we have to assume that $G$ has trivial center.}
$G\not\simeq\rmSL_2(\bbR)$. More precisely,
if $X$ is an irreducible non-compact, non-Euclidean symmetric space
with the exception of the hyperbolic plane $\Hsp^2$, then $G=\isom(X)$ is Mostow rigid.
Mostow proved this remarkable rigidity property for uniform lattices \cite{Mostow:1973book}.
% in $\isom(X)$ for rank one cases $X=\Hsp^n_\bbR$, $X=\Hsp^n_\bbC$,
% , and then higher rank $X$.
It was then extended to the non-uniform cases by Prasad \cite{prasad} (${\rm rk}(X)=1$)
and by Margulis \cite{Margulis-Mostow} (${\rm rk}(X)\ge 2$).

In the higher rank case, more precisely, if $X$ is a symmetric space without compact and Euclidean
factors with ${\rm rk}(X)\ge 2$, Margulis proved a stronger rigidity property, which became known as \emph{superrigidity}
\cite{Margulis:1974:ICM}.
Margulis' superrigidity for lattices in higher rank, was extended by Zimmer in
the cocycle superrigidity theorem \cite{zimmer-csr}.
Zimmer's cocycle superrigidity was used in \cite{FurmanME} to show that higher rank simple Lie groups $G$ are \emph{taut}
(albeit the use of term tautness in this context is new).
In \cite{MonodShalom} Monod and Shalom proved another case of cocycle superrigidity
and proved a version of tautness property for certain products $G=\Gamma_1\times\cdots\times\Gamma_n$ with $n\ge 2$.
In \cite{Kida:ME, Kida:thesis} Kida proved that mapping class groups are taut.
Kida's result may be viewed as a cocycle generalization of Ivanov's theorem \cite{ivanov}.

\subsubsection{\textbf{Measurable cocycles}} % (fold)
\label{ssub:measurable_cocycles}
Let us elaborate on this connection between tautness and \emph{rigidity} of \emph{measurable cocycles}.
Recall that a \emph{cocycle} over a group action $G\acts X$ to another group $H$ is a map
$c:G\times X\to H$ such that for all $g_1,g_2\in G$
\[
	c(g_2g_1,x)=c(g_2,g_1 x)\cdot c(g_1,x).
\]
Cocycles that are independent of the space variable
are precisely homomorphisms $G\to H$.
One can conjugate a cocycle $c:G\times X\to H$ by a map $f:X\to H$ to produce a new cocycle
$c^f:G\times X\to H$ given by
\[
	c^f(g,x)=f(g.x)^{-1}c(g,x)f(x).
\]
In our context, $G$ is a lcsc group, $H$ is lcsc or, more generally, a Polish group,
and $G\acts (X,\mu)$ is a measurable measure-preserving action on a
Lebesgue finite measure space.
In this context all maps, including the cocycle $c$, are assumed to be $\mu$-measurable,
and all equations should hold $\mu$-a.e.; we then say that $c$ is a \emph{measurable cocycle}.

Let $(\Omega,m)$ be a $(G,H)$-coupling and
$H\times X\xrightarrow{j}\Omega\xrightarrow{i^{-1}} G\times Y$
be as~in (\ref{e:ij-fd}).
Since the actions $G\acts \Omega$ and $H\acts \Omega$ commute,
$G$ acts on the space of $H$-orbits in $\Omega$, which is naturally identified
with $X$. This $G$-action preserves the finite measure $\mu$.
Similarly, we get the measure preserving $H$-action on $(Y,\nu)$.
These actions will be denoted by a dot, $g:x\mapsto g.x$, $h:y\mapsto h.y$, to
distinguish them from the $G\times H$ action on $\Omega$.
Observe that in $\Omega$ one has for $g\in G$ and almost every $h\in H$ and $x\in X$,
\[
	gj(h,x)= j(h h_1^{-1},g.x)
\]
for some $h_1\in H$ which depends only on $g\in G$ and $x\in X$, and therefore may be denoted by $\alpha(g,x)$.
One easily checks that the map
\[
	\alpha:G\times X\to H
\]
that was just defined, is a measurable cocycle.
Similarly, one obtains a measurable cocycle $\beta:H\times Y\to G$.
These cocycles depend on the choice of the measure isomorphisms in~(\ref{e:ij-fd}),
but different measure isomorphisms produce conjugate cocycles.
Identifying $(\Omega,m)$ with $(H,m_H)\times (X,\mu)$, the action $G\times H$
takes the form
\begin{equation}\label{e:me-via-coc}
	(g,h)j(h',x)= j(h h'\alpha(g,x)^{-1},g.x).
\end{equation}
Similarly, cocycle $\beta:H\times Y\to G$ describes the $G\times H$-action on
$(\Omega,m)$ when identified with $(G,m_G)\times(Y,\nu)$.
In general, we call a measurable coycle $G\times X\to H$ that arises from a
$(G,H)$-coupling as above an \emph{ME-cocycle}.

The connection between tautness and cocycle rigidity is in the observation (see Lemma~\ref{L:coc-taut})
that a $(G,G)$-coupling $(\Omega,m)$ is taut iff the ME-cocycle $\alpha:G\times X\to G$
is conjugate to the identity isomorphism, $\alpha(g,x)=f(g.x)^{-1}g f(x)$ 
by a unique measurable $f:X\to G$.
Hence one might say that 
\begin{center}
	$G$ is \emph{taut} if and only if it satisfies a \emph{cocycle version} of Mostow rigidity.
\end{center}

\subsubsection{\textbf{Integrability conditions}} % (fold)
\label{ssub:integrability_conditions}
Our first main result -- Theorem~\ref{T:SOn1-1-taut} below -- shows that $G=\isom(\Hsp^n)$, $n\ge3$, are
$1$-\emph{taut} groups, i.e.~all \emph{integrable} $(G,G)$-couplings are taut.
We shall now define these terms more precisely.

A \emph{norm} on a group $G$ is a map $|\cdot|:G\to [0,\infty)$ so that
$|gh|\le |g|+|h|$ and $\abs{g^{-1}}=\abs{g}$ for all $g,h\in G$. A norm on a lcsc group is \emph{proper} if it is measurable and the balls with respect to this norm are pre-compact.
Two norms $|\cdot |$ and $|\cdot|'$ are equivalent if there are $a,b>0$ such that
$|g|'\le a\cdot|g|+b$ and $|g|\le a \cdot |g|'+b$ for every $g\in G$.
On a compactly generated group\footnote{Every connected lcsc group is compactly
generated~\cite{stroppel}*{Corollary~6.12 on p.~58}.}
$G$ with compact generating symmetric set $K$ the function
$\abs{g}_K=\min\{n\in\bbN\mid g\in K^n\}$ defines a proper norm, whose equivalence class does
not depend on the chosen $K$. Unless stated otherwise, we mean a norm in this
equivalence class when referring to a proper norm on a compactly generated group.

\begin{definition}[Integrability of cocycles]\label{D:Lp-ME}
	Let $H$ be a compactly generated group with a proper norm $\abs{\cdot}$ and $G$ be
	a lcsc group. Let $p\in[1,\infty]$.
	A measurable cocycle $c:G\times X\to H$
	is $\rmL^p$-\emph{integrable} if for a.e.~$g\in G$
	\[
		\int_X |c(g,x)|^p\,d\mu(x)<\infty.
	\]
	For $p=\infty$ we require that the essential supremum of $|c(g,-)|$ is finite for
	a.e.~$g\in G$.
	If $p=1$, we also say that $c$ is \emph{integrable}. If $p=\infty$, we say that
	 $c$ is \emph{bounded}.
\end{definition}

The integrability condition is independent of
the choice of a norm within a class of equivalent norms. $\rmL^p$-integrability implies
$\rmL^q$-integrability whenever $1\le q\le p$. 	
In the Appendix~\ref{sub:Lp-integrability} we show that, if $G$ is also
compactly generated, the $\rmL^p$-integrability of $c$
implies that the above integral is uniformly bounded on compact subsets of $G$.

\begin{definition}[Integrability of couplings]
     A $(G,H)$-coupling $(\Omega,m)$ of compactly generated, unimodular, lcsc groups
	is $\rmL^p$-\emph{integrable},
	if there exist measure isomorphisms as in~(\ref{e:ij-fd})
	so that the corresponding ME-cocycles $G\times X\to H$ and $H\times Y\to G$
	are $\rmL^p$-integrable. If $p=1$ we just say that $(\Omega, m)$ is \emph{integrable}.
	Groups $G$ and $H$ that admit an $\rmL^p$-integrable $(G,H)$-coupling are said to be
	$\rmL^p$-\emph{measure equivalent}.
\end{definition}

For each $p\in[1,\infty]$, being $\rmL^p$-measure equivalent is an equivalence relation on compactly generated, unimodular, lcsc
groups (see Lemma~\ref{L:composition-of-lp-coc}). Furthermore,
$\rmL^p$-measure equivalence implies $\rmL^{q}$-measure equivalence if $1\le q\le p$.
So among the $\rmL^p$-measure equivalence relations, $\rmL^\infty$-measure equivalence is the strongest and $\rmL^1$-measure equivalence
is the weakest one; all being subrelations of the (unrestricted) measure equivalence.

Let $\Gamma<G$ be a lattice, and assume that $G$ is compactly generated and 
$\Gamma$ is finitely generated; as is the case for semi-simple Lie groups $G$.
Then the $(\Gamma,G)$-coupling $(G,m_G)$ is $\rmL^p$-integrable iff $\Gamma$
is an $\rmL^p$-\emph{integrable lattices} in $G$; 
if there exists a Borel cross-section $s:G/\Gamma\to G$ of the projection, so that
the cocycle $c:G\times G/\Gamma\to\Gamma$, $c(g,x)=s(g.x)^{-1}gs(x)$ is $\rmL^p$-integrable.
In particular $\rmL^\infty$-integrable lattices are precisely the uniform, i.e. cocompact ones.
Integrability conditions for lattices appeared for example in Margulis's proof of
superrigidity (cf. \cite[{V. \S 4}]{Margulis:book}), and in Shalom's \cite{shalom}.

\begin{definition}
A lcsc group $G$ is $p$-\emph{taut} if every $\rmL^p$-integrable $(G,G)$-coupling is taut.
\end{definition}

\medskip

\subsection{Statement of the main results} % (fold)
\label{sub:statement_of_the_main_results}

\begin{mthm}\label{T:SOn1-1-taut}
	The groups $G=\isom(\Hsp^n)$, $n\ge 3$, are $1$-taut.
\end{mthm}

This result has an equivalent formulation in terms of cocycles.

\begin{theorem}[Integrable cocycle strong rigidity]
	\label{T:SOn1-L1-cocycle}
	Let $G=\isom(\Hsp^n)$, $n\ge 3$, $G\acts (X,\mu)$ be a probability measure preserving action,
	and $c:G\times X\to G$ be an integrable ME-cocycle. Then there is
	a measurable map $f:X\to G$, which is unique up to null sets, such
	that for $\mu$-a.e.~$x\in X$ and every $g\in G$ we have
	\[
		c(g,x)=f(g.x)^{-1}\,g\,f(x).
	\]
\end{theorem}

Note that this result generalizes Mostow-Prasad rigidity for lattices in these groups.
This follows from the fact that any $1$-taut group satisfies Mostow rigidity for $L^1$-integrable lattices,
and the fact, due to Shalom, that all lattices in groups $G=\isom(\Hsp^n)$, $n\ge 3$,
are $\rmL^1$-integrable.
% To be precise, Shalom  was more concerned with $\rmL^2$-integrability
% and showed the following:
\begin{theorem}[\cite{shalom}*{Theorem~3.6}]\label{thm:Shalom and lattices}
	All lattices in simple Lie groups
	not locally isomorphic to
	$\isom(\Hsp^2)\simeq\rmPSL_2(\bbR)$,
	$\isom(\Hsp^3)\simeq\rmPSL_2(\bbC)$, are $\rmL^2$-integrable,
	hence also $\rmL^1$-integrable. Further,
	lattices in $\isom(\Hsp^3)$ are $\rmL^1$-integrable.
\end{theorem}
The second assertion is not stated in this form in~\cite{shalom}*{Theorem~3.6} but the proof
therein shows exactly that. In fact, for lattices in $\isom(\Hsp^n)$ Shalom shows $\rmL^{n-1-\epsilon}$-integrability.

Lattices in $G=\isom(\Hsp^2)\cong\rmPGL_2(\bbR)$, such as surface groups, admit a rich
space of deformations -- the Teichm\"uller space.
In particular, these groups do not satisfy Mostow rigidity, and therefore are not taut (they are not even $\infty$-taut).
However, it is well known viewing $G=\isom(\Hsp^2)\cong\rmPGL_2(\bbR)$ as acting on the circle $S^1\cong \partial\Hsp^2\cong \Psp^1$,
any abstract isomorphism $\tau:\Gamma\to\Gamma'$ between cocompact lattices
$\Gamma,\Gamma'<G$ can be realized by a conjugation in $\Homeo(S^1)$, that is,
\[
	\exists_{f\in \Homeo(S^1)}~\forall_{\gamma\in\Gamma}~~\pi\circ\tau(\gamma)
	=f^{-1}\circ \pi(\gamma)\circ f,
\]
where $\pi:G\to \Homeo(S^1)$ is the imbedding as above.
(Such $f$ is the "boundary map" constructed in Mostow's proof of strong rigidity:
the isomorphism $\tau:\Gamma\to \Gamma'$ gives rise to a quasi-isometry of $\Hsp^2$,
and Morse-Mostow lemma is used to extend this quasi-isometry to 
a (quasi-symmetric) homeomorphism $f$ of the boundary $S^1=\partial\Hsp^2$).
Motivated by this observation we generalize the notion of tautness as follows.

\begin{definition}\label{D:tautrel}
Let $G$ be a unimodular lcsc group, $\calG$	a Polish group, $\pi:G\to\calG$ a continuous homomorphism.
A $(G,G)$-coupling is \emph{taut relative}  to $\pi:G\to\calG$ if there exists
a up to null sets unique measurable map $\Phi:\Omega\to \calG$ such that
for $m$-a.e. $\omega\in\Omega$ and all $g_1, g_2\in G$
\[
	\Phi((g_1,g_2)\omega)=\pi(g_1)\Phi(\omega)\pi(g_2)^{-1}.
\]
%Such a $G\times G$-equivariant map is also a \emph{tautening map relative} to $\pi:G\to\calG$.
We say that $G$ is \emph{taut} (resp.~$p$-\emph{taut}) \emph{relative to} $\pi:G\to\calG$
if all (resp.~all $\rmL^p$-integrable) $(G,G)$-couplings are taut relative to $\pi:G\to\calG$.
\end{definition}
Observe that $G$ is taut iff it is taut relative to itself. Note also that if $\Gamma<G$ is a lattice,
then $G$ is taut iff $\Gamma$ is taut relative to the inclusion $\Gamma<G$; and $G$ is taut relative to $\pi:G\to\calG$
iff $\Gamma$ is taut relative to $\pi|_\Gamma:\Gamma\to\calG$ (Proposition~\ref{P:taut-lattice}).
If $\Gamma<G$ is $\rmL^p$-integrable, then these equivalences apply to $p$-tautness.

\begin{mthm}\label{thm:taut relative homeo}
The group $G=\isom(\Hsp^2)\cong\rmPGL_2(\bbR)$ is $1$-taut relative to the natural embedding $G<\Homeo(S^1)$.
Cocompact lattices $\Gamma<G$ are $1$-taut relative to the embedding $\Gamma<G<\Homeo(S^1)$. 	
\end{mthm}
We skip the obvious equivalent cocycle reformulation of this result.
\begin{remarks}\label{R:circle-taut}\hfill
\begin{enumerate}
	\item The $\rmL^1$-assumption cannot be dropped from Theorem~\ref{thm:taut relative homeo}.
	Indeed, the free group $\mathbf{F}_2$ can be realized as a lattice in $\rmPSL_2(\bbR)$,
	but most automorphisms of $\mathbf{F}_2$ cannot be realized by homeomorphisms of the circle.
	\item
	Realizing isomorphisms between surface groups in $\Homeo(S^1)$, one obtains somewhat regular maps:
	they are H\"older continuous and quasi-symmetric. We do not know (and do not expect)
	Theorem~\ref{thm:taut relative homeo} to hold with $\Homeo(S^1)$ being replaced by the corresponding
	subgroups.
\end{enumerate}	
\end{remarks}

\medskip

We now state the $\rmL^1$-ME rigidity result which is deduced from Theorem~\ref{T:SOn1-1-taut},
focusing on the case of countable, finitely generated groups.

\begin{mthm}[$\rmL^1$-Measure equivalence rigidity]
\label{T:ME-rigidity}
  	Let $G=\isom(\Hsp^n)$ with $n\ge 3$, and $\Gamma<G$ be a lattice.
  	Let $\Lambda$ be a finitely generated group, and let $(\Omega,m)$
    be an integrable $(\Gamma,\Lambda)$-coupling. Then
	\begin{enumerate}
		\item
		there exists a short exact sequence
		\[
			1\to F\to \Lambda\to\bar\Lambda\to 1
		\]
		where $F$ is finite and $\bar{\Lambda}$ is a lattice in $G$,
		\item
		and a measurable map $\Phi:\Omega\to G$ so that for $m$-a.e.~$\omega\in\Omega$
		and every $\gamma\in\Gamma$ and every $\lambda\in\Lambda$
		\[
			\Phi((\gamma,\lambda)\omega)=\gamma \Phi(\omega) \bar\lambda^{-1}.
		\]
	\end{enumerate}
	Moreover, if $\Gamma\times\Lambda\acts (\Omega,m)$ is ergodic, then
  	\begin{enumerate}[label=(\arabic*a),start=2]
  	\item
	either the push-forward measure $\Phi_\ast m$ is a positive multiple of the Haar measure $m_G$ or $m_{G^0}$;
  	\item
	or, one may assume that $\Gamma$ and $\bar\Lambda$ share a subgroup of finite index
	and $\Phi_*m$ is a positive multiple of the counting measure on the double
	coset $\Gamma e\bar\Lambda\subset G$.
  	\end{enumerate}
\end{mthm}

% The first part of the statement can be called $\rmL^1$\emph{-measure equivalence rigidity};
% it asserts that for each $n\ge 3$ the  family of all lattices $\Gamma<G=\isom(\Hsp^n)$
% forms, up to finite kernels, a complete $\rmL^1$-ME class in the family of finitely generated groups.
% The additional information given by $\Phi:\Omega\to G$ and the description of $\Phi_*m$
% is useful for applications.
This result is completely analogous to the higher rank case considered in \cite{FurmanME},
except for the $\rmL^1$-assumption.
We do not know whether Theorem~\ref{T:ME-rigidity} remains valid in the broader ME category, that is, without
the $\rmL^1$-condition, but should point out that if the $\rmL^1$ condition can be removed from Theorem~\ref{T:SOn1-1-taut}
then it can also be removed from Theorem~\ref{T:ME-rigidity}.

\medskip

Theorem~\ref{T:ME-rigidity} can also be stated in the broader context of unimodular lcsc groups,
in which case the $\rmL^1$-measure equivalence rigidity states that a compactly generated unimodular lcsc group $H$
that is $\rmL^1$-measure equivalent to $G=\isom(\Hsp^n)$, $n\ge 3$, admits a short exact sequence $1\to K\to H\to \bar{H}\to 1$
where $K$ is compact and $\bar{H}$ is either $G$, or its index two subgroup $G^0$, or is a lattice in $G$.

\medskip

Measure equivalence rigidity results have natural consequences for \emph{(stable, or weak) orbit equivalence}
of essentially free probability measure-preserving group actions
(cf. \cite{FurmanOEw, MonodShalom, Kida:OE, popa-cocycle}).
Two probability measure preserving actions $\Gamma\acts (X,\mu)$, $\Lambda\acts (Y,\nu)$
are \emph{weakly}, or \emph{stably}, \emph{orbit equivalent} if there
exist measurable maps $p:X\to Y$, $q:Y\to X$ with $p_*\mu\ll \nu$, $q_*\nu\ll \mu$ so that a.e.
\[
	p(\Gamma.x)\subset \Lambda.p(x),\quad q(\Lambda.y)\subset \Gamma.q(y),\qquad
	q\circ p(x)\in\Gamma.x,\quad p\circ q(y)\in \Lambda.y.
\]
(see~\cite{FurmanOEw}*{\S 2} for other equivalent definitions).
If $\Gamma_1,\Gamma_2$ are lattices in some lcsc group $G$,
then $\Gamma_1\acts G/\Gamma_2$ and $\Gamma_2\acts G/\Gamma_1$
are stably orbit equivalent via $p(x)=s_1(x)^{-1}$, $q(y)=s_2(y)^{-1}$, where $s_i:G/\Gamma_i\to G$
are measurable cross-sections.
Moreover, given any (essentially) free, ergodic, probability measure preserving (p.m.p.) action $\Gamma_1\acts (X_1,\mu_1)$
and $\Gamma_1$-equivariant quotient map $\pi_1:X_1\to G/\Gamma_2$, there exists a canonically defined free, ergodic
p.m.p. action $\Gamma_2\acts (X_2,\mu_2)$ with equivariant quotient $\pi_2:X_2\to G/\Gamma_1$
so that $\Gamma_i\acts (X_i,\mu_i)$ are stably orbit equivalent in a way compatible to $\pi_i:X_i\to G/\Gamma_{3-i}$~\cite{FurmanOEw}*{Theorem C}.

We shall now introduce integrability conditions on weak orbit equivalence, assuming
$\Gamma$ and $\Lambda$ are finitely generated groups.
Let $|\cdot|_\Gamma$, $|\cdot|_\Lambda$ denote some word metrics on $\Gamma$, $\Lambda$
respectively, and let $\Gamma\acts (X,\mu)$ be an essentially free action.
Define an extended metric $d_\Gamma:X\times X\to [0,\infty]$ on $X$ by
setting $d_\Gamma(x_1,x_2)=|\gamma |_\Gamma$ if $\gamma. x_1=x_2$ and set
$d_\Gamma(x_1,x_2)=\infty$ otherwise.
Let $d_\Lambda$ denote the extended metric on $Y$, defined in a similar fashion.
We say that $\Gamma\acts (X,\mu)$ and $\Lambda\acts (Y,\nu)$ are
$\rmL^s$\emph{-weakly/stably orbit equivalent}, if there exists maps $p:X\to Y$, $q:Y\to X$
as above, and such that for every $\gamma\in \Gamma$, $\lambda\in\Lambda$
%\begin{equation}\label{e:L1-OE}
\[
	\bigl(x\mapsto d_\Lambda(p(\gamma. x),p(x))\bigr)\in \rmL^s(X,\mu),\qquad
	\bigl(x\mapsto d_\Gamma(q(\lambda. y),q(y))\bigr)\in \rmL^s(Y,\nu).
\]
%\end{equation}
Note that the last condition is independent of the choice of word metrics.

The following result\footnote{The formulation of the virtual isomorphism case in terms of induced actions is due to Kida~\cite{Kida:OE}.}
is deduced from Theorem~\ref{T:ME-rigidity} in essentially the same way
Theorems A and C in~\cite{FurmanOEw} are deduced from the corresponding
measure equivalence rigidity theorem in~\cite{FurmanME}.
The only additional observation is that the constructions respect the integrability conditions.

\begin{mthm}[$\rmL^1$-Orbit equivalence rigidity]
	\label{T:L1OE-rigidity}
  	Let $G=\isom(\Hsp^n)$ where $n\ge 3$, and $\Gamma<G$ be a lattice.
	Assume that there is a finitely generated group $\Lambda$ and
  	essentially free, ergodic, p.m.p~actions $\Gamma\acts (X,\mu)$ and $\Lambda\acts (Y,\nu)$,
	which admit a stable $\rmL^1$-orbit equivalence $p:X,\to Y$, $q:Y\to X$ as above.
	Then either one the following two cases occurs:
	\begin{description}
	\item[Virtual isomorphism]
		There exists a short exact sequence
		$1\to F\to \Lambda\to\bar\Lambda\to 1$,
		where $F$ is a finite group and $\bar\Lambda<G$ is a lattice with $\Delta=\Gamma\cap\bar\Lambda$
		having finite index in both $\Gamma$ and $\bar\Lambda$, and an essentially free
		ergodic p.m.p~action $\Delta\acts (Z,\zeta)$
		so that $\Gamma\acts (X,\mu)$ is isomorphic to the induced action $\Gamma\acts \Gamma\times_\Delta (Z,\zeta)$,
		and the quotient action $\bar\Lambda\acts (\bar{Y},\bar\nu)=(Y,\nu)/F$ is isomorphic to the induced action
		$\bar\Lambda\acts \bar\Lambda\times_\Delta (Z,\zeta)$, or
	\item[Standard quotients]
		There exists a short exact sequence $1\to F\to \Lambda\to\bar\Lambda\to 1$,
		where $F$ is a finite group and $\bar\Lambda<G$ is a lattice, and for $G'=G$ or $G'=G^0$
		(only if $\bar\Lambda, \Gamma<G^0$), and equivariant measure space quotient maps
		\[
			\qquad\pi:(X,\mu)\to (G'/\bar\Lambda,m_{G'/\bar\Lambda}),\qquad
			\sigma:(Y,\nu)\to (G'/\Gamma,m_{G'/\Gamma})
		\]
		with $\pi(\gamma.x)=\gamma.\pi(x)$, $\sigma(\lambda.y)=\bar\lambda.\sigma(y)$.
		Moreover, the action $\bar\Lambda\acts (\bar Y,\bar\nu)=(Y,\nu)/F$
		is isomorphic to the canonical action associated to $\Gamma\acts (X,\mu)$ and
		the quotient map $\pi:X\to G'/\bar\Lambda$.
	\end{description}
\end{mthm}

The family of rank one simple real Lie groups $\isom(\Hsp^n)$ is the least rigid one
among simple Lie groups. As higher rank simple Lie groups are rigid with
respect to measure equivalence,
one wonders about the remaining families of simple real Lie groups:
$\isom(\Hsp^n_\bbC)\simeq\rmSU_{n,1}(\bbR)$, $\isom(\Hsp^n_\bbH)\simeq\rmSp_{n,1}(\bbR)$,
and the exceptional group $\isom(\Hsp^2_{\bbO})\simeq F_{4(-20)}$.
The question of measure equivalence
rigidity (or $\rmL^p$-measure equivalence rigidity) for the former family remains open,
but the latter groups are rigid with regard to $\rmL^2$-measure equivalence.
Indeed, recently, using harmonic maps techniques
(after Corlette~\cite{corlette} and Corlette-Zimmer~\cite{corlette+zimmer}),
Fisher and Hitchman~\cite{fisher} proved an $\rmL^2$-cocycle superrigidity result
for isometries of quaternionic hyperbolic
space $\Hsp^n_{\bbH}$ and the Cayley plane $\Hsp^2_{\bbO}$.
This theorem can be used to deduce the following.

\begin{theorem}[Corollary of \cite{fisher}]
	The rank one Lie groups $\isom(\Hsp^n_{\bbH})\simeq\rmSp_{n,1}(\bbR)$
	and $\isom(\Hsp^2_{\bbO})\simeq F_{4(-20)}$ are $2$-taut.
\end{theorem}

Using this result as an input to the general machinery described above one obtains:

\begin{corollary}
	The conclusions of Theorems~\ref{T:ME-rigidity} and~\ref{T:L1OE-rigidity}
	hold for all lattices in $\isom(\Hsp^n_{\bbH})$ and
	$\isom(\Hsp^2_{\bbO})$ provided the $\rmL^1$-conditions are replaced by $\rmL^2$-ones.
\end{corollary}

\subsection{Organization of the paper}
In section Section~\ref{sec:measure_equivalence} we show that all taut groups
are ME-rigid; this is stated in Theorems~\ref{T:reconstruction} and the more detailed
version in Theorem~\ref{T:split+hom}.
In Sections~\ref{s:boundary}~and~\ref{sec:mostow_rigidity_for_maximal_cocycles}
we develop the tools for proving tautness of $G=\isom(\Hsp^n)$ -- the statement 
that generalizes Mostow rigidity, Theorem~\ref{T:SOn1-1-taut} ($n\ge 3$),
and a generalization of Milnor-Wood-Ghys phenomenon, Theorem~\ref{thm:taut relative homeo} ($n=2$).
More precisely, in Section~\ref{s:boundary} we study the effect of an ME cocycle on boundary actions
and boundary maps.
This section contains various technical results on the crossroad of ergodic theory and geometry.
In Section~\ref{sec:mostow_rigidity_for_maximal_cocycles} these results are used to
analyze the effect of the boundary map on cohomology and bounded cohomology,
and specifically on the volume form and the Euler class. 
At a crucial point, when estimating the norm of the Euler class, Corollary~\ref{cor:euler-boundary},
we use a result from our companion paper~\cite{sobolev}, which relies on the integrability of the coupling.
This is the only place where the integrability assumption is used.
The main theorems stated in the introduction are then proved in Section~\ref{sec:proofs_of_the_main_results}.
General facts about measure equivalence which are used throughout 
are collected in Appendix~\ref{sec:appendix_measure_equivalence}.
In order to improve the readability of  Section~\ref{sec:mostow_rigidity_for_maximal_cocycles}
we also added Appendix~\ref{sec:cohomological tools}
which contains a brief discussion of bounded cohomology.

\subsection{Acknowledgements}
U. Bader and A. Furman were supported in part by the BSF grant 2008267.
U. Bader was also supported in part by the ISF grant~704/08 and the ERC grant~306706.
A. Furman was partly supported by the NSF grants DMS 0604611 and 0905977.
R. Sauer acknowledges support from the \emph{Deutsche Forschungsgemeinschaft},
made through grant SA 1661/1-2. 

We thank the referee for his detailed and careful report, especially for 
his recommendations that led to a restructuring of section~3 in a previous version. 

% section (end)

\section{Measure equivalence rigidity for taut groups} % (fold)
\label{sec:measure_equivalence}

This section contains general tools related to the notion of taut couplings and taut groups.
The results of this section apply to general unimodular lcsc groups, including countable groups,
and are not specific to $\isom(\Hsp^n)$ or semi-simple Lie groups. Whenever we refer to
$\rmL^p$-integrability conditions, we assume that the groups are also compactly generated.
We rely on some basic facts about measure equivalence which are collected in Appendix \ref{sec:appendix_measure_equivalence}.
The basic tool is the following:

\begin{theorem}\label{T:reconstruction}
	Let $G$ be a unimodular lcsc group that is taut (resp. $p$-taut).
	Any unimodular lcsc group $H$ that is measure equivalent
	(resp. $\rmL^p$-measure equivalent) to $G$ admits a short exact sequence with
	continuous homomorphisms
	\[
		1\to K\to H\to \bar{H}\to 1,
	\]
	where $K$ is compact and $\bar{H}$ is a closed subgroup in $G$
	such that $G/\bar{H}$ carries a $G$-invariant Borel probability measure.
\end{theorem}

Theorem~\ref{T:split+hom} below contains a more technical statement that applies
to more general situations.

\subsection{The strong ICC property and strongly proximal actions} % (fold)
\label{sub:the_strong_icc_property_and_strongly_proximal_actions}\hfill{}\\
We need to introduce a notion of \emph{strongly ICC} group $G$
and, more generally, the notion of a group $\calG$ being strongly ICC
relative to a subgroup $\calG_0<\calG$.

\begin{definition}\label{def:strongy ICC definition}
	A Polish group $\calG$ is \emph{strongly ICC relative} to $\calG_0<\calG$
	if $\calG\setminus\{e\}$ does not support any Borel probability measure that is
	invariant under the conjugation action of $\calG_0$ on $\calG\setminus\{e\}$.
	A Polish group $\calG$ is \emph{strongly ICC} if it is strongly ICC relative to itself.	
\end{definition}
The key properties of this notion are discussed in the appendix \ref{sub:strong_icc_property}.
Recall that a countable discrete group is said to be ICC (short for Infinite Conjugacy Classes) 
if all its non-trivial conjugacy classes are infinite.
Note that for a discrete group ICC condition is equivalent to the above strong ICC condition.
We will be concerned also with some other examples, that are given in
the following proposition.
\begin{prop}\label{exa:examples of strongly icc groups}\hfill{}
\begin{enumerate}
\item
	Any connected, center-free, semi-simple Lie group $G$ without compact factors is 
	strongly ICC relative to any unbounded Zariski dense subgroup.
	In particular, $G$ itself is strongly ICC.
\item
	For a semi-simple Lie group without compact factors $G$, and a parabolic subgroup $Q<G$, 
	the Polish group $\Homeo(G/Q)$ is strongly ICC relative to $G$.\\
	In particular, $\Homeo(S^1)$ is strongly ICC relative to $\rmPGL_2(\bbR)$, 
	or any lattice in $\rmPGL_2(\bbR)$.
\end{enumerate}
\end{prop}

Before proving this proposition let us recall the notion of \emph{minimal and strongly proximal}
action.
A continuous action $G\acts M$ of a (lcsc) group $G$ on a compact metrizable space $M$
is called \emph{minimal} if there are no closed $G$-invariant nontrivial subsets 
in a compact metrizable space $M$, 
and \emph{strongly proximal} if every $G$-invariant weak$^*$-closed set of probability 
measures on $M$ contains some Dirac measures. 
Clearly, the action $G\acts M$ is both minimal and strongly proximal if every 
$G$-invariant weak$^*$-closed set of probability measures on $M$ contains all the Dirac measures.
Thus, being \emph{minimal and strongly proximal} is easily seen to be equivalent to each 
of the following conditions:
\begin{enumerate}
	\item
	For every Borel probability measure
	$\nu\in\Prob(M)$ and every non-empty open subset $V\subset M$ one has
	\[
		\sup_{g\in G}\ g_*\nu(V)= 1.
	\]
	\item
	For every $\nu\in\Prob(M)$ the convex hull of the $G$-orbit $g_*\nu$
	is dense in $\Prob(M)$ in the weak-* topology.
\end{enumerate}
%Recall that this condition is satisfied by the standard action of $G=\rmPGL_2(\bbR)$
%and its lattices on the circle.
%More generally, any connected semi-simple center free group $G$ without compact factors $G$
%acts on $M=G/Q$ where $Q$ is a parabolic in a minimal and strongly proximal
%fashion~\cite{Margulis:book}*{Theorem~3.7 on p.~205}.
%
We need the following general statement.

\begin{lemma}\label{L:sICCinHomeo}
	Let $M$ be a compact metrizable space and $G<\Homeo(M)$ be a subgroup which acts
	minimally and strongly proximally on $M$. 
	Then $\Homeo(M)$ is strongly ICC relative to $G$.
\end{lemma}
\begin{proof}%[Proof of Lemma~\ref{L:sICCinHomeo}]
Let $\mu$ be a probability measure on $\Homeo(M)$.
The set of $\mu$-\emph{stationary} probability measures on $M$
\[
	\Prob_\mu(M)=\Bigl\{ \nu\in\Prob(M) \mid \nu=\mu*\nu=\int f_*\nu\,d\mu(f) \Bigr\}
\]
is a non-empty convex closed (hence compact) subset of $\Prob(M)$,
with respect to the weak$^*$ topology.
Suppose $\mu$ is invariant under conjugations by $g\in G$.
Since
\[
	g_*(\mu*\nu)=\mu^g*(g_*\nu)=\mu*(g_*\nu)
\]
it follows that $\Prob_\mu(M)$ is a $G$-invariant set.
Minimality and strong proximality of the $G$-action implies that $\Prob_\mu(M)=\Prob(M)$.
In particular, every Dirac measure $\nu_x$ is $\mu$-stationary; hence $\mu\{ f \mid f(x)=x\}=1$.
It follows that $\mu=\delta_e$.
\end{proof}

\begin{proof}[Proof of Proposition~{\ref{exa:examples of strongly icc groups}}]\hfill
\begin{enumerate}
\item
	See~\cite[Proof of Theorem 2.3]{locally-compact}).
\item
	This follows from Lemma~\ref{L:sICCinHomeo}, as by \cite{Margulis:book}*{Theorem~3.7 on p.~205} 
	the action of $G$ on $M=G/Q$ is minimal and strongly proximal.	\qedhere
\end{enumerate}
\end{proof}

%\medskip

Next consider two measure equivalent (countable) groups $\Gamma_1$ and $\Gamma_2$, 
and a continuous action $\Gamma_2\acts M$ on some compact metrizable space $M$.
Let $(\Omega,m)$ be a $(\Gamma_1,\Gamma_2)$-coupling.
Choosing a fundamental domain $X$ for $\Gamma_2\acts \Omega$ 
we obtain a probability measure-preserving action $\Gamma_1\acts (X,\mu)$ and
a measurable cocycle $\alpha:\Gamma_1\times X\to\Gamma_2$, that can be used to define
a $\Gamma_1$-action on $X\times M$ by
\[
	\gamma.(x,p)=\bigl(\gamma.x, \alpha(\gamma,x).p\bigr)\qquad (x\in X,\ p\in M,\ \gamma\in\Gamma_1).
\]
The space $X\times M$ and the above action $\Gamma_1\acts (X\times M)$ combine
ergodic-theoretic base action $\Gamma_1\acts (X,\mu)$ and topological dynamics
in the fibers $\Gamma_2\acts M$. 
In \cite[\S 3 and 4]{furstenberg-bourbaki} Furstenberg defines notions of
minimality and (strong) proximality for such actions.
We shall only need the former notion:
the action $\Gamma_1\acts X\times M$ is \emph{minimal} if there are only trivial measurable $\alpha$-equivariant
families $\{U_x\subset M \mid x\in X\}$ of open subsets $U_x\subset M$.
More specifically, whenever a measurable family of open subsets $U_x\subset M$ satisfies 
for all $\gamma\in \Gamma_1$ and $\mu$-a.e. $x\in X$
\begin{equation}\label{eq: equivariant family of open sets}
	U_{\gamma.x}=\alpha(\gamma,x)\,U_x
\end{equation}
one either has $\mu\{ x\in X \mid U_x=\emptyset\}=1$ or $\mu\{ x \in X\mid U_x=M\}=1$.

We shall need the following lemma (generalizing \cite[Proposition 4.4]{furstenberg-bourbaki}):

\begin{lemma}\label{L:furstenberg}
Let $\Omega$ be an ergodic $(\Gamma_1,\Gamma_2)$-coupling, and $\Gamma_2\acts M$ 
a minimal and strongly proximal action. 
Then the induced action $\Gamma_1\acts X\times M$ is minimal 
in the above sense. 
\end{lemma}

\begin{proof}
Let $i:\Gamma_2\times X\cong \Omega$ be a measure space isomorphism as in (\ref{e:ij-fd}).
%; in particular
%\[
%	(g_1,g_2): i(\gamma,x)\mapsto i(g_2 \gamma \alpha(g_1,x)^{-1},\ g_1.x)\qquad(g_i\in\Gamma_i).
%\]
Given a family $(U_x)$ as in~\eqref{eq: equivariant family of open sets}, consider the measurable family $\{O_\omega\}$ of open subsets of $M$ indexed by $\omega\in\Omega$,
defined by $O_{i(\gamma,x)}=\gamma U_x$. Then for $\omega=i(\gamma,x)$ and $\gamma_i\in \Gamma_i$ we have
\[
	O_{(\gamma_1,\gamma_2)\omega}=\gamma_2\gamma\alpha(\gamma_1,x)^{-1}U_{\gamma_1.x}
	=\gamma_2 \gamma U_x=\gamma_2 O_\omega.
\]
Note that  $\omega\to O_\omega$ is invariant under the action of $\Gamma_1$.
Therefore it descends to a measurable
family of open sets $\{V_y\}$ indexed by $y\in Y\cong \Omega/\Gamma_1$, and satisfying a.e. on $Y$
\[
	V_{\gamma_2.y}=\gamma_2 V_y\qquad (\gamma_2\in\Gamma_2).
\]
The claim about $\{U_x \mid x\in X\}$ is clearly equivalent to the similar claim about $\{V_y \mid y\in Y\}$.
By ergodicity, it suffices to reach a contradiction from the assumption
that $V_y\neq \emptyset, M$ for $\nu$-a.e. $y\in Y$, where $\nu$ is the probability measure
associated to $(\Omega,m)$.

The assumption that $(\Omega,m)$ is $\Gamma_1\times\Gamma_2$-ergodic 
is equivalent to ergodicity of the probability measure preserving action $\Gamma_2\acts (Y,\nu)$.
Since $M$ has a countable base for its topology, while $\mu(\{ y \mid V_y\neq \emptyset\})=1$,
it follows that there exists a non-empty open set $W\subset M$ for which 
\[
	A=\{ y\in Y \mid W\subset V_y\}
\]
has $\nu(A)>0$. Since $M\setminus V_y\neq \emptyset$ for $\nu$-a.e. $y\in Y$,
there exists a measurable map $s:Y\to M$ with $s(y)\notin V_y$ for $\nu$-a.e. $y\in Y$.
Let $\sigma\in \Prob(M)$ denote the distribution of $s(y)$, i.e.,
$\sigma(E)=\nu\{ y\in Y \mid s(y)\in E\}$. Then for any $\gamma_2\in \Gamma_2$
\begin{eqnarray*}
	\sigma(\gamma_2^{-1}W)&=&\nu\{ y\in Y \mid s(y) \in \gamma_2^{-1}W \} \\
	 &\le& \nu(Y\setminus \gamma_2^{-1}A)
	+\nu\{ y\in \gamma_2^{-1}A \mid s(y)\in \gamma_2^{-1} V_{\gamma_2.y}=V_y\} \\
	&=&1-\nu(\gamma_2^{-1}A)=1-\nu(A).
\end{eqnarray*}
So $\sigma(\gamma_2^{-1}W)\le 1-\nu(A)<1$ for all $\gamma_2\in \Gamma_2$, contradicting
the assumption that the action $\Gamma_2\acts M$ is minimal and strongly proximal.
\end{proof}

\medskip

\subsection{Tautness and the passage to self couplings} % (fold)
\label{sub:tautness_and_the_passage_to_self_couplings}\hfill{}\\
Theorem~\ref{T:reconstruction} is a direct consequence of the following, more technical statement
that constructs a representation for arbitrary groups measure equivalent to a given group $G$,
provided some specific $(G,G)$-coupling is taut.

\begin{theorem} \label{T:split+hom}
Let $G$, $H$ be unimodular lcsc groups, $(\Omega,m)$ a $(G,H)$-coupling,
$\calG$ a Polish group, and $\pi:G\to\calG$ a continuous homomorphism.
Assume that $\calG$ is strongly ICC relative to $\pi(G)$ and the $(G,G)$-coupling
$\Omega\times_H\dual{\Omega}$ is taut relative to $\pi:G\to\calG$.

Then there exists a continuous homomorphism $\rho:H\to \calG$ and a measurable map
$\Psi:\Omega\to \calG$ so that a.e.:
\[
	\Psi((g,h)\omega)  = \pi(g)\Psi(\omega)\rho(h)^{-1}\qquad (g\in G,\ h\in H)
\]
and the unique tautening map $\Phi:\Omega\times_H\dual{\Omega}\to \calG$ is given by
\[
		\Phi([\omega_1,\omega_2]) = \Psi(\omega_1)\cdot\Psi(\omega_2)^{-1}.
\]
The pair $(\Psi,\rho)$ is unique up to conjugations $(\Psi^{x}, \rho^{x})$ by $x\in\calG$, where
\[
	\Psi^{x}(\omega)=\Psi(\omega)x^{-1},\qquad \rho^{x}(h)=x\rho(h)x^{-1}.
\]
If, in addition, $\pi:G\to\calG$ has compact kernel and closed image $\bar{G}=\pi(G)$,
then the same
applies to $\rho:H\to\calG$, and there exists a Borel measure $\bar{m}$ on $\calG$,  which is
invariant under
\[
	(g,h):\ x\ \mapsto\ \pi(g)\,x\,\rho(h)
\]
and descends to finite measures on $\pi(G)\bs\calG$ and $\calG/\rho(H)$.
In other words, $(\calG,\bar{m})$ is a $(\pi(G),\rho(H))$-coupling
which is a quotient of $(\Ker(\pi)\times\Ker(\rho))\bs (\Omega,m)$.
\end{theorem}

\begin{proof}
We shall first construct a homomorphism $\rho:H\to \calG$ and the $G\times H$-equivariant map $\Psi:\Omega\to \calG$.	
Consider the space $\Omega^3=\Omega\times\Omega\times\Omega$ and the three maps
$p_{1,2},p_{2,3}$, $p_{1,3}$, where
\[
	p_{i,j}\colon\Omega^3\ \overto{}\ \Omega^2\ \overto{}\ \Omega \times_H \dual{\Omega}
\]
is the projection to the $i$-th and $j$-th factor followed by the natural projection.
Consider the $G^3\times H$-action on $\Omega^3$:
\[
	(g_1,g_2,g_3,h):(\omega_1,\omega_2,\omega_3)\mapsto ((g_1,h)\omega_1,(g_2,h)\omega_2,(g_3,h)\omega_3).
\]
For $i\in\{1,2,3\}$ denote by $G_i$ the corresponding $G$-factor in $G^3$.
For $i,j\in \{1,2,3\}$ with $i\neq j$ the
group $G_i\times G_j<G_1\times G_2\times G_3$ acts on
$\Omega\times_H \dual{\Omega}$ and on $\calG$ by
\[
	(g_i,g_j):\,[\omega_1,\omega_2]\mapsto [g_i\omega_1,g_j\omega_2],
	\qquad (g_i,g_j):x\mapsto \pi(g_i)\,x\,\pi(g_j)^{-1}\qquad(x\in\calG)
\]
respectively. Let $\{i,j,k\}=\{1,2,3\}$.
The map $p_{i,j}:\Omega^3\to \Omega\times_H\dual{\Omega}$
is $G_k\times H$-invariant and $G_i\times G_j$-equivariant.
This is also true of the maps
\[
	F_{i,j}=\Phi\circ p_{i,j}\colon\Omega^3 \xrightarrow{p_{i,j}} \Omega\times_H\dual{\Omega}
	\xrightarrow{\Phi} \calG,
\]
where $\Phi:\Omega\times_H\dual{\Omega}\to \calG$ is the tautening map.
For $\{i,j,k\}=\{1,2,3\}$, the three maps $F_{i,j}$, $F_{j,i}^{-1}$ and $F_{i,k}\cdot F_{k,j}$ are
all $G_k\times H$-invariant, hence factor through the natural map
\[
	\Omega^3\rightarrow\Sigma_k=(G_k\times H)\bs\Omega^3.
\] %=\frac{\Omega_i \times \frac{\Omega_k}{G}\times \Omega_j}{H}. \]
By an obvious variation on the argument in
Appendix~\ref{ssub:composition_of_couplings}
one verifies that $\Sigma_k$ is a
$(G_i, G_j)$-coupling.
The three maps $F_{i,j}$, $F_{j,i}^{-1}$ and $F_{i,k}\cdot F_{k,j}$ are $G_i\times G_j$-equivariant.
Since $\calG$ is strongly ICC relative to $\pi(G)$, there is at most one
$G_i\times G_j$-equivariant measurable map $\Sigma_k\to \calG$
according to Lemma~\ref{L:sICC2uniq}.
Therefore, we get $m^3$-a.e. identities
\begin{equation} \label{e:Fij}
	F_{i,j}=F_{j,i}^{-1}=F_{i,k}\cdot F_{k,j}.
\end{equation}
Denote by $\bar{\Phi}:\Omega^2\to \calG$ the composition
$\Omega^2\overto{} \Omega\times_H \dual{\Omega}\overto{\Phi} \calG$.
By Fubini's theorem, (\ref{e:Fij}) implies that for $m$-a.e. $\omega_2\in \Omega$,
for $m\times m$-a.e. $(\omega_1,\omega_3)$
\[
	\bar{\Phi}(\omega_1,\omega_3)
	=\bar{\Phi}(\omega_1,\omega_2)\cdot \bar\Phi(\omega_2,\omega_3)
	=\bar{\Phi}(\omega_1,\omega_2)\cdot\bar{\Phi}(\omega_3,\omega_2)^{-1}.
\]
Fix such a generic $\omega_2\in \Omega$ and define $\Psi:\Omega\to \calG$ by
$\Psi(\omega)=\bar\Phi(\omega,\omega_2)$.
Then for a.e. $[\omega,\omega']\in \Omega\times_H\Omega$
\begin{equation}\label{e:PhiPsi}
	\Phi([\omega,\omega'])=\bar{\Phi}(\omega,\omega')=\Psi(\omega)\cdot\Psi(\omega')^{-1}.
\end{equation}
We proceed to construct a representation $\rho:H\to \calG$.
Equation~(\ref{e:PhiPsi}) implies that for every $h\in H$ and for a.e $\omega,\omega'\in \Omega$:
\[
	\Psi(h \omega)\Psi(h \omega')^{-1}=\bar{\Phi}(h \omega,h \omega')
	=\bar{\Phi}(\omega,\omega')=\Psi(\omega)\Psi(\omega')^{-1},
\]
and in particular, we get
\[
	\Psi(h\omega)^{-1}\Psi(\omega)=\Psi(h\omega')^{-1}\Psi(\omega').
\]
Observe that the left hand side is independent of $\omega'\in\Omega$, while the right hand
side
is independent of $\omega\in\Omega$.  Hence both are $m$-a.e. constant, and
we denote by $\rho(h)\in\calG$ the constant value.  Being coboundaries the above
expressions are cocycles; being independent of the space
variable they give a homomorphism $\rho:H\to \calG$.
To see this explicitly, for $h,h'\in H$ we compute using $m$-a.e.~$\omega\in\Omega$:
\begin{align*}
    \rho(hh')&=\Psi(hh' \omega)^{-1}\Psi(\omega)\\
    &=\Psi(hh' \omega)^{-1}\Psi(h' \omega) \Psi(h' \omega)^{-1}\Psi(\omega)\\
    &=\rho(h)\rho(h').
\end{align*}
Since the homomorphism $\rho$ is measurable, it is
also continuous~\cite[Theorem B.3 on p.~198]{zimmer-book}.
By definition of $\rho$ we have for $h\in H$ and $m$-a.e $\omega\in\Omega$:
\begin{equation}\label{e:psi-h-eq}
  \Psi(h \omega)=\Psi(\omega)\rho(h)^{-1}.
\end{equation}
Since $\Psi(\omega)=\bar\Phi(\omega,\omega_2)$, it also follows that for $g\in G$
and $m$-a.e.~$\omega\in \Omega$
\begin{equation}\label{e:psi-g-eq}
	\Psi(g \omega)=\pi(g)\Psi(\omega).
\end{equation}
Consider the collection of all pairs $(\Psi,\rho)$
satisfying (\ref{e:psi-h-eq}) and (\ref{e:psi-g-eq}).
Clearly, $\calG$ acts on this set by $x:(\Psi,\rho)\mapsto (\Psi^{x},\rho^{x})=(\Psi\cdot x,x^{-1} \rho x)$;
and we claim that this action is transitive. Let $(\Psi_i,\rho_i)$, $i=1,2$, be two such pairs
in the above set. Then
\[
	\tilde\Phi_i(\omega,\omega')=\Psi_i(\omega)\Psi_i(\omega')^{-1}\qquad (i=1,2)
\]
are $G\times G$-equivariant measurable maps $\Omega\times\Omega\to\calG$, which are invariant
under $H$. Hence they descend to $G\times G$-equivariant maps
$\Phi_i:\Omega\times_H\dual{\Omega}\to\calG$.
The assumption that $\calG$ is strongly ICC relative to $\pi(G)$,
implies a.e. identities $\Phi_1=\Phi_2$, $\tilde\Phi_1=\tilde\Phi_2$.
Hence for a.e. $\omega,\omega'$
\[
	\Psi_1(\omega)^{-1}\Psi_2(\omega)=\Psi_1(\omega')^{-1}\Psi_2(\omega').
\]
Since the left hand side depends only on $\omega$, while the right hand side
only on $\omega'$,
it follows that both sides are a.e. constant $x\in\calG$.
This gives $\Psi_1=\Psi_2^x$. The a.e. identity
\[
	\Psi_1(\omega)\rho_1(h)=\Psi_1(h^{-1}\omega)=\Psi_2(h^{-1}\omega)x
	=\Psi_2(\omega)\rho_2(h)x=\Psi_1(\omega)x^{-1}\rho_2(h)x
\]
implies $\rho_1=\rho_2^x$.
This completes the proof of the first part of the theorem.

\smallskip

Next, we assume that $\Ker(\pi)$ is compact and $\pi(G)$ is closed in $\calG$,
and will show that the kernel $K=\Ker(\rho)$ is compact,
the image $\bar{H}=\rho(H)$ is closed in $\calG$, and that
$\calG/\bar{H}$, $\pi(G)\bs\calG$ carry finite measures.
These properties will be deduced from the assumption on $\pi$ and the
existence of the measurable map $\Psi:\Omega\to \calG$
satisfying~(\ref{e:psi-h-eq}) and~(\ref{e:psi-g-eq}).
We need the next lemma, which says that $\Omega$ has measure space isomorphisms
as in~(\ref{e:ij-fd}) with special properties.
\renewcommand{\qedsymbol}{}
\end{proof}

\begin{lemma}\label{lem:convenient fundamental domains}
Let $\rho:H\to \calG$ and $\Psi:\Omega\to \calG$ be as above.
Then there exist measure space isomorphisms $i:G\times Y\cong \Omega$ and
$j:H\times X\cong \Omega$
as in (\ref{e:ij-fd}) that satisfy in addition
\[
	\Psi(i(g,y))=\pi(g),\qquad \Psi(j(h,x))=\rho(h).
\]
\end{lemma}
\begin{proof}
We start from some measure space isomorphisms
$i_0:G\times Y\cong \Omega$ and $j_0:H\times X\cong \Omega$
as in (\ref{e:ij-fd}) and will replace them by
\[
	i(g,y)=i_0(gg_y,y),\qquad j(h,x)=j_0(hh_x,x)
\]
for some appropriately chosen measurable maps $Y\to G$, $y\mapsto g_y$ and $X\to H$, $x\mapsto h_x$. The conditions (\ref{e:ij-fd}) remain valid after any such alteration.

Let us construct $y\mapsto g_y$ with the required property; the map $x\mapsto h_x$ can be constructed in a similar manner.
By (\ref{e:psi-g-eq}) for $m_G\times\nu$-a.e.~$(g_1,y)\in G\times Y$ the value
\[
	\pi(g)^{-1}\Psi\circ i_0(gg_1,y)
\]
is $m_G$-a.e. independent of $g$; denote it by $f(g_1,y)\in \calG$.
Fix $g_1\in G$ for which
\[
	\Psi\circ i_0(gg_1,y)=\pi(g) f(g_1,y)
\]
holds for $m_G$-a.e.~$g\in G$ and $\nu$-a.e.~$y\in Y$.
There exists a Borel cross section $\calG\to G$ to $\pi:G\to\calG$.
Using such, we get a measurable choice for $g_y$ so that
\[
	\pi(g_y)=f(g_1,y)^{-1}\pi(g_1).
\]
Setting $i(g,y)=i_0(g g_y,y)$,
we get $m_G\times\nu$-a.e. that $\Psi\circ i (g,y)=\pi(g)$.
\end{proof}

\begin{lemma}\label{L:cover}
Given a neighborhood of the identity $V\subset H$ and a compact subset $Q\subset \calG$,
the set $\rho^{-1}(Q)$ can be covered by finitely many translates of $V$:
\[
	\rho^{-1}(Q)\subset h_1 V\cup\cdots\cup h_N V.
\]
\end{lemma}
\begin{proof}
Since $\pi:G\to\calG$ is assumed to be continuous, having closed image and compact kernel,
for any compact $Q\subset\calG$ the preimage $\pi^{-1}(Q)\subset G$ is also compact.  	
Let $W\subset H$ be an open neighborhood of identity so that $W\cdot W^{-1}\subset V$;
we may assume $W$ has compact closure in $H$. Then $\pi^{-1}(Q)\cdot W$ is precompact.
Hence there is an open precompact set $U\subset G$ with $\pi^{-1}(Q)\cdot W\subset U$.
Consider subsets $A=j(W\times X)$, and $B=i(U\times Y)$ of $\Omega$, where
$i$ and $j$ are as in the previous lemma. Then
\[
	m(A)=m_H(W)\cdot \nu(Y) >0,\qquad m(B)=m_G(U)\cdot \mu(X) <\infty.
\]
Let  $\{h_1,\dots,h_n\}\subset \rho^{-1}(Q)$ be such that $h_kW\cap h_l W=\emptyset$
for $k\ne l\in\{1,\dots,n\}$.
Then the sets $h_k A=j(h_k W\times X)$ are also pairwise disjoint and
have $m(h_k A)=m(A)>0$ for $1\le k\le n$. Since
\[
	\Psi(h_k A)=\rho(h_k W)=\rho(h_k)\rho(W)\subset Q\cdot\rho(W)\subset \rho(U),
\]
it follows that $h_kA\subset B$ for every $1\le k\le n$.
Hence $n\le m(B)/m(A)$.
Choosing a \emph{maximal} such set $\{h_1,\dots,h_N\}$, we obtain
the desired cover.
\end{proof}

\begin{proof}[Continuation of the proof of Theorem~\ref{T:split+hom}]
Lemma~\ref{L:cover} implies that the closed subgroup $K=\Ker(\rho)$ is compact.
More generally, it implies
that the continuous homomorphism $\rho:H\to \calG$ is proper, that is,
preimages of compact sets are compact.
Therefore $\bar{H}=\rho(H)$ is closed in $\calG$.

We push forward the measure $m$ to a measure $\bar{m}$ on $\calG$ via
the map $\Psi:\Omega\to \calG$. The measure
$\bar{m}$ is invariant under the action $x\mapsto \pi(g)\,x\,\rho(h)$.
Since $\bar{H}=\rho(H)\cong H/\ker(\rho)$ is closed in $\calG$, the space $\calG/\bar{H}$ is Hausdorff.
As $\Ker(\rho)$ and $\Ker(\pi)$ are compact normal subgroups in $H$ and $G$,
respectively,
the map $\Psi:\Omega\to \calG$ factors through
\[
	(\Omega,m)\overto{} (\Omega',m')=(\Ker(\pi)\times \Ker(\rho))\bs (\Omega,m)\xrightarrow{\Psi'}\calG.
\]
Let $\bar{G}=G/\Ker(\pi)$.
Starting from measure isomorphisms as in
Lemma~\ref{lem:convenient fundamental domains},
we obtain equivariant measure isomorphisms
$(\Omega',m')\cong (\bar{H}\times X,m_{\bar{H}}\times\mu)$ and
$(\Omega',m')\cong (\bar{G}\times Y,m_{\bar{G}}\times\nu)$. In particular,
$(\Omega',m')$ is a $(\bar{G},\bar{H})$-coupling.
The measure $\bar{m}$ on $\calG$ descends to the $\bar{G}$-invariant finite measure on $\calG/\bar{H}$ obtained by pushing forward $\mu$.
Similarly, $\bar{m}$ descends to the $\bar{H}$-invariant finite measure on $\bar{G}\bs\calG$
obtained by pushing forward $\nu$.
This completes the proof of Theorem~\ref{T:split+hom}.
\end{proof}

\begin{proof}[Proof of Theorem~\ref{T:reconstruction}]
	Theorem~\ref{T:reconstruction} immediately follows from Theorem~\ref{T:split+hom}.
	In case of $\rmL^p$-conditions, one observes that if
	$(\Omega,m)$ is an $\rmL^p$-integrable $(G,H)$-coupling, then $\Omega\times_H\dual{\Omega}$
	is an $\rmL^p$-integrable $(G,G)$-coupling (Lemma~\ref{lem:composition of lp couplings});
	so it is taut under the assumption that $G$ is $p$-taut.
\end{proof}

\medskip

\subsection{Lattices in taut groups} % (fold)
\label{sub:lattices_in_taut_groups}

% subsection lattices_in_taut_groups (end)
\begin{proposition}[Taut groups and lattices]\label{P:taut-lattice}\hfill{}\\
Let $G$ be a unimodular lcsc group,
$\calG$ a Polish group, $\pi:G\to\calG$ a continuous homomorphism.
Assume that $\calG$ is strongly ICC relative to $\pi(G)$ and
let $\Gamma<G$ be a lattice (resp. a $p$-integrable lattice).

Then $G$ is taut (resp. $p$-taut) relative to $\pi:G\to\calG$ iff $\Gamma$
is taut (resp. $p$-taut) relative to $\pi|_\Gamma:\Gamma\to\calG$. 
In particular, $G$ is taut iff any/all lattices in $G$ are taut relative
to the inclusion $\Gamma<G$.
\end{proposition}

For the proof of this proposition we shall need the following.

\begin{lemma}[Induction]\label{L:induction-taut}
	Let $G$ be a unimodular lcsc group, $\calG$ a Polish group,
	$\pi:G\to\calG$ a continuous homomorphism, and
	$\Gamma_1,\Gamma_2<G$ lattices.
	Let $(\Omega,m)$ be a $(\Gamma_1,\Gamma_2)$-coupling, and assume
	that the $(G,G)$-coupling $\bar\Omega=G\times_{\Gamma_1}\Omega\times_{\Gamma_2} G$
	is taut relative to $\pi:G\to\calG$.
	Then there exists a $\Gamma_1\times\Gamma_2$-equivariant map $\Omega\to \calG$.
\end{lemma}

\begin{proof}
It is convenient to have a concrete model for $\bar{\Omega}$.
Choose Borel cross-sections $\sigma_i$ from  $X_i=G/\Gamma_i$ to $G$,
and form the cocycles $c_i:G\times X_i\to \Gamma_i$ by
\[
	c_i(g,x)=\sigma_i(g.x)^{-1}g\sigma_i(x),\qquad (i=1,2).
\]
Then, suppressing the obvious measure from the notations, $\bar{\Omega}$ identifies with
$X_1\times X_2\times\Omega$, while the $G\times G$-action is given by
\[
	(g_1,g_2): (x_1,x_2,\omega)\mapsto (g_1.x_1,g_2.x_2, (\gamma_1,\gamma_2)\omega)
	\qquad\text{where}\qquad \gamma_i=c_i(g_i,x_i).
\]
By the assumption there exists a measurable map $\bar\Phi:\bar\Omega\to \calG$
so that
\[
	\bar\Phi((g_1,g_2)(x_1,x_2,\omega))
	=\pi(g_1)\cdot \bar\Phi(x_1,x_2,\omega)\cdot \pi(g_2)^{-1}\qquad (g_1,g_2\in G)
\]
for a.e.~$(x_1,x_2,\omega)\in\bar\Omega$.
Fix a generic pair $(x_1,x_2)\in X_1\times X_2$, denote $h_i=\sigma_i(x_i)$ and consider
$g_i=\gamma_i^{h_i}$ ($=h_i\gamma_i h_i^{-1}$), where $\gamma_i\in\Gamma_i$
for $i\in\{1,2\}$.
Then $g_i.x_i=x_i$, $c_i(g_i,x_i)=\gamma_i$ and the map $\Phi':\Omega\to\calG$ defined by
$\Phi'(\omega)=\bar\Phi(x_1,x_2,\omega)$
satisfies $m$-a.e.
\begin{align*}
	\Phi'((\gamma_1,\gamma_2)\omega)=\bar\Phi((g_1,g_2)(x_1,x_2,\omega))
	&=\pi(g_1)\cdot \Phi'(\omega)\cdot \pi(g_2)^{-1}\\
	&=\pi(\gamma_1^{h_1})\cdot \Phi'(\omega)\cdot \pi(\gamma_2^{h_2})^{-1}.
\end{align*}
Thus $\Phi(\omega)=\pi(h_1)^{-1} \Phi'(\omega) \pi(h_2)$ is a $\Gamma_1\times\Gamma_2$-equivariant measurable
map $\Omega\to \calG$, as required.
\end{proof}

\begin{proof}[Proof of Proposition~\ref{P:taut-lattice}]
Assuming that $G$ is taut relative to $\pi:G\to\calG$ and $\Gamma<G$ is a lattice, we shall
show that $\Gamma$ is taut relative to $\pi|_\Gamma:\Gamma\to\calG$.

Let $(\Omega,m)$ be a $(\Gamma,\Gamma)$-coupling.
Then the $(G,G)$-coupling $\bar\Omega=G\times_{\Gamma}\Omega\times_{\Gamma} G$ is taut relative to $\calG$,
and by Lemma~\ref{L:induction-taut}, $\Omega$ admits a $\Gamma\times\Gamma$-tautening map $\Phi:\Omega\to\calG$.
Since $\calG$ is strongly ICC relative to $\pi(G)<\calG$, the map
$\Phi:\Omega\to \calG$ is unique as a $\Gamma\times\Gamma$-equivariant map
(Lemma~\ref{L:sICC2uniq}.(\ref{i:2lattice})).
This shows that $\Gamma$ is taut relative to $\calG$.

Observe, that if $G$ is assumed to be only $p$-taut, while $\Gamma<G$ to be $\rmL^p$-integrable,
then the preceding argument for the existence of $\Gamma\times\Gamma$-tautening map
for a $\rmL^p$-integrable $(\Gamma,\Gamma)$-coupling $\Omega$ still applies.
Indeed, the composed coupling $\bar\Omega=G\times_{\Gamma}\Omega\times_{\Gamma} G$
is then $\rmL^p$-integrable and therefore admits a $G\times G$-tautening map $\bar{\Phi}:\bar\Omega\to\calG$,
leading to a $\Gamma\times\Gamma$-tautening map $\Phi:\Omega\to \calG$.
Finally, $\calG$ is strongly ICC relative to $\pi(\Gamma)$ by Lemma~\ref{lem:ICC-lattices},
and the uniqueness of the $\Gamma$ tautening map follows from Lemma~\ref{L:sICC2uniq}(\ref{i:unique}).

Next assume that $\Gamma<G$ is a lattice and $\Gamma$ is taut (resp.~$p$-taut) relative to $\pi|_\Gamma:\Gamma\to\calG$.
Let $(\Omega,m)$ be a $(G,G)$-coupling (resp.~a $\rmL^p$-integrable one).
Then $(\Omega,m)$ is also a $(\Gamma,\Gamma)$-coupling (resp.~a $\rmL^p$-integrable one).
Since $\Gamma$ is assumed to be taut (resp. $p$-taut)
there is a $\Gamma\times\Gamma$-equivariant map $\Phi:\Omega\to\calG$.
As $\calG$ is strongly ICC relative to $\pi(G)$ it follows from (\ref{i:extn}) in Lemma~\ref{L:sICC2uniq}
that $\Phi:\Omega\to \calG$ is automatically $G\times G$-equivariant.
The uniqueness of tautening maps follows from the strong ICC assumption by Lemma~\ref{L:sICC2uniq}(\ref{i:unique}).
\end{proof}

\begin{remark}
	The explicit assumption that $\calG$ is strongly ICC relative to $\pi(G)$
	is superfluous. If no integrability assumptions are imposed, the strong ICC
	follows from the tautness assumption by Lemma~\ref{L:uniq2sICC}.
	However, if one assumes merely $p$-tautness, the above lemma yields strong ICC
	property for a restricted class of measures; and the argument that this is sufficient
	becomes unjustifiably technical in this case.
\end{remark}
% subsection (end)

%%%%%%%%%%%%%%%%%

\section{Some boundary theory}
\label{s:boundary}

Throughout this section we refer to the following 

\begin{setup}\label{basic setup}
    We fix $n \ge 2$ and introduce the following setting: 
	\begin{itemize}
		\item[--] $G=\isom(\Hsp^n)$.
		\item[--] $\Gamma<G^0$ be a torsion-free uniform lattice
		in the connected component of the identity. 
		\item[--] $M=\Gamma\bs\Hsp^n$ is a compact hyperbolic manifold with 
		$\Gamma\simeq \pi_1(M)$. 
		\item[--] The boundary $\partial\Hsp^n$ is denoted by $B$; we 
		identify it with the $(n-1)$-sphere 
		and equip it with the Lebesgue measure $\nu$\footnote{any probability measure in the Lebesgue measure class would do it}.
		\item[--] $(\Omega,m)$ is an ergodic $(\Gamma,\Gamma)$-coupling. 
		\item[--]  For better readability we denote
		the left copy of $\Gamma$ by $\Gamma_l$ and the right copy by $\Gamma_r$.
		We fix a $\Gamma_l$-fundamental domain $X\subset \Omega$,
		and denote by $\mu$ the restriction of $m$ to $X$.
		\item[--] Identifying $X$ with $\Gamma_l\bs\Omega$, we get an ergodic action of $\Gamma_r$ on $(X,\mu)$.
		We denote by $\alpha:\Gamma_r\times X\to \Gamma_l<G$ the associated cocycle.
	\end{itemize}
  	Boundary theory, in the sense of Furstenberg~\cite{Furstenberg}
	(see~\cite{burger-mozes}*{Corollary~3.2}, or \cite{MonodShalom-cocycle}*{Proposition~3.3}
	for a detailed argument applying to our situation), yields the existence of
	an essentially unique measurable map, called \emph{boundary map} or \emph{Furstenberg map},
	\begin{equation}\label{eq:boundary map}
			\phi:X\times B\to B
			\qquad\text{satisfying}\qquad
			\phi(\gamma x,\gamma b)=\alpha(\gamma, x)\phi(x,b)
	\end{equation}
	for	every $\gamma\in\Gamma$ and a.e.~$(x,b)\in X\times B$.
\end{setup}

\begin{prop} \label{prop:reduction}
Let  $\calG<\Homeo(B)$ be a closed subgroup containing $\Gamma$ and denote the inclusion
$\Gamma<\calG$ by $\pi$.
If for a.e $x\in X$ the function $\phi(x,\cdot):B\to B$ coincides a.e with an element of $\calG$ then
$\Omega$ is taut with respect to the inclusion $\pi:\Gamma\to \calG$.
\end{prop}

\begin{proof}
Consider
the set $\rmF(B,B)$ of measurable functions $B\to B$, where two functions
are identified if they agree on $\nu$-conull set.
We endow $\rmF(B,B)$ with the topology of
convergence in measure. The Borel $\sigma$-algebra of this topology turns $\rmF(B,B)$ into
a standard Borel space~\cite{fisher-nonergodic}*{Section~2A}.
%Two different Polish topologies on $W$ with the same Borel algebra give rise to the
%same standard Borel space $\rmF(S,W)$~\cite{fisher-nonergodic}*{Remark~2.5}.
By \cite{kechris}*{Corollary~15.2 on p.~89} the
measurable injective map $j\colon\calG\rightarrow\rmF(B,B)$
is a Borel isomorphism of $\calG$ onto its measurable image.

The map $\phi$ gives rise to
a measurable map $f:X\to\rmF(B,B)$ defined for almost every $x\in X$ by
$f(x)= \phi(x,\cdot)$~\cite{fisher-nonergodic}*{Corollary~2.9},
which can be regarded as a measurable map $f:X\to\calG$.
Equation~(\ref{eq:boundary map}) gives
\begin{equation}\label{eq:conjugation}
	\pi\circ\alpha(\gamma,x)=f(\gamma.x)\pi(\gamma) f(x)^{-1},
\end{equation}
thus
by Lemma~\ref{L:coc-taut},
there is a tautening map $\Omega\to \calG$.

Note that by Proposition~\ref{exa:examples of strongly icc groups}(2),
$\Homeo(B)$ is strongly ICC relative to $G$.
It follows by Lemma~\ref{lem:ICC-lattices} that it is also strongly ICC relative to $\Gamma$.
Therefore $\calG$ is strongly ICC relative to $\Gamma$,
and by Lemma~\ref{L:sICC2uniq}(\ref{i:unique}) the tautening is unique.
\end{proof}

\subsection{Preserving maximal simplices of the boundary} % (fold)
\label{sub:preserving maximal simplices}\hfill{}\\
Recall that a geodesic simplex in $\bar\Hsp^n=\Hsp^n\cup\partial \Hsp^{n}$ is called \emph{regular}
if any permutation of its vertices can be realized by an element in $\isom(\Hsp^n)$.
The set of ordered $(n+1)$-tuples on the boundary $B$
that form the vertex set of an ideal regular simplex is denoted by $\Sigma^{\rm reg}$.
The set $\Sigma^{\rm reg}$ is a disjoint union $\Sigma^{\rm reg}=\Sigma^{\rm reg}_+\cup \Sigma^{\rm reg}_-$
of two subsets that correspond to the positively and negatively oriented
ideal regular $n$-simplices, respectively.
%
% It is well known in hyperbolic geometry that a geodesic simplex in $\bar\Hsp^n$ has
% maximal (non-oriented) volume among all such geodesic simplices if and only if it is
% an ideal regular simplex.
We denote by $v_{\mathrm{max}}$ the maximum possible volume of an ideal simplex.

\begin{lemma}[Key facts from Thurston's proof of Mostow rigidity]\label{lem:key facts from Mostow}\hfill
\begin{enumerate}
	\item \label{K:i}The diagonal $G$-action on $\Sigma^{\rm reg}$ is simply transitive. The diagonal
	$G^0$-action on $\Sigma^{\rm reg}_-$ and $\Sigma^{\rm reg}_+$ are simply transitive,
	respectively.
	\item\label{K:ii} An ideal simplex has non-oriented volume $v_{\mathrm{max}}$ if and only if it
	is regular.
	\item Let $n\ge 3$.
	Let $\sigma,\sigma'$ be two regular ideal simplices having a common face of codimension
	one. Let $\rho$ be the reflection along the hyperspace spanned by this face. Then
	$\sigma=\rho(\sigma')$.
\end{enumerate}
\end{lemma}

\begin{proof}
	(1) See the proof of~\cite{ratcliffe}*{Theorem~11.6.4 on p.~568}.\\ % TODO (Roman) better reference?
	\noindent (2) The statement
	is trivial for $n=2$, as all non-degenerate ideal triangles in $\bar\Hsp^2$
	are regular, and $G$ acts simply transitively on them.
	The case $n=3$ is due to Milnor,
	and Haagerup and Munkholm~\cite{haagerup} proved the general case $n\ge 3$. \\
	\noindent (3) This is a
	key feature distinguishing the $n\ge 3$ case from the $n=2$ case
	where Mostow rigidity fails.
	See~\cite{ratcliffe}*{Lemma~13 on p.~567}.
\end{proof}

We shall need the following lemma, which is due to
Thurston~\cite{thurston}*{p.~133/134} in dimension $n=3$. Recall that $B=\partial \Hsp^n$
is equipped with the Lebesgue measure class.
We consider the natural measure $m_{\Sigma^{\rm reg}_+}$
on $\Sigma^{\rm reg}_+$ corresponding to the Haar measure on $G^0$
under the simply transitive action of $G^0$ on $\Sigma^{\rm reg}_+$.

\begin{lemma}\label{lem:key lemma in mostow rigidity}
Let $n\ge 3$ and $\phi:B\to B$ be a Borel map such that
$\phi^{n+1}=\phi\times\cdots\times\phi$ maps a.e.~point
in $\Sigma^{\rm reg}_+$ into $\Sigma^{\rm reg}_+$.
Then there exists a unique $g_0\in G^0=\isom_+(\Hsp^n)$
with $\phi(b)=g_0 b$ for a.e.~$b\in B$.	
\end{lemma}

% If the maps $\phi\colon B\to B$ and $\phi'\colon B\to B$ coincide up to null sets, then
% $\phi^{n+1}\colon\Sigma^{\rm reg}_+\to B^{n+1}$ and
% $\phi'^{n+1}\colon\Sigma^{\rm reg}_+\to B^{n+1}$ coincide
% up to $m_{\Sigma^{\rm reg}_+}$-null sets.

\begin{proof}
	Fix a regular ideal simplex $\sigma=(b_0,\dots,b_n)\in\Sigma^{\rm reg}_+$,
	and identify $G^0$ with $\Sigma^{\rm reg}_+$ via $g\mapsto g\sigma$.
	Then there is a Borel map $f:G^0\to G^0$ such that for a.e. $g\in G^0$
	\begin{equation}\label{e:def-f}
		(\phi(gb_0),\dots,\phi(gb_n))=(f(g)b_0,\dots,f(g)b_n).	
	\end{equation}
	Interchanging $b_0$, $b_1$ identifies $\Sigma^{\rm reg}_+$ with $\Sigma^{\rm reg}_-$,
	and allows to extend $f$ to a measurable map $G\to G$ satisfying (\ref{e:def-f}) for a.e. $g\in G$.
	Let $\rho_0,\dots,\rho_n\in G$ denote the reflections in the codimension one faces
	of $\sigma$.
	Then Lemma~\ref{lem:key facts from Mostow} (3) implies that
	\[
		f(g\rho)=f(g)\rho\qquad\text{for a.e.}\qquad g\in G
	\]
	for $\rho$ in $\{\rho_0,\dots,\rho_n\}$. It follows that the same applies
	to each $\rho$ in the countable group $R<G$ generated by
	$\rho_0,\dots,\rho_n$.
	We claim that there exists $g_0\in G$ so that $f(g)=g_0 g$
	for a.e. $g\in G$, which implies that $\phi(b)=g_0 b$ also holds a.e. on $B$.

	The case $n=3$ is due to Thurston (\cite{thurston}*{p.~133/134}).
	So hereafter we focus on $n>3$, and will show that in this case the
	group $R$ is dense in $G$
	(for $n=2,3$ it forms a lattice in $G$).
	Consequently the $R$-action on $G$ is ergodic with respect to the Haar measure.
	Since $g\mapsto f(g)g^{-1}$ is a measurable $R$-invariant map on $G$, it follows that it
	is a.e. a constant $g_0\in G^0$, i.e., $f(g)=g_0 g$ a.e. proving the lemma.

	It remains to show that for $n>3$, $R$ is dense in $G$.
	Not being able to find a convenient reference for this fact,
	we include the proof here.
	
	For $i\in\{0,\dots,n\}$ denote by $P_i<G$ the stabilizer of $b_i\in\partial\Hsp^n$,
	and let $U_i<P_i$ denote its unipotent radical.
	We shall show that $U_i$ is contained in the closure $\overline{R\cap P_i}<P_i$
	(in fact, $\overline{R\cap P_i}=P_i$ but we shall not need this).
	Since unipotent radicals of any two opposite parabolics, say $U_0$ and $U_1$, generate
	the whole connected simple Lie group $G^0$, this would show $G^0<\bar{R}<G$.
	Since $R$ is not contained in $G^0$, it follows that $\bar{R}=G$ as claimed.

	Let $f_i:\partial \Hsp^n\to \Euc^{n-1}\cup\{\infty\}$ denote the stereographic projection
	taking $b_i$ to the point at infinity.
	Then $f_i P_i f_i^{-1}$ is the group of
	similarities $\isom(\Euc^{n-1})\rtimes\bbR^\times_+$
	of the Euclidean space $\Euc^{n-1}$.
	We claim that the subgroup of translations $\bbR^{n-1}\cong U_i<P_i$
	is contained in the closure of $R_i=R\cap P_i$.
	To simplify notations we assume $i=0$.
	The set of all $n$-tuples $(z_1,\dots,z_n)$ in $\Euc^{n-1}$
	for which $(b_0,f_0^{-1}(z_1),\dots,\dots,f_0^{-1}(z_n))$ is a regular ideal simplex
	in $\bar{\Hsp}^{n}$ is precisely the set of all regular Euclidean
	simplices in $\Euc^{n-1}$~\cite{ratcliffe}*{Lemma~3 on p.~519}.
	So conjugation by $f_0$ maps the group $R_0=R\cap P_0$
	to the subgroup of $\isom(\Euc^{n-1})$
	generated by the reflections in the faces of the Euclidean simplex
	$\Delta=(z_1,\dots,z_n)$, where $z_i=f_0(b_i)$.
	For $1\le j<k\le n$ denote by $r_{jk}$ the
	composition of the reflections in the $j$th and $k$th faces of $\Delta$; it is a rotation
	leaving fixed the co-dimension two affine hyperplane $L_{jk}$ containing $\{z_i\mid i\neq j,k\}$.
	The angle of this rotation is $2\theta_n$, where $\theta_n$ is the dihedral angle
	of the simplex $\Delta$. One can easily check that $\cos(\theta_n)=-1/(n-1)$,
	using the fact that the unit normals $v_i$ to the faces of $\Delta$ satisfy
	$v_1+\dots+v_n=0$ and $\langle v_i, v_j\rangle=\cos(\theta_n)$ for all $1\le i<j\le n$.
	Thus $w=\exp(\theta_n\sqrt{-1})$ satisfies $w+1/w=-2/(n-1)$. Equivalently, $w$ is a root of
	\[
		p_n(z)=(n-1)z^2+2z+(n-1).
	\]
	This condition on $w$ implies that $\theta_n$ is not a rational multiple of $\pi$.
	Indeed, otherwise, $w$ is a root of unit, and therefore is a root of some
	cyclotomic polynomial
	\[
		c_m(z)=\prod_{k\in\{1..m-1\mid \gcd(k,m)=1\}}(z-e^{\frac{2\pi k i}{m}})
	\]
	whose degree is Euler's totient function $\deg(c_m)=\phi(m)$.
	The cyclotomic polynomials are irreducible over $\bbQ$.
	So $p_n(z)$ and $c_m(z)$ share a root only if they are proportional, which in particular
	implies $\phi(m)=2$.
	The latter happens only for $m=3$, $m=4$ and $m=6$; corresponding to
	$c_3(z)=z^2+z+1$, $c_4(z)=z^2+1$, and $c_6(z)=z^2-z+1$.
	The only proportionality between these polynomials is $p_3(z)=2 c_2(z)$;
	and it is ruled out by the assumption $n>3$.

	Thus the image of $R_0$ in $\isom(\Euc^{n-1})$ is not discrete.
	Let
	\[
		\pi:\overline{R\cap P_0}\to \isom(\Euc^{n-1})\to {\rm O}(\bbR^{n-1})
	\]
	denote the homomorphism defined by taking the linear part.
	Then $\pi(r_{jk})$ is an irrational rotation in ${\rm O}(\bbR^{n-1})$ leaving invariant the linear subspace
	parallel to $L_{jk}$. The closure of the subgroup generated by
	this rotation is a subgroup $C_{jk}<{\rm O}(\bbR^{n-1})$, isomorphic to ${\rm SO}(2)$.
	The group $K<{\rm O}(\bbR^{n-1})$ generated by all such $C_{jk}$
	acts irreducibly on $\bbR^{n-1}$, because there is no subspace orthogonal to all $L_{jk}$.
	Since $\overline{R\cap P_0}$ is not compact (otherwise there would
	be a point in $\Euc^{n-1}$ fixed by all reflections in faces of $\Delta$),
	the epimorphism	$\pi:\overline{R\cap P_0}\to K$ has a non-trivial kernel $V<\bbR^{n-1}$,
	which is invariant under $K$.
	As the latter group acts irreducibly, $V=\bbR^{n-1}$ or, equivalently,
	$U_0<\overline{R\cap P_0}$.
	This completes the proof of the lemma.
\end{proof}

% subsection Preserving maximal simplices of the boundary (end)

\medskip

\subsection{Boundary simplices in general position}

%We are interested in continuity properties of $\vol$ on the space $B^{n+1}$ of $(n+1)$-tuples
%and, more generally, continuity of the $k$-dimensional volume $\vol^k$ of the space of ideal
%$k$-simplices defined by a $(k+1)$-tuple $(z_0,\dots,z_k)\in B^{k+1}$ for $k\in \{2,\dots,n\}$.
%These functions are not continuous in general. This is clear for $k=2$ as $\vol^2(z_0,z_1,z_2)$
%takes only three values $\{-\pi,0,+\pi\}$.
%For $k>2$ one can find sequences of regular ideal
%$k$-simplices $\sigma_i=(z_0^{(i)},\dots,z_k^{(i)})$ where three of the vertices
%$z_0^{(i)}$, $z_1^{(i)}$, $z_2^{(i)}$ converge to the same limit point $z\in B$ while the rest are
%constant $z_j^{(i)}=z_j$, $j=3,\dots,k$.
%In this case $\sigma_i\to \sigma$ with $\vol^k(\sigma)=0$.

%However, $\vol^k$ is continuous at all $(k+1)$-tuples $(z_0,\dots,z_k)$ that are "in general position"
%\cite{ratcliffe}*{Theorem~11.4.2 on p.~541}, where the notion of "general position" can be defined

\begin{defn} \label{def:gen-pos}
For $0\leq k\leq n$, a $(k+1)$-tuple of points in $B$, $(z_0,\dots,z_k)\in B^{k+1}$ is said to be in \emph{general position} if the following equivalent conditions hold:
\begin{enumerate}
	\item The $k$-volume of the ideal $k$-simplex with vertices $\{z_0,\dots,z_k\}$ is positive,
	\item The points $\{z_0,\dots,z_k\}$ lie on the boundary of a unique isometrically embedded copy of
	$\Hsp^{k}$ in $\Hsp^n$,
	\item The points $\{z_0,\dots,z_k\}$ do not lie on the boundary of some isometrically embedded
	copy of $\Hsp^{k-1}$ in $\Hsp^n$.
\end{enumerate}
The set of $(k+1)$-tuples in a general position in $B^{k+1}$ is denoted $B^{(k+1)}$.
\end{defn}

We shall use the term $(k-1)$-\emph{sphere} to denote the boundary of an isometrically
embedded copy of $\Hsp^k$ in $\Hsp^n$; with $0$-spheres meaning pairs of distinct points.

\begin{lemma}\label{L:gen-pos}
	Consider the boundary map $\phi(x,\cdot)=\phi_x$ from~\eqref{eq:boundary map}. 
    For $\mu\times\nu^{n+1}$-a.e. point $(x,b_0,\dots,b_n)$, 
    the $(n+1)$-tuple $(\phi_x(b_0),\dots,\phi_x(b_n))$ is in general position. 
\end{lemma}

\begin{remark}
	In fact we prove a more general statement. The only important properties of
	our setting are the fact that $\alpha$ is Zariski dense 
	(in particular, is not measurably cohomologous to a cocycle
	taking values in a stabilizer of $\Hsp^k\subset \Hsp^n$ with $k<n$)
	and that the diagonal measure class-preserving action
	\[
		\Gamma\acts (X\times B\times B,\mu\times\nu\times\nu)
	\]
	is ergodic, which, in our setting, follows from the Howe-Moore theorem. 
\end{remark}
\begin{proof}[Proof of Lemma~\ref{L:gen-pos}]
	Denote by $\eta_x \in\Prob(B)$ the push-forward of $\nu$ under the map $\phi_x:B\to B$.
	%$\eta_x(E)=\phi_x^{-1}(E)$.
	For $k\in \{2,\dots,n\}$ and $x\in X$ let
	\[
		E_k =\left\{ x\in X \mid \eta_x^{k+1}(B^{k+1}\setminus B^{(k+1)})>0\right\}.
	\]
	This is a measurable subset of $X$, which is $\Gamma$-invariant since 
	$\eta_{\gamma.x}=\alpha(\gamma,x)_*\eta_x$ while $B^{(k+1)}$ is a Borel, 
	in fact open, $G$-invariant subset of $B^{k+1}$.
	Ergodicity of $\Gamma\acts (X,\mu)$ implies that $\mu(E_k)=0$ or $\mu(E_k)=1$.
	The sets $E_k$ are also nested: $E_{k-1}\subset E_{k}$, because any subset of a
	$(k+1)$-tuple in general position, is itself in general position.
	
	We claim that $\mu(E_n)=0$. By contradiction, let $k$ be the smallest integer in $\{2,\dots,n\}$
	with $\mu(E_k)>0$. Then, in fact, $\mu(E_k)=1$ by the ergodicity argument above.
	Since $\mu(E_{k-1})=0$ for $\mu$-a.e. $x\in X$ and $\nu^k$-a.e. $(b_1,\dots,b_k)\in B^k$
	the points $(\phi_x(b_1),\dots,\phi_x(b_k))$ are in general position, and therefore
	define a unique $(k-2)$-sphere
	\[
		S_x(b_1,\dots,b_k)\subset B.
	\]
	On the other hand, $\mu(E_k)=1$ means that for $\mu$-a.e. $x\in X$
	\[
		\nu^{k+1}\{ (b_0,\dots,b_k) \mid \phi_x(b_0)\in S_x(b_1,\dots,b_k)\}>0.
	\]
	By Fubini's theorem, there is a measurable family of measurable subsets
	$A_x\subset B^k$ with $\nu^k(A_x)>0$, so that for $(b_1,\dots,b_k)\in A_x$
	\[
		\eta_x(S_x(b_1,\dots,b_k))>0.
	\]
	Denote by $\calS$ the space of all $(k-2)$-spheres $S\subset B$, and let
	\[
		\calS_x=\{S\in \calS \mid \eta_x(S)>0\}.
	\]
	Using $\eta_{g.x}=\alpha(g,x)_*\eta_x$ we deduce that
	\[
		\calS_{\gamma.x}=\alpha(\gamma,x)\,\calS_{x}.
	\]
	Hence the set $\{ x\in X\times B \mid \calS_{x}\ne \emptyset\}$
	is measurable and $\Gamma$-invariant.
	We just argued above that this set has positive measure,
	hence by ergodicity of $\Gamma\acts (X,\mu)$, it has full measure.
	
	Our main claim is that $\calS_x$ consists of a single $(k-2)$-sphere:
	\begin{equation}\label{e:Sx}
		\calS_x=\{ S_x \}.
	\end{equation}
	This claim leads to a desired contradiction as follows:
	equivariance of $\calS_x$ becomes the $\mu$-a.e. identity
	$\alpha(\gamma,x) S_x=S_{\gamma.x}$.
	Fix a $(k-2)$-sphere $S_0$ and a measurable map $f:X\to G$ with
	$S_x=f(x)S_0$. Then the $f$-conjugate of $\alpha$
	\[
		\alpha^f(\gamma,x)=f(\gamma.x)^{-1}\alpha(\gamma,x)f(x)
	\]
	takes values in the stabilizer of $S_0$ in $G$, which is a proper algebraic
	subgroup $\isom(\Hsp^k)<\isom(\Hsp^n)=G$. But this is impossible for an ME-cocycle.
	
	\medskip
	
	It remains to show (\ref{e:Sx}).
	Consider any two measurable families $S_x, S'_x \in \calS_x$ indexed by $x\in X$,
	and let
	\[	
		F=\{ x\in X \mid S_x\ne S'_x\quad \textrm{and}\quad \eta_x(S_x\cap S'_x)>0\}.
	\]
	We claim that $\mu(F)=0$.
	Indeed, for $x\in F$ the intersection $R_x=S_x\cap S'_x$ is a sphere of
	dimension $\le (k-3)$, and therefore $k$-tuples of points in $R_x$
	are not in general position. This implies
	\[
		\eta^{k}_x\bigl( B^k\setminus B^{(k)}\bigr)\ge\eta^{k}_x \bigl(R_x^k\bigr)=\bigl(\eta_x(R_x)\bigr)^k>0
	\]
	meaning that $x\in E_{k-1}$. As $\mu(E_{k-1})=0$, it follows that $\mu(F)=0$.

	We now claim that a.e. $\calS_x$ has at most countably many elements (spheres).
	It suffices to show that for every $\epsilon>0$ for
	$\mu$-a.e. $x$ the set
	\[
		\calS_{x}^{>\epsilon}=\{ S\in \calS_x \mid \eta_x(S)>\epsilon\}.
	\]
	is finite. We will show that its cardinality is bounded by $1/\epsilon$.
	Otherwise it is possible to find a positive measure set $Y\subset X$
	and $m>1/\epsilon$ maps $S_{i,y}\in \calS^{>\epsilon}_y$, $y\in Y$,
	$1\le i \le m$, so that for $i\ne j$ one has $S_{i,y}\ne S_{j,y}$.
	But this is impossible, because for a.e. $y\in Y$ one has
	$\eta_y(S_{i,y}\cap S_{j,y})=0$ for every pair $i\ne j$, and therefore
	\[
		1\ge \eta_y(\bigcup_{i=1}^m S_{i,y})= \sum_{i=1}^m \eta_y(S_{i,y}) >m\epsilon>1.
	\]
	Therefore, a.e. $\calS_x$ is countable, and one can enumerate these collections
	by a fixed sequence $\calS_x=\{S_{i,x}\}_{i=1}^\infty$ of $(k-2)$-spheres
	with $S_{i,x}$ varying measurably in $x\in X$.
	For $x\in X$ let
	\[	P_{i,x}=\{ (b,b')\in B \times B \mid \phi_x(b), \phi_x(b')\in S_{i,x} \}. 
	\]
	We have $\nu^2(P_{i,x})=\eta_x(S_{i,x})^2>0$.
	The union
	\[
		P_x=\bigcup_{i=1}^\infty P_{i,x}=\bigl\{ (b,b') \mid \exists_{S\in \calS_x}~ \phi_x(b),\phi_x(b')\in S\bigr\}
	\]
	satisfies $	\alpha(\gamma,x) P_x=P_{\gamma.x}$.
	Therefore $\{ (x,b,b') \mid (b,b')\in P_x\}$ is a measurable,
	$\Gamma$-invariant set of positive $\mu\times\nu\times\nu$-measure.
	Hence from the ergodicity of the measure-class preserving action
	$\Gamma\acts (X\times B\times B,\mu\times\nu\times\nu)$, this set has full measure.
	In particular, for $\mu$-a.e. $x\in X$, one has
	\[
		\sum_{i} \eta_x(S_{i,x})^2=\sum_{i} \nu^2(P_{i,x})=\nu^2(P_x)=1
	\]
	while
	\[
		\sum_{i} \eta_x(S_{i,x})=\eta_x(\bigcup_{S\in \calS_x} S)\le 1.
	\]
	This is possible, only if exactly one $S_{i,x}$ has full $\eta_x$-measure,
	i.e., if $\calS_x$ consists of a single sphere $\calS_x=\{S_x\}$, as claimed.
	This completes the proof of the lemma.
\end{proof}
%----------------

\subsection{A Lebesgue differentiation lemma}

\begin{lemma}\label{lem:Lebesgue_differentiation}
  Fix points $o\in \Hsp^n$ and $b_0\in \partial\Hsp^n$.
  Denote by $d=d_o$ the	visual metric on $\partial\Hsp^n$
  associated with $o$.
  Let $\{z^{(k)}\}_{k=1}^\infty$ be a sequence in $\Hsp^n$ converging
	radially to $b_0$.
    Let $\phi\colon B\to B$ be a measurable map.
  For every $\epsilon>0$ and for a.e.~$g\in G$ we have
  \[
  \lim_{k\to\infty} \nu_{z^{(k)}}\left\{b\in B \mid
    d(\phi(gb),\phi(gb_0))>\epsilon\right\} = 0.
  \]
\end{lemma}

\begin{proof}
  For the domain of $\phi$, it is convenient to represent
  $\partial\Hsp^n$ as the boundary $\hat\bbR^n=\{(x_1,\dots,
  x_n,0)\mid x_i\in\bbR\}\cup\{\infty\}$ of the upper half space model
  \[
	\Hsp^n=\{(x_1,\dots,x_{n+1})\mid x_{n+1}>0\}\subset\bbR^{n+1}.
  \]
  We may assume that $o=(0,\dots,0,1)$ and
  $b_0=0\in\bbR^{n}\subset\hat\bbR^n$. The points $z^{(k)}$ lie on
  the line $l$ between $o$ and $b_0$. The subgroup of $G$ consisting
  of reflections along hyperplanes containing $l$ and perpendicular to
  $\{x_{n+1}=0\}$ leaves the measures $\nu_{z^{(k)}}$ invariant, i.e.~each
  $\nu_{z^{(k)}}$ is ${\rm O}(n)$-invariant. Since the probability measure
  $\nu_{z^{(k)}}$ is in the Lebesgue measure class, the
  Radon-Nikodym theorem, combined with the ${\rm O}(n)$-invariance, yields
  the existence of a measurable functions
  $h_k:[0,\infty)\to[0,\infty)$ such that for any bounded measurable function $l$
  \[
  \int l\,d\nu_{z^{(k)}}=\int_0^\infty
  \left(\frac{1}{\vol(B(0,r))}\int_{B(0,r)}l(y)\,dy\right) h_k(r)\,dr
  \]
  holds\footnote{$\vol(B(0,r))$ is here the Lebesgue measure of the Euclidean ball of radius $r$ around $0\in\bbR^n$.} and
  \[
  \int_0^\infty h_k(r)\,dr=1.
  \]
  Since the $\nu_{z^{(k)}}$ weakly converge to the Dirac measure at
  $0\in\bbR^n$, we have for every $r_0>0$
  \begin{equation}\label{eq:convergence}
    \lim_{k\to\infty}\int_{r_0}^\infty h_k(r)\,dr=0.
  \end{equation}
  For the target of $\phi$, we represent $B=\partial \Hsp^n$ as the
  boundary $S^{n-1}\subset\bbR^n$ of the Poincare disk model. The
  visual metric is then just the standard metric of the unit
  sphere. Considering coordinates in the target, it
  suffices
  to prove that every measurable function $f:\hat\bbR^n\to [-1,1]$
  satisfies
  \begin{equation*}%\label{eq:reduced statement}
    \lim_{k\to\infty} \int_{\hat\bbR^n}\abs{f(gx)-f(g0)}d\nu_{z^{(k)}}(x) = 0.
  \end{equation*}
  for a.e.~$g\in G$.
  By the Lebesgue differentiation theorem the set $L_f$ of points
  $x\in \bbR^n$ with the property
  \begin{equation}\label{eq:lebesgue diff}
    \lim_{r\to 0}\frac{1}{\vol(B(0,r))} \int_{B(x,r)} |f(y)-f(x)|\,dy =0
  \end{equation}
  is conull in $\bbR^n$. The subset of elements $g\in G$ such that $g0\in L_f$ and
  $g0\ne \infty$ is conull
  with respect to the Haar measure. From now on we fix such an element $g\in G$.
  By compactness there is $L>0$ such that the diffeomorphism of $\hat\bbR^n$
  given by $g$ has Lipschitz constant at most $L$ and its Jacobian satisfies
  $\abs{{\rm Jac}(g)}>1/L$ everywhere on $\bbR^n\subset\hat\bbR^n$.
  Let $\epsilon>0$.
  According to~\eqref{eq:lebesgue diff} choose $r_0>0$ such that for all $r<r_0$
  \begin{equation}\label{eq: L estimate}
  	\frac{L}{\vol(B(0,r))}\int_{B(g0,Lr)}\abs{f(y)-f(g0)}dy<\frac{\epsilon}{2}.
  \end{equation}
  According to~\eqref{eq:convergence} choose $k_0\in\bbN$ such that
  \[
  		\int_{r_0}^\infty h_k(r)\,dr<\frac{\epsilon}{4}
  \]
  for every $k> k_0$. So we obtain that
  \begin{align*}
  	 \int_{\hat\bbR^n}\abs{f(gx)-f(g0)}d\nu_{z^{(k)}}&<
		\int_0^{r_0}\frac{1}{\vol(B(0,r))}\int_{B(0,r)}\abs{f(gx)-f(g0)}dx\; h_k(r)dr+\frac{\epsilon}{2}\\
		&\le \int_0^{r_0}\frac{L}{\vol(B(0,r))}\int_{gB(0,r)}\abs{f(y)-f(g0)}dy\; h_k(r)dr+\frac{\epsilon}{2}
  \end{align*}
	for $k>k_0$. Because of $gB(0,r)\subset B(g0,Lr)$ and~\eqref{eq: L estimate}
	we obtain that for $k>k_0$
	\[
		\int_{\hat\bbR^n}\abs{f(gx)-f(g0)}d\nu_{z^{(k)}}<\epsilon.
		\qedhere
	\]
\end{proof}

%%%%%%%%%%%

%%%%%%%%%%%%%%%%%

\section{Cohomological tools} % (fold)
\label{sec:mostow_rigidity_for_maximal_cocycles}
	
The aim of this section is to prove that the boundary map 
\mbox{$\phi_x=\phi(x,\cdot): B\to B$}, which is associated 
to a $(\Gamma,\Gamma)$-coupling with $\Gamma<\isom(\Hsp^n)$ and 
introduced in Setup~\ref{basic setup}, satisfies the 
assumption of Lemma~\ref{lem:key lemma in mostow rigidity}. This will be achieved 
in Corollary~\ref{cor:euler-boundary}. The conclusion of Lemma~\ref{lem:key lemma in mostow rigidity} is a crucial ingredient in the proof of 
Theorem~\ref{T:SOn1-1-taut}. To prove Corollary~\ref{cor:euler-boundary} we 
have to develop and rely on a fair amount of cohomological machinery. 
For the reader's convenience a brief introduction to the subject of bounded 
cohomology is given in Appendix~\ref{sec:cohomological tools}.

\subsection{The cohomological induction map} % (fold)
\label{sub:the_cohomological_induction_map_for_integrable_me_couplings}

The \emph{cohomological induction map} associated to an arbitrary 
ME-coupling
was introduced by Monod and Shalom~\cite{MonodShalom}.

\begin{proposition}[Monod-Shalom]\label{prop:induction from Monod-Shalom}
	Let $(\Omega,m)$ be a $(\Gamma,\Lambda)$-coupling. Let $Y\subset \Omega$ be a measurable
	fundamental domain for the $\Gamma$-action.
	Let
	$\chi\colon\Omega\to\Gamma$ be the measurable $\Gamma$-equivariant map uniquely
	defined by
	$\chi(\omega)^{-1}\omega\in Y$ for $\omega\in\Omega$.
	The maps
	\begin{gather*}
		\rmCb^\bullet(\chi)\colon
		\rmCb^\bullet(\Gamma,\rmL^\infty(\Omega))
		\to\rmCb^\bullet(\Lambda,\rmL^\infty(\Omega))\\
		\rmCb^k(\chi)(f)(\lambda_0,\dots,\lambda_k)(y)
		=	f\bigl(\chi(\lambda_0^{-1}y)),\dots,\chi(\lambda_k^{-1}y)\bigr)(y)
	\end{gather*}
	defines
	a $\Gamma\times\Lambda$-equivariant chain morphism with regard to the following
	actions: The $\Gamma\times\Lambda$-action on
	$\rmCb^\bullet(\Gamma,\rmL^\infty(\Omega))\cong \rmL^\infty(\Gamma^{\bullet+1}\times\Omega)$
	is induced
	by $\Gamma$ acting diagonally on $\Gamma^{\bullet+1}\times\Omega$ and by $\Lambda$
	acting only on $\Omega$. The $\Gamma\times\Lambda$-action on
	$\rmCb^\bullet(\Lambda,\rmL^\infty(\Omega))\cong\rmL^\infty(\Lambda^{\bullet+1}\times\Omega)$ is induced by $\Lambda$ acting diagonally on $\Lambda^{\bullet+1}\times\Omega$
	and by $\Gamma$ acting only on $\Omega$.
	
	The chain map $\rmCb^\bullet(\chi)$ induces, after taking $\Gamma\times\Lambda$-invariants and identifying $\rmL^\infty(\Gamma\bs\Omega)$ with
	$\rmL^\infty(\Omega)^\Gamma$ and similarly for $\Lambda$,
	an isometric isomorphism
	\[
		\rmHb^\bullet(\chi) \colon \rmHb^\bullet(\Gamma,\rmL^\infty(\Lambda\bs\Omega))
		\xrightarrow{\cong}\rmHb^\bullet(\Lambda,\rmL^\infty(\Gamma\bs\Omega)).
	\]
	in cohomology. This map
	does not depend on the choice of $Y$, or equivalently $\chi$, and will be
	denoted by $\rmHb^\bullet(\Omega)$. We call $\rmHb^\bullet(\Omega)$
	the \emph{cohomological induction map associated to $\Omega$}.
\end{proposition}

\begin{proof}
	Apart from the fact that the isomorphism is isometric, this is exactly
	Proposition~4.6 in~\cite{MonodShalom} (with $S=\Omega$ and $E=\bbR$). The proof therein
	relies on~\cite{monod-book}*{Theorem~7.5.3 in~\S 7}, which also yields the
	isometry statement.
\end{proof}

\begin{proposition}\label{prop:induction reciprocity}
	Retain the setting of the previous proposition. Let $\alpha\colon\Lambda\times Y\to\Gamma$ be the corresponding ME-cocycle. Let $B_\Gamma$ and $B_\Lambda$ be
	standard Borel
	spaces endowed with probability Borel measures and measure-class preserving
	Borel actions of $\Gamma$ and $\Lambda$, respectively. 
	Let $\phi\colon B_\Lambda\times \Gamma\bs\Omega\to B_\Gamma$
	be a measurable $\alpha$-equivariant map
	(upon identifying $Y$ with $\Gamma\bs\Omega$).
	Then the chain morphism (see Subsection~\ref{sub:banach_modules} for notation) 
	\begin{gather*}
		\rmCb^\bullet(\phi)\colon\calB^\infty(B_\Gamma^{\bullet+1},\bbR)\to\rmLweak^\infty(B_\Lambda^{\bullet+1}, \rmL^\infty(\Omega))\\
		\rmCb^k(\phi)(f)(\dots,b_i,\dots)(\omega)=f\bigl(\dots,\chi(\omega)\phi(b_i,[\omega]),\dots\bigr).
	\end{gather*}
	is $\Gamma\times\Lambda$-equivariant with regard to the following actions: The
	action on $\calB^\infty(B_\Gamma^{\bullet+1},\bbR)$ is induced from $\Gamma$ acting
	diagonally $B^{\bullet+1}$ and $\Lambda$ acting trivially. The
	action on $\rmLweak^\infty(B_\Lambda^{\bullet+1}, \rmL^\infty(\Omega))\cong \rmL^\infty(B_\Lambda^{\bullet+1}\times\Omega)$ is induced from $\Lambda$ acting
	diagonally on $B_\Lambda^{\bullet+1}\times\Omega$ and from $\Gamma$ acting only on
	$\Omega$.
		% 
		% Further, every $\Gamma\times\Lambda$-chain morphism from
		% $\calB^\infty(B_\Gamma^{\bullet+1},\bbR)$ to
		% $\rmLweak^\infty(B_\Lambda^{\bullet+1}, \rmL^\infty(\Omega))$ that induces
		% the same homomorphism on $\rmHb^0$ as $\rmCb^\bullet(\phi)$
		% is
		% equivariantly chain homotopic to $\rmCb^\bullet(\phi)$.
\end{proposition}

\begin{proof}
	Firstly, we show equivariance of $\rmCb^\bullet(\phi)$. By definition we
	have
	\begin{equation*}
		\rmCb^\bullet(\phi)((\gamma,\lambda)f)(\dots,b_i,\dots)(\omega) =
		f\bigl(\dots,\gamma^{-1}\chi(\omega)\phi(b_i,[\omega]),\dots\bigr).
	\end{equation*}
	By definition, $\Gamma$-equivariance of $\chi$, and $\alpha$-equivariance of $\phi$
	we have
	\begin{equation*}
		\rmCb^\bullet(\phi)(f)\bigl(\dots,\lambda^{-1}b_i,\dots)(\gamma^{-1}\lambda^{-1}\omega\bigr)=
		f\bigl(\dots,\gamma^{-1}\chi(\lambda^{-1}\omega)\alpha(\lambda^{-1},[\omega])\phi(b_i,[\omega]),\dots\bigr).
	\end{equation*}
	It remains to check that
	\[
		\chi(\lambda^{-1}\omega)\alpha(\lambda^{-1}, [\omega])=\chi(\omega).
	\]
	Since both sides are $\Gamma$-equivariant, we may assume that $\omega\in Y$, i.e.,
	$\chi(\omega)=1$. In this
	case it follows from the defining properties of $\chi$ and $\alpha$.
	% 
	% Next we prove the uniqueness up to equivariant chain homotopy. By
	% Proposition~\ref{prop:B-complex strong resolution} the complex
	% $\calB^\infty(B_\Gamma^{\bullet+1},\bbR)$ is a strong resolution of the trivial
	% $\Gamma\times\Lambda$-module $\bbR$. The $\Gamma\times\Lambda$-action on
	% $B_\Lambda^{\bullet+1}\times\Omega$ is amenable if the $\Lambda$-action on
	% $B_\Lambda^{\bullet+1}\times\Gamma\bs\Omega$
	% is amenable~\cite{adams+elliot}*{Corollary~C}. The latter action is amenable
	% since the $\Lambda$-action on $B_\Lambda$ is amenable and
	% because of~\cite{zimmer-book}*{Proposition~4.3.4 on p.~79}. By
	% Theorem~\ref{thm:boundary resolutions by Burger-Monod}
	% $\rmLweak^\infty(B_\Lambda^{\bullet+1}, \rmL^\infty(\Omega))
	% \cong \rmL^\infty(B_\Lambda^{\bullet+1}\times\Omega)$ is a relatively injective,
	% strong resolution of the trivial $\Gamma\times\Lambda$-module, and
	% Theorem~\ref{thm:main homological theorem} yields
	% uniqueness up
	% to equivariant homotopy.
\end{proof}

\begin{remark} The map $\rmCb^\bullet(\phi)$ cannot be defined on
	$\rmL^\infty(B_\Gamma^{\bullet+1},\bbR)$ since we do not assume that $\phi$ preserves the
	measure class. The idea to work with the complex $\calB^\infty(B_\Gamma^{\bullet+1},\bbR)$
	to circumvent this problem in the context of boundary maps
	is due to Burger and Iozzi~\cite{burger+iozzi}.
\end{remark}

\medskip

\subsection{The Euler number in terms of boundary maps} % (fold)
\label{sub:the_cohomological_induction_map_and_boundary_maps}
In this subsection we retain the notation in Setup~\ref{basic setup}. 
In Burger-Monod's functorial theory of bounded 
cohomology~\cites{burger+monod,monod-book} 
the measurable map
\begin{equation}\label{eq: volume cocycle}
	\dvol_b\colon B^{n+1}\rightarrow\bbR
\end{equation}
that assigns to $(b_0,\dots,b_n)$
the oriented volume of the geodesic, ideal simplex with vertices $b_0,\dots,b_n$ is a
$\Gamma$-invariant (even $G^0$-invariant) cocycle and defines
an element $\dvol_b\in\rmHb^n(\Gamma,\bbR)$
(Theorem~\ref{thm:boundary resolutions by Burger-Monod}).
The forgetful map (comparison map) from bounded
cohomology to ordinary cohomology is denoted by
\[\comp^\bullet\colon\rmHb^\bullet(\Gamma,\bbR)\to\rmH^\bullet(\Gamma,\bbR).\]
We consider the induction homomorphism
\[
	\rmHb^\bullet(\Omega)\colon \rmHb^\bullet(\Gamma_l,\rmL^\infty(\Gamma_r\bs\Omega))\to
	                         \rmHb^\bullet(\Gamma_r,\rmL^\infty(\Gamma_l\bs\Omega))
\]
in bounded cohomology associated to $\Omega$ (see Subsection~\ref{sub:the_cohomological_induction_map_for_integrable_me_couplings}).
Let
\begin{align*}
\rmHb^\bullet(j^\bullet)&\colon \rmHb^\bullet(\Gamma_l,\bbR)\to\rmHb^\bullet(\Gamma_l,\rmL^\infty(\Gamma_r\bs\Omega))\\
\rmHb^\bullet(\rmI^\bullet)&\colon \rmHb^\bullet(\Gamma_r,\rmL^\infty(\Gamma_l\bs\Omega))\to \rmHb^\bullet(\Gamma_r,\bbR)
\end{align*}
be the homomorphisms induced by inclusion of constant functions in the coefficients and by
integration in the coefficients, respectively.
Inspired by the classical Euler number of a surface representation we define:
\begin{definition}[Higher-dimensional Euler number]\label{eq:Euler number}
    Denote by $[\Gamma]\in \rmH_n(\Gamma,\bbR)\cong \rmH_n(\Gamma\bs\Hsp^n,\bbR)$ the homological fundamental class
    of the manifold $\Gamma\bs\Hsp^n$.
	The \emph{Euler number} $\euler(\Omega)$ of $\Omega$
	is the evaluation of the cohomology class
	$\comp^n\circ\rmHb^n(\rmI^\bullet)\circ\rmHb^n(\Omega)\circ\rmHb^n(j^\bullet)(\dvol_b)$ against the fundamental class $[\Gamma]$
\begin{equation} \label{eq:euler}
		\euler(\Omega)=\bigl\langle 	\comp^n\circ\,\rmHb^n(\rmI^\bullet)\circ\rmHb^n(\Omega)\circ\rmHb^n(j^\bullet)(\dvol_b), [\Gamma]\bigr\rangle.
\end{equation}
\end{definition}
In a recent paper~\cite{burger+bucher+iozzi} Bucher-Burger-Iozzi use a related notion to study maximal (in 
a similar sense as in Corollary~\ref{cor:maximality}) 
representations of $\rmSO_{n,1}$. 
In the Burger-Monod approach to bounded cohomology one can realize bounded cocycles
in the bounded cohomology of $\Gamma$
as cocycles on the boundary $B$. However, it is not immediately clear how the
evaluation of a bounded $n$-cocycle realized on $B$ at the fundamental class
of $\Gamma\bs\Hsp^n$ can be explicitly computed since the fundamental class is not
defined in terms of the boundary.
Lemma~\ref{lem:image of fund class under Patterson-Sullivan map}
below achieves just that. Let us now describe two important ingredients that enter 
the proof of Lemma~\ref{lem:image of fund class under Patterson-Sullivan map}. 

The first ingredient is the \emph{cohomological Poisson transform} which is 
expressed by the visual measures on $B=\partial\Hsp^n$. 

\begin{definition}\label{def:visual measure}
	For $z\in\Hsp^n$ let $\nu_z$ be the \emph{visual measure} at $z$ on the boundary
	$B=\partial\Hsp^n$ at infinity, that is, $\nu_z$ is the push-forward of the
	Lebesgue measure on the unit tangent sphere $\mathrm{T}^1_z\Hsp^n$
	under the homeomorphism
	$\mathrm{T}^1_z\Hsp^n\to\partial\Hsp^n$ given by the exponential map.
    For a $(k+1)$-tuple $\sigma=(z_0,\ldots, z_k)$ of points in $\Hsp^n$ we denote
    the product of the $\nu_{z_i}$ on $B^{k+1}$ by $\nu_\sigma$.
\end{definition}

	The measure $\nu_z$ is the unique Borel probability measure on $B$
	that is invariant with
	respect to the stabilizer of $z$. All visual measures are in the same measure class.
	Moreover, we have 
	\[\nu_{g z}=g_\ast\nu_z=\nu_z(g^{-1}\_)\text{ for every $g\in G$.}\]
	The cohomological Poisson transform (see Definition~\ref{def:poisson transform} for its general formulation) 
	is the $\Gamma$-morphism of chain complexes 
	$\rmPT^\bullet: \rmL^\infty(B^{\bullet+1},\bbR)\to\rmCb^\bullet(\Gamma, \bbR)$ 
	with 
	\begin{align}\label{eq: poisson concrete}
		\rmPT^n(f)(\gamma_0,\ldots,\gamma_n)&=\int_{B^{n+1}}f(\gamma_0 b_0,\ldots, \gamma_n b_n)d\nu_{x_0}\ldots d\nu_{x_0}\notag\\
		&=\int_{B^{n+1}}f(b_0,\ldots, b_n)d\nu_{(\gamma_0x_0,\ldots,\gamma_nx_0)}
	\end{align}
	where $x_0\in\Hsp^n$ is a base point. The map $\rmPT^\bullet$ is independent of 
	the choice of $x_0$ (see the remark after Definition~\ref{def:poisson transform}). 
	
	The second ingredient is 
	Thurston's description of singular homology
	by \emph{measure cycles}~\cite{thurston}:
	Let $M$ be a topological space.
	We equip the space $\calS_k(M)=\map(\Delta^k,M)$ of continuous maps from the
	standard $k$-simplex to $M$ with the compact-open topology.
	The group $\rmCm_k(M)$ is the vector space of all signed, compactly supported
	Borel measures on $\calS_k(M)$ with finite total variation.
	The usual face maps
	$\partial_i: \calS_k(M)\to\calS_{k-1}(M)$
	are measurable, and the maps
	$\rmCm_k(M)\rightarrow\rmCm_{k-1}(M)$ that send $\mu$ to $\sum_{i=0}^k
	(-1)^i(\partial_i)_{\ast}\mu$
	turn $\rmCm_\bullet(M)$ into a chain complex. 
	The map
	\begin{equation*}%\label{eq: Dirac measure map into measure chains}
	  \rmD_\bullet\colon \rmC_\bullet(M)\rightarrow\rmCm_\bullet(M),~\sigma\mapsto\delta_\sigma
	\end{equation*}
	that maps a singular simplex $\sigma$ to the point measure
	concentrated at $\sigma$ is a chain map that induces an (isometric)
	homology isomorphism provided $M$ is homeomorphic to a
	CW-complex~\cites{loeh,zastrow}. 
   
    Next we recall Thurston's 
	\emph{smearing construction}, which describes an explicit representative of the 
	fundamental class of a closed hyperbolic 
	manifold $M=\Gamma\bs\Hsp^n$.
	
	 For any positively oriented geodesic $n$-simplex $\sigma$ in $\Hsp^n$, 
	 let $\smear(\sigma)$ denote
	 the push-forward of the normalized Haar measure on $G^0=\isom(\Hsp^n)^0$ 
	 under the measurable map
	  \[
	   \Gamma\bs G^0\rightarrow \map(\Delta^n, \Gamma\bs\Hsp^n),~g\mapsto
	  \pr(g\sigma).
	  \]
	  Let $\rho\in G$ be the orientation reversing isometry that maps $(z_0,z_1,\dots,z_n)$
	  to $(z_1,z_0,\dots,z_n)$.
	  By~\cite{ratcliffe}*{Theorem~11.5.4 on p.~551} the image of 
	the fundamental class in $\rmH_n(\Gamma\bs\Hsp^n,\bbR)$ under the map 
	$\rmH_n(\rmD_\bullet)$ 
	is represented by the signed 
	  measure\footnote{The reader should note that in \emph{loc.~cit.} the Haar measure is
	  normalized by $\vol(\Gamma\bs\Hsp^n)$ whereas we normalize it by
	  $1$.}
	  \begin{equation}\label{eq: thurston representative}
	 \frac{\vol(\Gamma\bs\Hsp^n)}{2\vol(\sigma_0)}\bigl(\smear(\sigma_0)-
	  \smear(\rho\circ\sigma_0)\bigr)
	  \end{equation}
	  for any positively oriented geodesic $n$-simplex $\sigma_0$ in $\Hsp^n$.
	
\begin{lemma}\label{lem:image of fund class under Patterson-Sullivan map}
  Let $\Gamma\subset G^0$ be a torsion-free and uniform lattice.
  Let $\sigma_0=(z_0,\dots,z_n)$ be a positively oriented geodesic
  simplex in $\Hsp^n$. Let $[\Gamma]\in \rmH_n(\Gamma,\bbR)\cong\rmH_n(\Gamma\bs\Hsp^n,\bbR)$ be the fundamental class 
of $\Gamma\bs\Hsp^n$. Let
  $f\in \rmL^\infty(B^{n+1},\bbR)^\Gamma$ be an alternating cocycle. Then
  \begin{equation*}
    \bigl\langle \comp^n\circ\,\rmHb^n(\rmPT^\bullet)([f]), [\Gamma]\bigr\rangle =
    \frac{\vol(\Gamma\bs\Hsp^n)}{\vol(\sigma_0)}\int_{B^{n+1}}\int_{\Gamma\bs G^0}
	f(gb_0,\dots,gb_n)\,d\nu_{\sigma_0} dg.
\end{equation*}
\end{lemma}

\begin{proof}
   Fix a basepoint $x_0\in\Hsp^n$.  Consider the $\Gamma$-equivariant
  chain homomorphism $j_k\colon\rmC_k(\Gamma)\rightarrow\rmC_k(\Hsp^n)$
  that maps $(\gamma_0,\dots,\gamma_k)$ to the geodesic simplex with
  vertices $(\gamma_0 x_0,\dots,\gamma_k x_0)$.
  Let
  $\calB^\infty\bigl(\calS_\bullet(\Hsp^n),\bbR\bigr)\subset \rmC^\bullet(\Hsp^n,\bbR)$
  be the subcomplex
  of bounded measurable singular cochains on $\Hsp^n$. 
  From~\eqref{eq: poisson concrete} we see that 
the
Poisson transform $\rmPT^\bullet$ factorizes as
  \[
 \rmL^\infty(B^{\bullet+1},\bbR)\xrightarrow{\rmP^\bullet}\calB^\infty
	\bigl(\calS_\bullet(\Hsp^n),\bbR\bigr)\xrightarrow{\rmR^\bullet}\rmCb^\bullet(\Gamma,\bbR)
  \]
  where 
  \begin{align*}
  	    \rmP^k(l)(\sigma)&=\int_{B^{k+1}}l(b_0,\dots,b_k)d\nu_{\sigma}\text{ for $\sigma\in\calS_k(\Hsp^n)$, and}\\
        \rmR^k(f)&=f\circ j_k.
  \end{align*}
  For every $k\ge 0$ there is a Borel section
  $s_k:\calS_k(\Gamma\bs\Hsp^n)\rightarrow\calS_k(\Hsp^n)$
  of the projection~\cite{loeh}*{Theorem~4.1}.
  The following pairing is independent of the choice of $s_k$ and
  descends to cohomology:
  \begin{gather*}
    \langle\_,\_\rangle_m\colon\calB^\infty\bigl(\calS_\bullet(\Hsp^n),\bbR\bigr)^\Gamma\otimes \rmCm_\bullet(\Gamma\bs\Hsp^n)\rightarrow\bbR\\
    \langle l,\mu\rangle_m=
    \int_{\calS_\bullet(\Gamma\bs\Hsp^n)}l\bigl(s_\bullet(\sigma)\bigr)d\mu(\sigma)
  \end{gather*}
  One sees directly from the definitions that for every $x\in\rmH_n(\Gamma,\bbR)$
  \begin{align}\label{eq:fundamental cycles and adjunction}
  \bigl\langle \comp^n\circ\, \rmH^n\bigl(\rmPT^\bullet)([f]),x\bigr\rangle &=
  \bigl\langle \comp^n\circ\,\rmH^n\bigl(\rmR^\bullet)\circ\rmH^n(\rmP^\bullet)(f), x\bigr\rangle\\
		&=\bigl\langle \rmH^n(\rmP^\bullet)([f]), \rmH_n(\rmD_\bullet\circ j_\bullet)(x)\bigr\rangle_m. \notag
  \end{align}
	Now we plug in $x=[\Gamma]$. Since the homology class 
	$\rmH_n(\rmD_\bullet\circ j_\bullet)([\Gamma])$ is represented by 
	the measure cycle~\eqref{eq: thurston representative} and 
	$f$ is alternating, the assertion is implied. 
\end{proof}

Theorem~\ref{thm:bounded volume cocycle to volume coycle} is known to experts; we 
prove it for the lack of a
good reference. Although it can be seen as a special case of
Theorem~\ref{thm:evaluation by fundamental class} we separate the proofs.
The proofs of
Theorems~\ref{thm:bounded volume cocycle to volume coycle}
and~\ref{thm:evaluation by fundamental class} are given at the end of the
subsection.

\begin{theorem}\label{thm:bounded volume cocycle to volume coycle}
Let $\Gamma\subset G^0$ be a torsion-free and uniform lattice. Then
\[
	\bigl\langle\comp^n(\dvol_b), [\Gamma]\bigr\rangle=\vol(\Gamma\bs\Hsp^n).
\]
Equivalently, this means that $\comp^n(\dvol_b)=\dvol$.
\end{theorem}

\begin{theorem}\label{thm:evaluation by fundamental class}
  Let $(\Omega,m)$ be an ergodic $(\Gamma,\Gamma)$-coupling of a
  torsion-free and uniform lattice $\Gamma\subset G^0$.
  Let 
	\[
		\phi\colon X\times B\to B
	\] 
	be the $\alpha$-equivariant boundary map from~\eqref{eq:boundary map}, 
	where $\alpha\colon\Gamma\times X\to\Gamma$ is a ME-cocycle for $\Omega$. 	
  	If $\sigma=(z_0,\dots,z_n)$ with $z_i\in B$ is a positively oriented ideal regular simplex, then the Euler number of $\Omega$ satisfies
	\begin{equation*}
 		\euler(\Omega)=
    	\frac{\vol(\Gamma\bs\Hsp^n)}{v_{\mathrm{max}}}\int_{\Gamma\bs G^0}\int_X
     	\vol\bigl(\phi_x(gz_0),\dots,\phi_x(gz_n)\bigr)\,d\mu(x)\,dg, 
	\end{equation*}
	where $v_{\mathrm{max}}$ is the volume of a positively oriented
	ideal maximal simplex in $B^{n+1}$ and the
	quotient $\Gamma\bs G^0$ carries the normalized Haar measure.
\end{theorem}
Note that the function 
\[
	g\mapsto \int_X
     \vol\bigl(\phi_x(gz_0),\dots,\phi_x(gz_n)\bigr)\,d\mu(x) 
\] 
in the previous statement 
is $\Gamma$-invariant by $\alpha$-equivariance of $\phi$, $G^0$-invariance of the volume, and $\Gamma$-invariance of $\mu$. 
So the integral in Theorem~\ref{thm:evaluation by fundamental class} makes sense.

The following immediate corollary, which we will not use in this paper, can be viewed 
as a higher-dimensional cocycle analog of the Milnor-Wood
inequality for homomorphisms of a surface group into $\Homeo_+(S^1)$.
We will present an independent stronger result, valid under an integrability assumption,
in Corollary~\ref{cor:maximality}.

\begin{cor}[Higher-dimensional Milnor-Wood inequality] \label{cor:MW}
In the setting of Theorem~\ref{thm:evaluation by fundamental class} we have
$\abs{\euler(\Omega)}\le\vol(\Gamma\bs\Hsp^n)$.
\end{cor}

We shall need the auxiliary
Lemma~\ref{lem:evaluation by fund class - first step: visual measures}
%and~\ref{lem:Lebesgue_differentiation}
before we conclude the proof
of Theorem~\ref{thm:evaluation by fundamental class} at the end of this subsection. We retain the setting
of Theorem~\ref{thm:evaluation by fundamental class} for the rest of this
subsection.

\begin{lemma}\label{lem:evaluation by fund class - first step: visual measures}
  If $\sigma=(z_0,\dots,z_n)$ with $z_i\in\Hsp^n$ is a positively
  oriented geodesic simplex, then the Euler number of $\Omega$ satisfies 
  \begin{equation*}
	\euler(\Omega)=
    \frac{\vol(\Gamma\bs\Hsp^n)}{\vol(\sigma)}
    \int_{B^{n+1}}
\int_{G^0/\Gamma}\int_X
     \vol\bigl(\phi_x(gb_0),\dots,\phi_x(gb_n)\bigr)\,d\mu(x)\,dg
    d\nu_\sigma.
  \end{equation*}
\end{lemma}

\begin{proof}
  For better readability, we keep the notational distinction between $\Gamma_l$ and 
  $\Gamma_r$ from Setup~\ref{basic setup} and denote the copy of $B$ on which
  $\Gamma_l$ acts by $B_l$; similarly for $B_r$.

  Consider the diagram below. The unlabeled maps are the obvious ones, sending a
  function to its equivalence class up to null sets and inclusion of constant functions.

  All the maps are $\Gamma_l\times\Gamma_r$-equivariant chain morphisms as
  explained now. On $\rmLweak^\infty(B_l^{\bullet+1},\bbR)$ and
  $\rmCb^\bullet(\Gamma_l,\bbR)$ we have the usual $\Gamma_l$-actions and the
trivial $\Gamma_r$-actions.
  The Poisson transform in the lower row is then
  clearly $\Gamma_l\times\Gamma_r$-equivariant. The actions on the domain and target
  of the maps $\rmCb^\bullet(\chi)$
  and $\rmCb^\bullet(\phi)$ are defined in
  Propositions~\ref{prop:induction from Monod-Shalom}
  and~\ref{prop:induction reciprocity}, and is proved there that these
  maps are
$\Gamma_l\times\Gamma_r$-equivariant. The Poisson transform in the upper row, which is
  $\Gamma_r$-equivariant, is also $\Gamma_l$-equivariant, since $\Gamma_l$ acts only
  by its natural action on $\Omega$.
  \[\xymatrix{
	\calB^\infty(B_l^{\bullet+1},\bbR)\ar[d]\ar[r]^-{\rmCb^\bullet(\phi)} & \rmLweak^\infty(B_r^{\bullet+1},\rmL^\infty(\Omega))\ar[r]^-{\rmPT^\bullet} & \rmCb^\bullet(\Gamma_r,\rmL^\infty(\Omega))\\
	\rmLweak^\infty(B_l^{\bullet+1},\bbR)\ar[r]^{\rmPT^\bullet} & \rmCb^\bullet(\Gamma_l,\bbR)\ar[r]&\rmCb^\bullet(\Gamma_l,\rmL^\infty(\Omega))\ar[u]^{\rmCb^\bullet(\chi)}
  }\]
  The diagram describes two $\Gamma_l\times\Gamma_r$-equivariant chain morphisms
\[
	\phi, \psi: \calB^\infty(B_l^{\bullet+1},\bbR)\to \rmCb^\bullet(\Gamma_r,\rmL^\infty(\Omega))
\]
for which we want to prove, using Theorem~\ref{thm:main homological theorem}, that they are $\Gamma_l\times \Gamma_r$-chain homotopic. By Proposition~\ref{prop:B-complex strong resolution} the source $\calB^\infty(B_l^{\bullet+1},\bbR)$ is a strong 
$\Gamma_l\times\Gamma_r$-resolution of $\bbR$. It is shown 
in~\cite{MonodShalom}*{Proof of Proposition~4.6.} that the 
target $\rmCb^\bullet(\Gamma_r,\rmL^\infty(\Omega))$ is a relatively injective and 
strong $\Gamma_l\times\Gamma_r$-resolution of $L^\infty(\Omega)$. Both $\phi$ and 
$\psi$ as the lower map make the diagram 
\[
	\xymatrix{\bbR\ar[r]\ar[d] & \rmL^\infty(\Omega)\ar[d]\\
	          \rmCb^\bullet(\Gamma_r,\rmL^\infty(\Omega)) \ar[r] & \rmCb^\bullet(\Gamma_r,\rmL^\infty(\Omega))
	},
\] 
where the upper map is the inclusion of constant functions, 
commutative, that is, $\phi$ and $\psi$ are morphisms between the augmented resolutions. 
By Theorem~\ref{thm:main homological theorem}, $\phi$ and $\psi$ are equivariantly 
chain homotopic. Taking invariants and cohomology, this means that the following diagram is commutative: 
\[\xymatrix{
	\rmH^n\bigl(\calB^\infty(B_l^{\bullet+1},\bbR)^{\Gamma_l}\bigr)\ar[d]\ar[r]^-{\rmH^n(\phi)} & \rmH^n\bigl(\rmLweak^\infty(B_r^{\bullet+1},\rmL^\infty(\Gamma_l\bs\Omega))\bigr)\ar[r]^-{\rmH^n(\rmPT^\bullet)} & \rmHb^n\bigl(\Gamma_r,\rmL^\infty(\Gamma_l\bs\Omega)\bigr)\\
	\rmH^n\bigl(\rmLweak^\infty(B_l^{\bullet+1},\bbR)^{\Gamma_l}\bigr)\ar[r]^-{\rmH^n(\rmPT^\bullet)} & \rmHb^n\bigl(\Gamma_l,\bbR\bigr)\ar[r]^-{\rmHb^n(j^\bullet)}&\rmHb^n(\Gamma_l,\rmL^\infty(\Gamma_r\bs\Omega)\bigr)\ar[u]^{\rmHb^n(\Omega)}
 }\]

  The volume cocycle $\dvol_b$, which we defined as a
  cocycle in $\rmLweak^\infty(B^{n+1},\bbR)$, is everywhere defined
  and everywhere $\Gamma$-invariant and strictly satisfies the
  cocycle condition; hence it lifts to a cocycle in
  $\calB^\infty(B^{n+1},\bbR)$ which we denote by
  $\dvol_{\mathrm{strict}}$.
  Now we have 
  \begin{align*}%\label{eq: another euler}
  	\euler(\Omega) &= \bigl\langle\comp^n\circ\,\rmHb^n(\rmI^\bullet)\circ \rmHb^n(\Omega)\circ\rmHb^n(j^\bullet)\circ
	\rmHb^n(\rmPT^\bullet)(\dvol_b),[\Gamma]\bigr\rangle\\
	&=\bigl\langle\comp^n\circ\,\rmHb^n(\rmI^\bullet)\circ
	\rmHb^n(\rmPT^\bullet)\circ\rmHb^n(\phi)(\dvol_{\rm strict}),[\Gamma]\bigr\rangle\\
	&= \bigl\langle\comp^n\circ\,
	\rmHb^n(\rmPT^\bullet)\circ\rmHb^n(\rmI^\bullet)\circ \rmHb^n(\phi)(\dvol_{\rm strict}),[\Gamma]\bigr\rangle.
	\end{align*}
   Here the first equality is just the definition of 
   the Euler class as in Definition~\ref{eq:Euler number}; just be aware that there we  
   denoted $\rmHb^n(\rmPT^\bullet)(\dvol_b)$ by the same symbol $\dvol_b$ since 
   $\rmHb^n(\rmPT^\bullet)$ is a canonical isomorphism between two resolutions computing  
bounded cohomology in the functorial approach. The second equality 
follows by the commutativity of the above diagram. The third equality is true 
since the cohomological Poisson transform is natural in the coefficients, hence 
  the integration $\rmHb^n(\rmI^\bullet)$ in the coefficients and 
  $\rmHb^n(\rmPT^\bullet)$ interchange. 
  We invoke
Lemma~\ref{lem:image of fund class under Patterson-Sullivan map} with  $[f]=\rmHb^n(\rmI^\bullet)\circ \rmHb^n(\phi)(\dvol_{\rm strict})$ to conclude the proof.
\end{proof}

%----------------

\begin{proof}[Proofs of Theorems~\ref{thm:bounded volume cocycle to volume coycle} and~\ref{thm:evaluation by fundamental class}]
  We start with the proof of Theorem~\ref{thm:evaluation by fundamental class}.
  For every $i\in\{0,\dots,n\}$ we pick a sequence
  $(z_i^{(k)})_{k\in\bbN}$ on the geodesic ray from a basepoint $o\in\Hsp^n$ to $z_i$
  converging to $z_i$.  Let $\sigma_k$ be the geodesic simplex
  spanned by the vertices $z_0^{(k)},\dots, z_n^{(k)}$.  By
  Lemma~\ref{lem:evaluation by fund class - first step: visual
    measures},
  \begin{equation*}%\label{eq:summand comparision}
	\euler(\Omega)=
    \frac{\vol(\Gamma\bs\Hsp^n)}{\vol(\sigma)}
    \int_{B^{n+1}}
\int_{\Gamma\bs G^0}\int_X
     \vol\bigl(\phi_x(gb_0),\dots,\phi_x(gb_n)\bigr)\,d\mu(x)\,dg
    d\nu_\sigma.
\end{equation*}
  We now let $k$ go to $\infty$. Note that the left hand side does not
  depend on $k$.  First of all, the volumes $\vol(\sigma_k)$ converge
  to $\vol(\sigma)=v_{\mathrm{max}}$.
  By Lemma~\ref{lem:Lebesgue_differentiation},
  \begin{equation*}
 	\lim_{k\to \infty}\nu_{\sigma_k}
	\bigl\{(b_0,\dots,b_n)\mid d(\phi_x(gz_i),\phi_x(gb_i))<\epsilon\bigr\}= 1
\end{equation*}
  for every $\epsilon>0$ and a.e.~$(x,g)\in X\times G$.

It is shown in \cite{ratcliffe}*{Theorem~11.4.2 on p.~541} that the volume, $\vol$, is
a continuous function on the open set $B^{(n+1)}$ of all $(n+1)$-tuples in general position
(see Definition~\ref{def:gen-pos}).
Thus, by Lemma~\ref{L:gen-pos}, $\vol$ is continuous
	at a.e. $(\phi_x(gz_0),\dots,\phi_x(gz_n))$, and therefore
  % Since the volume is continuous on $B^{n+1}$ for
  % $n\ge 3$~\cite{ratcliffe}*{Theorem~11.4.2 on p.~541} and
  % constant on non-degenerate ideal simplices for $n=2$,
  % this implies that
  \begin{equation*}
    \lim_{k\to\infty}\int_{B^{n+1}}
    \vol\bigl(\phi_x(gb_0),\dots,\phi_x(gb_n)\bigr)
    \,d\nu_{\sigma_k}=
    \vol\bigl(\phi_x(g z_0),\dots,\phi_x(g z_n)\bigr),
\end{equation*}
  for a.e.~$(x,g)\in X\times G$,
  which finally yields Theorem~\ref{thm:evaluation by fundamental class} by the dominated
  convergence theorem.
  The proof of Theorem~\ref{thm:bounded volume cocycle to volume coycle} is even easier
  since it does not require
Lemma~\ref{lem:Lebesgue_differentiation}.
  One obtains from Lemma~\ref{lem:image of fund class under Patterson-Sullivan map} that
  \begin{multline*}
    \langle \comp^n(\dvol_b), [\Gamma]\rangle=
    \frac{\vol(\Gamma\bs\Hsp^n)}{\vol(\sigma)}
    \int_{B^{n+1}}
\int_{\Gamma\bs G^0}\int_X
    \vol\bigl(gb_0,\dots,gb_n\bigr)\,d\mu(x)\,dg
    d\nu_\sigma,
\end{multline*}
   which converges for $k\to\infty$ to $\vol(\Gamma\bs\Hsp^n)$ by continuity
   of the function $\vol$ at a.e.~point $(\phi_x(gb_0),\dots,\phi_x(gb_n))\in B^{n+1}$
   (Lemma~\ref{L:gen-pos})
	and the weak convergence of $\nu_{z_i^{(k)}}$
   to the point measure at $z_i$ for every $i\in\{0,\dots,n\}$.
\end{proof}

\subsection{Adding integrability assumption}
%\subsubsection{\textbf{Maximality of the Euler number provided $\Omega$ is integrable}} % (fold)
\label{ssub:maximality_of_the_euler_number_provided_omega_is_integrable}

In this subsection we appeal to a general result from our companion
paper~\cite{sobolev}, which relies on the
integrability of the coupling.
We get that, in the presence of such an integrability assumption, the Milnor-Wood inequality
given in Corollary~\ref{cor:MW} becomes an equality, see Corollary~\ref{cor:maximality} below.

\begin{theorem}[\cite{sobolev}*{Theorem~5.12} and~\cite{sobolev}*{Corollary~1.11}]\label{thm:main result about induction in cohomology}
	Let $M$ and $N$ be closed, oriented, negatively curved manifolds of
	dimension $n$.
	Let $(\Omega,\mu)$ be
	an ergodic, integrable ME-coupling $(\Omega,\mu)$ of the fundamental groups
	$\Gamma=\pi_1(M)$ and $\Lambda=\pi_1(N)$,
        and set $c= \frac{\mu({\Lambda}\bs\Omega)}{\mu({\Gamma}\bs\Omega)}$.
	Suppose that $x_\Gamma^b\in\rmHb^n(\Gamma,\bbR)$ is an element that maps
    to the cohomological fundamental class
    $x_\Gamma\in\rmH^n(\Gamma,\bbR)\cong \rmH^n(M,\bbR)$ of $M$
    under the comparison map. Define
	$x_\Lambda\in\rmH^n(\Lambda,\bbR)$ analogously.
	Then the composition
	\begin{multline}\label{eq:composition induction cohomological}
		\rmHb^n(\Gamma, \bbR)\xrightarrow{\rmHb^n(j^\bullet)} \rmHb^n(\Gamma,\rmL^\infty(\Lambda\bs\Omega))
	    \xrightarrow{\rmHb^n(\Omega)}\rmHb^n(\Lambda,\rmL^\infty(\Gamma\bs\Omega))\\
     	\xrightarrow{\rmHb^n(\rmI^\bullet)} \rmHb^n(\Lambda,\bbR)\xrightarrow{\comp^n} \rmH^n(\Lambda,\bbR)
	\end{multline}
    sends $x_\Gamma^b$ to $\pm c\cdot x_\Lambda$. 
	Furthermore, if $\Gamma\cong\Lambda$, then $c=1$.
\end{theorem}

\begin{cor}[Maximality of the Euler class] \label{cor:maximality}
Retain the setting of Theorem~\ref{thm:evaluation by fundamental class}. 
If, in addition, the coupling $\Omega$ is integrable, then
\begin{equation}\label{eq:eval against fundamental class}
\euler(\Omega)=\pm\vol\bigl(\Gamma\bs\Hsp^n\bigr).
\end{equation}
\end{cor}

\begin{proof}
We apply Theorem~\ref{thm:main result about induction in cohomology} to 
$M=N=\Gamma\bs\Hsp^n$ and $\Lambda=\Gamma$.
One has 
\[
	\dvol=\vol(M)\cdot x_{\Gamma}
\]
because the top degree cohomology is one-dimensional, 
and the evaluation against the homological fundamental class
gives the equality.
By Theorem~\ref{thm:bounded volume cocycle to volume coycle}
\[
	\dvol=\comp^n(\dvol_b).
\]
Thus, $x_\Gamma^b=x_\Lambda^b=\dvol_b/\vol(M)$ satisfy the conditions of the theorem.
Since in this case $c=1$,
we conclude that
$\dvol_b$ is mapped to $\pm\dvol$
under~\eqref{eq:composition induction cohomological}.
Equation~(\ref{eq:euler}) of Definition~\ref{eq:Euler number}
gives
\begin{align*} 
	\euler(\Omega)&=\bigl\langle \comp^n\circ\;\rmHb^n(\rmI^\bullet)\circ\rmHb^n(\Omega)\circ
	\rmHb^n(j^\bullet)(\dvol_b), [\Gamma]\bigr\rangle\\
	&=\bigl\langle \pm\dvol, [\Gamma]\bigr\rangle
	=\pm\vol\bigl(\Gamma\bs\Hsp^n\bigr).\qedhere
\end{align*}
\end{proof}

%%%%%

Recall that $\Sigma^{\rm reg}_+$ (resp. $\Sigma^{\rm reg}_-$) denotes the set of
positively (resp. negatively) oriented regular ideal simplices 
(see Subsection~\ref{sub:preserving maximal simplices}).
We think of $\Sigma^{\rm reg}_+$ and $\Sigma^{\rm reg}_-$ as subsets of $B^{n+1}$
- the $(n+1)$-tuples of points on the boundary $B=\partial\Hsp^n$.
Since $\Sigma^{\rm reg}_+$ (resp.~$\Sigma^{\rm reg}_-$) is a single $G^0$-orbit, terms like  \emph{a.e. point on} $\Sigma^{\rm reg}_+$ refer to the Haar measure on $G^0$.

\begin{cor} \label{cor:euler-boundary}
  	Retain the setting of Theorem~\ref{thm:evaluation by fundamental class}. 
	If, in addition, the coupling $\Omega$ is integrable, then the
	Borel map 
	\[
		\phi_x^{n+1}=\phi_x\times\cdots\times\phi_x:B^{n+1}\to B^{n+1}
	\] 
	either maps a.e.~point in $\Sigma^{\rm reg}_+$ into $\Sigma^{\rm reg}_+$ for $\mu$-a.e $x\in X$, 
	or a.e. point of $\Sigma^{\rm reg}_+$ is mapped into $\Sigma^{\rm reg}_-$ for $\mu$-a.e $x\in X$.
\end{cor}

\begin{proof}
Fix $(z_0,\dots,z_n)\in \Sigma^{\rm reg}_+$.
Combining Corollary~\ref{cor:maximality}
with
Theorem~\ref{thm:evaluation by fundamental class} we get
\begin{equation*}
\int_{\Gamma\bs G^0}\int_X
     \vol\bigl(\phi_x(gz_0),\dots,\phi_x(gz_n)\bigr)\,d\mu(x)\,dg=\pm  v_{\mathrm{max}}.
\end{equation*}
By Lemma~\ref{lem:key facts from Mostow} (2), an ideal simplex has an oriented volume
$v_{\mathrm{max}}$ iff it is in $\Sigma^{\rm reg}_+$ and $-v_{\mathrm{max}}$ iff it is in $\Sigma^{\rm reg}_-$.
Combining this with the fact that
the absolute value of the integrand on the left hand side in the above formula is a priori at most
$v_{\mathrm{max}}$ implies 
that either
for a.e. $(g,x)\in G^0\times X$,
$\bigl(\phi_x(gz_0),\dots,\phi_x(gz_n)\bigr)\in \Sigma^{\rm reg}_+$
or
for a.e. $(g,x)\in G^0\times X$,
$\bigl(\phi_x(gz_0),\dots,\phi_x(gz_n)\bigr)\in \Sigma^{\rm reg}_-$.
But
by Lemma~\ref{lem:key facts from Mostow} (1), 
$G^0$ acts simply transitive on $\Sigma^{\rm reg}_+$,
thus for the set of $g\in G^0$ satisfying the above generic condition, the set of
ideal simplices of the form $(gz_0,\dots,gz_n)$ is of full measure in $\Sigma^{\rm reg}_+$.
The proof now follows by an application of Fubini's theorem.
\end{proof}

%%%%%%%%%%%%%%

\section{Proofs of the main results} % (fold)
\label{sec:proofs_of_the_main_results}

%In this section we use the results of Section~\ref{sec:measure_equivalence} and
%Theorems~\ref{T:SOn1-1-taut} and~\ref{thm:taut relative homeo} to prove the remaining
%results stated in the introduction.

\subsection{Proof of Theorems~\ref{T:SOn1-1-taut} and \ref{thm:taut relative homeo}}

\begin{proof}[Proof of  Theorem~\ref{T:SOn1-1-taut}]\hfill{}\\
We aim to show that the group $G=\isom(\Hsp^n)$ is $1$-taut for any $n\ge 3$.
Fix a cocompact torsion-free lattice $\Gamma<G^0$ in the connected component of $e\in G$.
By Proposition~\ref{exa:examples of strongly icc groups} $G$ is strongly ICC,
thus Proposition~\ref{P:taut-lattice} applies and it is enough to show that $\Gamma$ is 1-taut
relative to $G$ (note that a cocompact lattice is integrable).
By Lemma~\ref{lem:ergodic dec} (applied to $\Gamma$) it is enough to show that every integrable ergodic $(\Gamma,\Gamma)$-coupling is taut relative to $G$.

Let $(\Omega,m)$ be an integrable ergodic $(\Gamma,\Gamma)$-coupling. 
We adopt the Setup~\ref{basic setup} and consider the boundary map~(\ref{eq:boundary map}) $\phi:X\times B\to B$. 
Let $\Sigma^{\rm reg}_+,\Sigma^{\rm reg}_-\in B^{n+1}$
denote the sets of positively and negatively oriented regular ideal simplices,
as defined in Subsection~\ref{sub:preserving maximal simplices}.
Then Corollary~\ref{cor:euler-boundary} implies that
either for $\mu$-a.e $x\in X$,
$\phi_x^{n+1}=\phi_x\times\cdots\times\phi_x$ maps a.e.~point
in $\Sigma^{\rm reg}_+$ into $\Sigma^{\rm reg}_+$,
or
for $\mu$-a.e $x\in X$,
$\phi_x^{n+1}$ maps a.e.~point
in $\Sigma^{\rm reg}_+$ into $\Sigma^{\rm reg}_-$.

We now use the assumption that $n\geq 3$, and 
apply Lemma~\ref{lem:key lemma in mostow rigidity}
to deduce that for a.e $x\in X$ there exists a unique $g_x\in G$
with $\phi_x(b)=g_x b$ for a.e.~$b\in B$.	
Proposition~\ref{prop:reduction} applied to $\calG$ being the image of $G$ in $\Homeo(B)$
yields that $\Omega$ is taut with respect to $G$.
\end{proof}

We now set the stage for the proof of Theorem~\ref{thm:taut relative homeo} 
which deals with the case $n=2$. If we normalize the 
volume cocycle~\eqref{eq: volume cocycle} by the volume $v_{\mathrm{max}}$ 
of a non-degenerate positively oriented ideal $2$-simplex in $\Hsp^2\cup\bar\Hsp^2$ -- 
they all have the same volume -- we obtain the \emph{orientation cocycle} $c$ 
defined on triples of points on the circle $S^1=B=\partial\Hsp^2$ by
\[
	c(b_0,b_1,b_3)=v_{\mathrm{max}}^{-1}\cdot \vol(b_0,b_1,b_2).
\]
It takes values in $\{-1,0,1\}$ with $c (b_0,b_1,b_2)=1$ if the triple $(b_0,b_1,b_2)$ consists of distinct points
in the positive orientation/cyclic order, $c=-1$ if the cyclic order is reversed, and $c=0$ if the triple is degenerate.
Let $\nu$ denote a probability measure in the Lebesgue class, and suppose that
$\phi:(S^1,\nu)\to S^1$ is a measurable map so that for $\nu^3$-a.e. $(b_0,b_1,b_2)$:
\[
	c\bigl(\phi(b_0), \phi(b_1), \phi(b_2)\bigr)=c(b_0,b_1,b_2).
\]
It follows from~\cite{iozzi}*{Proposition~5.5} that
the following conditions on such measurable orientation preserving $\phi:(S^1,\nu)\to S^1$ are equivalent:
\begin{enumerate}
	\item The push-forward measure $\phi_*\nu$ has full support;
	\item $\phi$ agrees a.e. with a homeomorphism $f\in \Homeo(S^1)$.
\end{enumerate}
Let $\Gamma<G=\isom(\Hsp^2)$ be a lattice. 
Let $\alpha:\Gamma\times X\to\Gamma$ be the ME-cocycle associated with an ergodic
$(\Gamma,\Gamma)$-coupling $(\Omega,m)$ and an identification $i:\Gamma\times X\to\Omega$.
Let $\phi_x\colon (S^1,\nu)\to S^1$, $x\in X$, be the boundary 
map~\ref{eq:boundary map} associated
to $\alpha$ as in
Section~\ref{s:boundary}. 

\begin{proposition}\label{prop:orientation cocycle maximal}
    If the orientation cocycle is preserved by $\phi_x$ a.e., that is,
	\[
		c\bigl(\phi_x(b_0),\phi_x(b_1),\phi_x(b_2)\bigr)=c(b_0,b_1,b_2)\quad\text{$\nu^3$-a.e}
	\]
    for a.e.~$x\in X$, then the map
    $\phi_x$ agree a.e. with a homeomorphism $f_x\in \Homeo(S^1)$ for a.e. $x\in X$.
\end{proposition}

\begin{proof}
	We prove that the measurable family of open sets $U_x=S^1\setminus \supp(\phi_x\nu)$
	satisfies a.e. $U_x=\emptyset$.
	The fact that $\nu$ is $\Gamma$-quasi-invariant and the identity
	\[
		\phi_{\gamma.x}(\gamma b)=\alpha(\gamma,x)\phi_x(b)
	\]
	imply the following \emph{a priori} equivariance of $\{U_x \mid x\in X\}$
	\begin{equation}\label{e:open-inv-set}
		U_{\gamma.x}=\alpha(\gamma,x)\,U_x.
	\end{equation}
	Since $U_x\neq S^1$ for every $x\in X$, the proposition
	is implied by Lemma~\ref{L:furstenberg} and the fact that
	the action of $G=\rmPSL_2(\bbR)$ and of its lattices on the
	circle $S^1$ is minimal and strongly
	proximal~\cite{furstenberg-bourbaki}*{Propositions~4.2 and~4.4}.
\end{proof}

\begin{proof}[Proof of Theorem~\ref{thm:taut relative homeo}]\hfill\\
	We fix a cocompact torsion-free lattice $\Gamma<G^0$ (a surface group) 
	in the connected component of $G=\isom(\Hsp^2)$. We identify $S^1$ as 
	$B=\partial\Hsp^2$ and embed $G$ in the Polish group $\calG:=\Homeo(S^1)$ 
	accordingly. 
	Exactly as at the start of the proof 
	of Theorem~\ref{T:SOn1-1-taut} ($n\ge 3$ was not needed for that) one 
	sees that it suffices to show that $\Gamma$ is $1$-taut relative to $\calG$. 
	
By Lemma~\ref{lem:ergodic dec} it is enough to show that every integrable ergodic 
$(\Gamma,\Gamma)$-coupling is taut relative to $\calG$.
Let $(\Omega,m)$ be such a coupling. We adopt the Setup~\ref{basic setup} and consider the boundary map~(\ref{eq:boundary map}) $\phi:X\times B\to B$. 

Corollary~\ref{cor:euler-boundary} implies
that there is $\sigma\in\{1,-1\}$ such that
for $\mu$-a.e $x\in X$ and a.e. triple $(b_1,b_2,b_3)\in (S^1)^3$:
\[
	c\bigl(\phi_x(b_1),\phi_x(b_2),\phi_x(b_3)\bigr)=\sigma\cdot c(b_1,b_2,b_3). 
\]
That is, either a.e. $\phi_x$ preserves the cyclic order of a.e.~triple,
or a.e $\phi_x$ reverses the cyclic order of a.e.~triple.
In either case, by Proposition~\ref{prop:orientation cocycle maximal} we conclude
that for a.e $x\in X$, $\phi_x$ agree a.e. with a homeomorphism $f_x\in \Homeo(S^1)$.
It follows with Proposition~\ref{prop:reduction} that the ergodic integrable
$(\Gamma,\Gamma)$-coupling $\Omega$ is taut relative to $\calG=\Homeo(S^1)$.
\end{proof}

%%%%%%%%

\subsection{Measure equivalence rigidity: Theorem~\ref{T:ME-rigidity}} % (fold)
\label{sub:measure_equivalence_rigidity_for_nge3_}\hfill{}\\
Let $G=\isom(\Hsp^n)$, $n\ge 3$.
Let $\Gamma<G$ be a lattice, and
$\Lambda$ a finitely generated group which admits an integrable $(\Gamma,\Lambda)$-coupling $(\Omega,m)$.
By Lemma~\ref{lem:composition of lp couplings} the $(\Gamma,\Gamma)$-coupling $\Omega\times_\Lambda\dual{\Omega}$ is integrable.

By Theorem~\ref{T:SOn1-1-taut} and Proposition~\ref{P:taut-lattice} the lattice $\Gamma$
is $1$-taut relative to the inclusion $\Gamma<G$. Hence
the coupling $\Omega\times_\Lambda\dual{\Omega}$ is taut. By
Proposition~\ref{exa:examples of strongly icc groups} the group $G$ is strongly ICC relative
to $\Gamma<G$.
Applying Theorem~\ref{T:split+hom} we obtain a continuous homomorphism $\rho:\Lambda\to G$
with finite kernel $F$, image $\bar{\Lambda}=\rho(\Lambda)$ being discrete in $G$,
and a measurable $\id_\Gamma\times\rho$-equivariant map
$\Psi:\Omega\to G$.

To complete the proof of statement (1) of Theorem~\ref{T:ME-rigidity} and case $n\ge 3$ of
Theorem~\ref{T:ME-rigidity2and3} it remains to show that $\bar\Lambda$
is not merely discrete, but is actually a lattice in $G$.
This can be deduced from the application of Ratner's theorem below which is needed
for the precise description of the push-forward measure $\Psi_*m$ on $G$
as stated in part (2) of Theorem~\ref{T:ME-rigidity}.
Let us also give the following direct argument which relies only on the strong ICC property of $G$.

\medskip

Consider the composition $(G,\Lambda)$-coupling $\widetilde\Omega=G\times_\Gamma\Omega$, and the
$(G,G)$-coupling $\widetilde\Omega\times_\Lambda\dual{\widetilde\Omega}$.
Since $\Gamma$ is an integrable lattice in $G$
(Theorem~\ref{thm:Shalom and lattices}) by Lemma~\ref{lem:composition of lp couplings}
both $\widetilde\Omega$ and $\widetilde\Omega\times_\Lambda\dual{\widetilde\Omega}$ are
integrable couplings.
Theorem~\ref{T:SOn1-1-taut} provides a unique tautening map
\[
	\widetilde\Phi:\widetilde\Omega\times_\Lambda\dual{\widetilde\Omega}\to G.
\]
Applying Theorem~\ref{T:reconstruction} (a special case of Theorem~\ref{T:split+hom} with $\calG=G$),
we obtain a homomorphism $\widetilde\rho:\Lambda\to G$ with finite kernel and image being a lattice in $G$.
There is also a ${\rm Id}_G\times \widetilde\rho$-equivariant measurable map
\[
	\widetilde\Psi:\widetilde\Omega=G\times_\Gamma\Omega \to G.
\]
We claim that $\rho, \widetilde\rho:\Lambda\to G$ are conjugate representations.
To see this observe that since $G$ is strongly ICC, there is only one tautening map
$\widetilde\Omega\times_\Lambda\dual{\widetilde\Omega}\to G$.
This implies the a.e. identity
\[
	\widetilde\Psi([g_1,\omega_1])\widetilde\Psi([g_2,\omega_2])^{-1}=g_1 \Psi(\omega_1)\Psi(\omega_2)^{-1} g_2^{-1}.
\]
Equivalently, we have a.e. identity
\[
	\Psi(\omega_1)^{-1}g_1^{-1}\widetilde\Psi([g_1,\omega_1])
	=\Psi(\omega_2)^{-1}g_2^{-1}\widetilde\Psi([g_2,\omega_2]).
\]
Hence the value of both sides are a.e. equal to a constant $g_0\in G$.
It follows that for a.e. $g\in G$ and $\omega\in \Omega$
\[
	g^{-1}\widetilde\Psi([g,\omega])=\Psi(\omega)g_0.
\]
Finally, the fact that $\Psi$, $\widetilde\Psi$ are $\rho$-, $\widetilde\rho$- equivariant respectively, implies:
\[
	\widetilde\rho(\lambda)=g_0 \rho(\lambda) g_0^{-1}\qquad(\lambda\in\Lambda).
\]
In particular, $\bar\Lambda=g_0^{-1}\widetilde\rho(\Lambda)g_0$ is a lattice in $G$.

\medskip

We proceed with the proof of statement (2): given the ${\rm Id}_\Gamma\times\rho$-equivariant
measurable map $\Psi:\Omega\to G$ we shall describe the pushforward $\Psi_*m$ on $G$.
(We shall use the discreteness of $\bar\Lambda=\rho(\Lambda)$, but the fact that it is
a lattice will not be needed; in fact, it will follow from the application of Ratner's theorem.)
Recall that the measure $\Psi_*m$ is invariant under the action
$x\mapsto \gamma x \rho(\lambda)^{-1}$,
and descends to a finite $\Gamma$-invariant measure $\mu$ on $G/\bar{\Lambda}$
and to a finite $\bar{\Lambda}$-invariant measure $\nu$ on $\Gamma\bs G$.
Assuming $m$ was $\Gamma\times\Lambda$-ergodic, $\mu$ and $\nu$ are ergodic under the $\Gamma$-
and $\bar\Lambda$-action, respectively.
One can now apply Ratner's theorem \cite{ratnerICM} to describe $\mu$, and thereby $\Psi_*m$,
as in~\cite{FurmanME}*{Lemma 4.6}.
For the reader's convenience we sketch the arguments.

Let $\Lambda^0=\bar{\Lambda}\cap G^0$; so either $\Lambda^0=\bar{\Lambda}$ or
$[\bar{\Lambda}:\Lambda^0]= 2$.
In the first case we set $\mu'=\mu$, in the
latter case let $\mu'$ denote the $2$-to-$1$ lift of $\mu$ to $G/\Lambda^0$.
Let $\Gamma^0=\Gamma\cap G^0$, and let $\mu^0$ be an ergodic component of $\mu'$ supported
on $G^0/\Lambda^0$.
We consider the homogeneous space
$Z=G^0/\Gamma^0\times G^0/\Lambda^0$ which is endowed with the
following probability measure
\[
	\tilde\mu^0=\int_{G^0/\Gamma^0} \delta_{g\Gamma^0}\times g_*\mu^0\,dm_{G^0/\Gamma^0}.
\]
Observe that $\tilde\mu^0$ well defined because $\mu^0$ is $\Gamma^0$-invariant.
Moreover, $\tilde\mu^0$ is invariant and ergodic for the action of the diagonal
$\Delta(G^0)\subset G^0\times G^0$ on $Z$.
Since $G^0$ is a connected group generated by unipotent elements,
Ratner's theorem shows that $\tilde\mu^0$ is \emph{homogeneous}.
This means that there is a connected Lie subgroup $L<G^0\times G^0$ containing $\Delta(G^0)$
and
a point $z\in Z$ such that the stabilizer $L_z$ of $z$
is a lattice in $L$ and $\tilde\mu^0$ is the push-forward of the normalized Haar
measure $m_{L/L_z}$ to the $L$-orbit $Lz\subset Z$.
Since $G^0$ is a simple group, there are only two possibilities for $L$: either
(i) $L=G^0\times G^0$ or (ii) $L=\Delta(G^0)$.

In case (i), $\tilde\mu^0$ is the Haar measure on
$G^0/\Gamma^0\times G^0/\Lambda^0$,
and $\mu^0$ is the Haar measure on $G^0/\Lambda^0$.
(In particular, $\Lambda^0$ is a lattice in $G^0$,
and $\Lambda$ is a lattice in $G$). The original measure $\mu$
may be either the $G$-invariant measure $m_{G/\bar{\Lambda}}$, or a $G^0$-invariant measure
on $G/\bar{\Lambda}$. In the latter case, by possibly multiplying $\Phi$ and
conjugating $\rho$
with some $x\in G\setminus G^0$, we may assume that
$\mu$ is the $G^0$-invariant probability measure on $G^0/\bar\Lambda$.

In case (ii), the fact that $L_z$ is lattice in $L=\Delta(G^0)$, implies that
$\mu^0$ and the original measure $\mu$ are atomic.
Since $\Gamma$ acts ergodically on $(G/\bar{\Lambda},\mu)$,
this atomic measure is necessarily supported and equidistributed on a finite $\Gamma$-orbit
of some $g_0\bar\Lambda\in G/\bar{\Lambda}$.
It follows that $\Gamma\cap g_0^{-1}\bar\Lambda g_0$ has finite index in $\Gamma$.
(This also implies that $\bar\Lambda$ is a lattice in $G$).
Upon multiplying $\Psi$ and conjugating $\rho$ by $g_0\in G$,
we may assume that $\Phi_*m$ is equidistributed on the double coset $\Gamma e \bar\Lambda$
and that $\Gamma$, $\bar\Lambda$ are commensurable lattices.
This completes the proof of Theorem~\ref{T:ME-rigidity}.

\subsection{Convergence actions on the circle: case $n=2$ of Theorem~\ref{T:ME-rigidity2and3}} % (fold)
\label{sub:actions_on_the_circle}\hfill{}\\
Let $\Gamma$ be a uniform lattice in
$G=\isom(\Hsp^2)\cong \rmPGL_2(\bbR)$. The group $G$ is a subgroup
of $\Homeo(S^1)$ by the natural action of $\rmPGL_2(\bbR)$ on
$S^1\cong\Psp^1$.
Consider
a compactly generated unimodular group $H$ that is $\rmL^1$-measure equivalent
to $\Gamma$. We will prove a more general statement than in
Theorem~\ref{T:ME-rigidity2and3}, which is formulated for discrete $H=\Lambda$.
Since $\Gamma$ is uniform, hence integrable in $G$, we can induce any
integrable $(\Gamma,H)$-coupling to an integrable
$(G,H)$-coupling (Lemma~\ref{lem:composition of lp couplings}). Let
$(\Omega,m)$ be an integrable $(G,H)$-coupling $(\Omega,m)$.

From Theorem~\ref{T:split+hom}
we obtain a continuous homomorphism $\rho:H\to\Homeo(S^1)$
with compact kernel and closed image $\bar{H}<\Homeo(S^1)$ and,
by pushing forward $m$,
a measure $\bar{m}$ on $\Homeo(S^1)$
that is invariant under all bilateral translations on $f\mapsto g f \rho(h)^{-1}$
with $g\in G$ and $h\in H$
and descends to a finite $\bar H$-invariant
measure $\mu$ on $G\bs\Homeo(S^1)$ and a finite $G$-invariant measure
$\nu$ on $\Homeo(S^1)/\bar{H}$.

The next step is to show that $\bar{H}$ can be conjugated into $G$.
To this end,
we shall use the existence of the finite $\bar{H}$-invariant measure $\mu$ on $G\bs \Homeo(S^1)$,
which may be normalized to a probability measure. We need the following
theorem which we prove relying on the deep work by
Gabai~\cite{Gabai} and Casson-Jungreis~\cite{Casson} on the determination of
convergence groups as Fuchsian groups.

\begin{theorem}\label{T:intoPGL2}
Let $\mu$ be a Borel probability measure on $G\bs\Homeo(S^1)$.
Then the stabilizer $H_\mu=\{ f\in \Homeo(S^1) \mid f_*\mu=\mu\}$ for the action by the right translations is conjugate to a closed subgroup of $G$.
\end{theorem}

\begin{proof}
We fix a metric $d$ on the circle, say $d(x,y)=\measuredangle(x,y)$.
Let $\rmTpl\subset S^1\times S^1\times S^1$ be the space of distinct triples on the circle. The group $\Homeo(S^1)$ acts diagonally on $\rmTpl$.
We denote elements in $\rmTpl$ by bold letters $\mathbf x\in\rmTpl$; the
coordinates of $\mathbf x\in\rmTpl$ or $\mathbf y\in\rmTpl$
will be denoted by $x_i$ or $y_i$ where $i\in\{1,2,3\}$, respectively.
For $f\in\Homeo(S^1)$ we write $f(\mathbf x)$ for $(f(x_1),f(x_2),f(x_3))$.
We equip $\rmTpl$ with the metric, also denoted by $d$, given by
\[
	d(\mathbf{x},\mathbf{y})=\max_{i\in\{1,2,3\}}d(x_i,y_i).
\]
The following lemma will eventually allow us to apply the work of Gabai-Casson-Jungreis.
\renewcommand{\qedsymbol}{}
\end{proof}
\begin{lemma}\label{L:Hmu}
For every compact subset $K\subset \rmTpl$ and every $\epsilon>0$ there is $\delta>0$
so that for all $h,h'\in H_\mu$ and $\mathbf{y}\in K\cap h^{-1}K$ and
$\mathbf{y}'\in K\cap h'^{-1}K$
one has the implication:
\[
	d(\mathbf{y},\mathbf{y}')<\delta~\text{ and }~d(h(\mathbf{y}),h'(\mathbf{y}'))<\delta
 	~\Longrightarrow~
	\sup_{x\in S^1} d(h(x),h'(x))<\epsilon
\]
\end{lemma}
\begin{proof}
For an arbitrary triple $\mathbf{z}\in \rmTpl$ and $x\in S^1\setminus\{z_3\}$ consider the real valued \emph{cross-ratio}
\begin{equation*}%\label{e:cross-ratio}
	[x,z_1;z_2,z_3]=\frac{(x-z_1)(z_2-z_3)}{(x-z_3)(z_2-z_1)}.
\end{equation*}
In this formula we view the circle as the one-point compactification of the real line.
Denote by $[z_1,z_2]_{z_3}$ the circle arc from $z_1$ to $z_2$ not including $z_3$.
As a function in the first variable,  $[\_,z_1;z_2,z_3]$
is a monotone homeomorphism between the closed arc $[z_1,z_2]_{z_3}$
and the interval $[0,1]$.
For $f\in\Homeo(S^1)$ and $\mathbf{z}\in\rmTpl$ we define the function
\[
	F_{\mathbf{z},f}:[z_1,z_2]_{z_3}\to[0,1],~~ F_{\mathbf{z},f}(x)=[f(x),f(z_1);f(z_2),f(z_3)].
\]
Since the cross-ratio is invariant under $G$~\cite{ratcliffe}*{Theorem~4.3.1 on p.~116},
we have
$F_{\mathbf{z},gf}(x)=F_{\mathbf{z},f}(x)$ for any $g\in G$.
Hence we may and will use the notation $F_{\mathbf{z},Gf}(x)$.
We now average $F_{\mathbf{z},Gf}(x)$ with regard to the measure
$\mu$ and obtain the function
$\bar{F}_{\mathbf{z}}:[z_1,z_2]_{z_3}\to[0,1]$ with
\[
	\bar{F}_{\mathbf{z}}(x)=\!\int\limits_{G\bs\Homeo(S^1)} \!\!\!F_{\mathbf{z},Gf}(x)\,d\mu(Gf).
\]
The $H_\mu$-invariance of $\mu$
implies that
\begin{equation}\label{eq: H_mu equivariance}
	\bar{F}_{h(\mathbf{z})}(h(x))=\bar{F}_{\mathbf{z}}(x)
\end{equation}
for every $h\in H_\mu$ and every $x\in [z_1,z_2]_{z_3}$. Let us introduce the
following notation: Whenever $K\subset\rmTpl$ is a subset, we denote by $\widetilde K$
the subset
\[
	\widetilde{K}=\bigl\{ (x,\mathbf{z}) \mid \mathbf{z}\in K,~ x\in [z_1,z_2]_{z_3}\bigr\}
	\subset S^1\times S^1\times S^1\times S^1.
\]
Next let us establish the following continuity properties:
\begin{enumerate}
	\item \label{i:barFinv}
	For every compact $K\subset\rmTpl$ and every $\epsilon>0$ there is
	$\eta>0$ such that:
	\[
		\forall_{(s,\mathbf{z}),(t,\mathbf z)\in\widetilde{K}}~~\Bigl( |\bar{F}_{\mathbf{z}}(t)-\bar{F}_{\mathbf{z}}(s)|<\eta
		~\Rightarrow~ d(t,s)<\frac{\epsilon}{5}\Bigr)\]
	\item \label{i:equiF}
	For every compact $K\subset\rmTpl$ and every $\eta>0$ there is
	$\delta>0$ such that:
	\[
		\forall_{(t,\mathbf{y}),(t,\mathbf{z})\in \widetilde{K}}~~\Bigl(
		d(\mathbf{y}, \mathbf{z})<\delta
		~\Rightarrow~
		|\bar{F}_{\mathbf{y}}(t)-\bar{F}_{\mathbf{z}}(t)|<\frac{\eta}{2}\Bigr)
	\]
\end{enumerate}
\emph{Proof of~(\ref{i:barFinv}):}
Let $K\subset\rmTpl$ be compact and $\epsilon>0$. Let $f\in\Homeo(S^1)$.
The family of homeomorphisms
$\bar{F}_{\mathbf{z},Gf}\colon [z_1,z_2]_{z_3}\to[0,1]$ depends
continuously on $\mathbf{z}\in\rmTpl$.
The inverses of these functions are equicontinuous
when $\mathbf{z}$ ranges in a compact subset.
Hence there exists
$\theta(Gf)>0$ such
that for every $\mathbf{z}\in K$ and all $t,s\in [z_1,z_2]_{z_3}$ we have the implication
\[
	\abs{F_{\mathbf{z},Gf}(t)-F_{\mathbf{z},Gf}(s)}<\theta(Gf)
		~\Rightarrow~d(t,s)<\frac{\epsilon}{5}.
\]
The set $G\bs\Homeo(S^1)$ is the union of an increasing sequence of measurable sets
\[
	A_n=\bigl\{Gf\in G\bs H \mid \theta(Gf)>\frac{1}{n} \bigr\}.
\]
Fix $n$ large enough so that $\mu(A_n)>1/2$.
We claim that $\eta=(2n)^{-1}$ satisfies (\ref{i:barFinv}).
Suppose that $\mathbf{z}\in K$ and $t,s\in [z_1,z_2]_{z_3}$
satisfy $d(t,s)>\epsilon/5$.
Up to exchanging $t$ and $s$, we may assume that $[s,z_1;z_2,z_3]\ge [t,z_1;z_2,z_3]$.
Then $F_{\mathbf{z},Gf}(s)\ge F_{\mathbf{z},Gf}(t)$ for all $f\in\Homeo(S^1)$, and
\[
	\bar{F}_{\mathbf{z}}(s)-\bar{F}_{\mathbf{z}}(t)\ge
	\int_{A_n} (F_{\mathbf{z},Gf}(s)-F_{\mathbf{z},Gf}(t))\,d\mu>\mu(A_n)\cdot \frac{1}{n}>\eta.
\]
\emph{Proof of (\ref{i:equiF}):} Let $K\subset\rmTpl$ be compact, and let $\eta>0$.
Let $f\in\Homeo(S^1)$. Since $\widetilde K$ is compact, $\bar{F}_\mathbf{z}(x)$ as a function
on $\widetilde K$ is equicontinuous. Hence
there is $\delta(Gf)>0$ such that for all $(x,\mathbf y)\in \widetilde K$ and
$(x,\mathbf z)\in \widetilde K$ with $d(\mathbf{y}, \mathbf{z})<\delta(Gf)$ we have
\[
	\abs{F_{\mathbf{y},Gf}(x)-F_{\mathbf{z},Gf}(x)}<\frac{\eta}{2}.
\]
The set $G\bs\Homeo(S^1)$ is the union of an increasing sequence of measurable sets
\[
	B_n=\bigl\{Gf\in G\bs H \mid \delta(Gf)>\frac{1}{n}\bigr\}.
\]
We choose $n\in\bbN$ with $\mu(B_n)>1-\eta/2$ and set $\delta=n^{-1}$.
Then for $(x,\mathbf y)\in \widetilde K$ and
$(x,\mathbf z)\in \widetilde K$ with $d(\mathbf{y}, \mathbf{z})<\delta$ we have
\[
	|\bar{F}_{\mathbf{y}}(x)-\bar{F}_{\mathbf{z}}(x)|\le
	\int_{B_n}|F_{\mathbf{y},Gf}(x)-F_{\mathbf{z},Gf}(x)|\,d\mu(Gf) + \frac{\eta}{2}
	<\eta,
\]
proving (\ref{i:equiF}).

\smallskip

We can now complete the proof of the lemma. Let $K\subset \rmTpl$ be a compact subset. Let  $\epsilon>0$. We can choose $r>0$ such that
\[
	K\subset\bigl\{\mathbf{x}\in\rmTpl \mid d(x_1, x_2), d(x_2,x_3), d(x_3,x_1)\ge r\bigl\}.
\]
For the given $\epsilon$ and $K$ let $\eta>0$ be as in~(\ref{i:barFinv}).
For the given $\epsilon$ and $K$ and this $\eta$ let $\delta>0$ be as in~(\ref{i:equiF}).
We may also assume that
\[
	\delta<\frac{\epsilon}{5}<\frac{r}{3}.
\]
Consider $h,h'\in H_\mu$ and $\mathbf{y}, \mathbf{y}'\in K$ where $\mathbf{z}=h(\mathbf{y})$,
$\mathbf{z}'=h'(\mathbf{y}')$ are also in $KÅ$, and assume that
$d(\mathbf{y}, \mathbf{y}')<\delta$ and $d(\mathbf{z}, \mathbf{z}')<\delta$.
There are several possibilities for the cyclic order of the points $\{y_1, y_1', y_2, y_2', y_3, y_3'\}$,
but since the pairs $\{y_i, y_i'\}$ of corresponding points in the triples $\mathbf{y}, \mathbf{y}'$
are closer ($d(y_i,y_i')<\delta<r/3$)  than the separation between the points in the triples
($d(y_i,y_j), d(y_i', y_j')\ge r$),
these points define a  partition of the circle into three long arcs $L_{ij}$
separated by three short arcs $S_k$ (possibly degenerating into points) in the following
cyclic order
\[
	S^1=L_{12} \cup S_2 \cup L_{23} \cup S_3 \cup L_{31} \cup S_1.
\]
The end points of the arc $S_i$ are $\{y_i, y_i'\}$; and if $(i,j,k)=(1,2,3)$ up to a cyclic permutation, then
\[
	L_{ij}=[y_i,y_j]_{y_k}\cap[y_i',y_j']_{y_k'}.
\]
Note that for any $x\in L_{ij}$ we have
\[
	h(x), h'(x)\in [z_i, z_j]_{z_k}\cap [z_i', z_j']_{z_k'}.
\]
Using~(\ref{i:equiF}) and~\eqref{eq: H_mu equivariance} we obtain
\begin{align*}
	|\bar{F}_{\mathbf{z}}(h(x))-\bar{F}_{\mathbf{z}}(h'(x))| &\le
	 |\bar{F}_{\mathbf{z}}(h(x))-\bar{F}_{\mathbf{z}'}(h'(x))|
		+|\bar{F}_{\mathbf{z}'}(h'(x))-\bar{F}_{\mathbf{z}}(h'(x))|\\
	&\le |\bar{F}_{\mathbf{z}}(h(x))-\bar{F}_{\mathbf{z}'}(h'(x))|+\frac{\eta}{2}\\
	&=|\bar{F}_{\mathbf{y}}(x)-\bar{F}_{\mathbf{y}'}(x)|+\frac{\eta}{2}
	<\eta.
\end{align*}
By (\ref{i:barFinv}) it follows that $d(h(x),h'(x))<\epsilon/5$
for every $x\in L_{12}\cup L_{23}\cup L_{31}$.
It remains to consider points $x\in S_i$, $i=1,2,3$, which can be controlled via
the behavior of the endpoints $y_i, y_i'$ of the short arc $S_i$.

First observe that the image $h(S_i)$ of $S_i$ is the short arc
defined by $h(y_i), h(y_i')$. Indeed, on one hand the two points are close:
\[
	d(h(y_i), h(y_i'))\le d(h(y_i), h'(y_i'))+d(h'(y_i'),h(y_i'))<\delta+\frac{\epsilon}{5}<\frac{2}{5}\epsilon.
\]
On the other hand, $S^1\setminus S_i$ of $S_i$ contains a point $y_j$ with
$j\in \{1,2,3\}\setminus \{i\}$; therefore $h(y_j)\notin h(S_i)$.
Since $h(\mathbf{y})\in K$ we have
\[d(h(y_i),h(y_j))\ge r>2\epsilon/5.\]
Hence $h(S_i)$ is the short arc defined by $2\epsilon/3$-close points $h(y_i), h(y_i')$,
implying
\[
	d(h(x), h(y_i))<\frac{2}{5}\epsilon\qquad (x\in S_i).
\]
Similarly, $h'(S_i)$ is the short arc defined by $2\epsilon/5$-close points $h'(y_i), h'(y_i')$,
and
\[
	d(h'(x), h'(y_i))<\frac{2}{5}\epsilon\qquad (x\in S_i).
\]
Since $y_i\in L_{ij}$, $d(h(y_i), h'(y_i))<\epsilon/5$. Therefore for any $x\in S_i$
\[
	d(h(x),h'(x))\le d(h(x), h(y_i))+d(h(y_i),h'(y_i))+d(h'(x), h'(y_i))<\epsilon.\qedhere
\]
\end{proof}

\begin{proof}[Continuation of the proof of Theorem~\ref{T:intoPGL2}]\hfill{}\\
We claim that $H_\mu<\Homeo(S^1)$ is a \emph{convergence group}, i.e., for any
compact subset $K\subset \rmTpl$ the set
\[
	H(\mu,K)=\{ h \in H_\mu \mid h^{-1} K\cap K\ne\emptyset \}
\]
is compact. In particular, the Polish group
$H_\mu$ is locally compact. Let us fix a compact subset $K\subset\rmTpl$.
Since $H(\mu,K)$ is a closed subset in the Polish group $\Homeo(S^1)$, it suffices to show that any sequence
$\{h_n\}_{n=1}^\infty$ in $H(\mu,K)$ contains a Cauchy subsequence.
Choose triples $\mathbf{y}_n\in h_n^{-1} K\cap K$.
Upon passing to a subsequence, we may assume
that the points $\mathbf y_n$ converge to some $\mathbf y\in K$ and
the points $\mathbf z_n=h_{n}(\mathbf{y}_{n})$ converge to some
$\mathbf z\in K$.
Let $\epsilon>0$. For the given $\epsilon$ and $K$
let $\delta>0$ be as in Lemma~\ref{L:Hmu}. Choose
$N\in \mathbf{N}$ be large enough to ensure that $d(\mathbf{y}_{n},\mathbf{y}_m)<\delta$ and
$d(\mathbf{z}_{n}, \mathbf{z}_m)<\delta$ for all $n,m>N$.
It follows from Lemma~\ref{L:Hmu} that $h_n$ and $h_m$ are $\epsilon$-close whenever $n,m>N$.
This proves that $H_\mu$ is a convergence group on the circle.

Finally, it follows that $H_\mu$ is conjugate to a closed subgroup of $G$.
For discrete groups this is a well known results of Gabai \cite{Gabai} and Casson -- Jungreis \cite{Casson}.
The case of non-discrete convergence group $H_\mu<\Homeo(S^1)$ can be argued more directly.
The closed convergence group $H_\mu$ is a locally compact subgroup of $\Homeo(S^1)$;
the classification of all such groups is well known, and the only ones with convergence property
are conjugate to
$\rmPGL_2(\bbR)$~\citelist{\cite{ghys}*{pp.~345--348}\cite{locally-compact}*{pp.~51--54}}.
\end{proof}

We return to the proof of Theorem~\ref{T:ME-rigidity2and3} in case of $n=2$.
Starting from an integrable $(G,H)$-coupling $(\Omega,m)$ between $G=\rmPGL_2(\bbR)$
and an unknown compactly generated unimodular group $H$ a continuous
representation $\rho:H\to \Homeo(S^1)$ with compact kernel and closed image
was constructed.
Theorem~\ref{T:intoPGL2} implies that, up to conjugation, we may assume that
\[
	\bar{H}=\rho(H)<G=\rmPGL_2(\bbR).
\]
Since $\bar{H}$ is measure equivalent to $G=\rmPGL_2(\bbR)$, it is non-amenable.

Case (1): $\bar{H}<G=\rmPGL_2(\bbR)$ is non-discrete.
(This does not occur in the original formulation of Theorem~\ref{T:ME-rigidity2and3},
but is included in the broader context of lcsc $H$ adapted in this proof).
There are only two non-discrete non-amenable closed subgroups of $G$:
the whole group $G$ and its index two subgroup $G^0=\rmPSL_2(\bbR)$.
Both of these groups may appear as $\bar{H}$; in fact,
direct products of the form $H\cong G\times K$ or $H\cong G^0\times K$
with compact $K$ and certain almost direct products
$G'\times K'/C$ as in~\cite{locally-compact}*{Theorem~A} give rise to
an integrable measure equivalence between $H$ and
$G$ (cf.~\cite{locally-compact}*{Theorem~C}).

Case (2): $\bar{H}$ is discrete. We claim that such $\bar{H}$ is a cocompact lattice
in $G$. Indeed, every finitely generated discrete non-amenable subgroup of $G$
is either cocompact or is virtually a free group $\bbF_2$.
The latter possibility is ruled out by the following.
\begin{lemma}
The free group $\bbF_2$ is not $\rmL^1$-measure equivalent to $G$.
\end{lemma}
Note that these groups are measure equivalent since $\bbF_2$ forms a lattice in $G$.
\begin{proof}
Assuming $\bbF_2$ is $\rmL^1$-measure equivalent to $G$,
one can construct an integrable measure equivalence between $G$ and the automorphism
group $H={\rm Aut}({\rm Tree}_4)$ of the $4$-regular tree, which contains $\bbF_2$
as a cocompact lattice. By Theorems~\ref{thm:taut relative homeo} and~\ref{T:split+hom}
this would yield a continuous homomorphism
$H\to \Homeo(S^1)$ with closed image.
This leads to a contradiction, because $H$ is totally disconnected
and virtually simple~\cite{tits-trees}*{Th\'eor\`eme~4.5},
while $\Homeo(S^1)$ has no non-discrete totally disconnected
subgroups~\cite{ghys}*{Theorem~4.7 on p.~345}.
\end{proof}

% subsection actions_on_the_circle (end)

% section proofs_of_the_main_results (end)
\begin{appendix}

\section{Measure equivalence} % (fold)
\label{sec:appendix_measure_equivalence}
The appendix contains some general facts related to measure equivalence (Definition~\ref{D:ME}),
the strong ICC property (Definition~\ref{def:strongy ICC definition}),
and the notions of taut couplings and groups (Definition~\ref{D:M-rigidity}).

\subsection{The category of couplings}
\label{subs:composition}\hfill{}\\
Measure equivalence is an equivalence relation on unimodular lcsc groups.
Let us describe explicitly the constructions which show
reflexivity, symmetry and transitivity of measure equivalence.

\subsubsection{\textbf{Tautological coupling}} % (fold)
\label{ssub:tautological_coupling}
The tautological coupling is the $(G\times G)$-coupling
$(G,m_G)$ given by
$(g_1,g_2):g\mapsto g_1 g g_2^{-1}$. It demonstrates reflexivity of measure equivalence.
\subsubsection{\textbf{Duality}} % (fold)
\label{ssub:duality}
Symmetry is implied by the following:
Given a $(G,H)$-coupling $(\Omega,m)$ the dual $(\dual{\Omega},\dual{m})$ is the $(H,G)$-coupling $\dual{\Omega}$ with the same underlying measure space $(\Omega,m)$ and
the $H\times G$-action $(h,g):\dual{\omega}\mapsto (g,h)\dual{\omega}$.

\subsubsection{\textbf{Composition of couplings}} % (fold)
\label{ssub:composition_of_couplings}
Compositions defined below shows that measure equivalence is a transitive relation.
Let $G_1,H, G_2$ be unimodular lcsc groups, and $(\Omega_i,m_i)$ be a
$(G_i,H)$-coupling for $i\in\{1,2\}$. We describe the $(G_1,G_2)$-coupling 
$\Omega_1\times_H\dual{\Omega}_2$ modeled on
the space of $H$-orbits on $(\Omega_1\times\Omega_2,m_1\times m_2)$ 
with respect to the diagonal $H$-action. Consider measure isomorphisms for $(\Omega_i,mi)$
as in~\eqref{e:ij-fd}: For $i\in\{1,2\}$ there are
finite measure spaces $(X_i,\mu_i)$ and $(Y_i,\nu_i)$,
measure-preserving actions $G_i\acts (X_i,\mu_i)$ and $H\acts (Y_i,\nu_i)$,
measurable cocycles $\alpha_i:G_i\times X_i\to H$ and $\beta_i:H\times Y_i\to G_i$,
and measure space isomorphisms $G_i\times Y_i\cong \Omega_i\cong H\times X_i$
with respect to which the $G_i\times H$-actions are given by
\begin{align*}
	(g_i,h)\ :\ & (h',x)\mapsto (h h' \alpha_i(g_i,x)^{-1},\, g_i.x),\\
	(g_i,h)\ :\ & (g',y)\mapsto (g_i g' \beta(h,y)^{-1},\,h.y).
\end{align*}
The space $\Omega_1\times_H\dual{\Omega}_2$ with its natural $G_1\times G_2$-action
is equivariantly isomorphic to $(X_1\times X_2\times H,\mu_1\times\mu_2\times m_H)$
endowed with the $G_1\times G_2$-action
\[
	(g_1,g_2): (x_1,x_2,h)\mapsto (g_1.x_1,\, g_2.x_2,\, \alpha_1(g_1,x_1) h \alpha_2(g_2,x_2)^{-1}).
\]
To see that it is a $(G_1,G_2)$-coupling, we identify this space with $Z\times G_1$
equipped with the action
\[
	(g_1,g_2): (g',z)\mapsto (g_1 g' c(g_2,z)^{-1}, g_2.z)\qquad(g'\in G_1,\ z\in Z)
\]
where $Z=X_2\times Y_1$, while the action $G_2\acts Z$ and the cocycle $c:G_2\times Z\to G_1$
are given by
\begin{equation}\label{e:composition-coc}
		\begin{aligned}
		&g_2:(x,y)\mapsto (g_2.x,\alpha_2(g_2,x).y),\\
		&c(g_2,(x,y))=\beta_1(\alpha_2(g_2,x), y).
	\end{aligned}
\end{equation}
% Then $\Omega_1\times_H\dual{\Omega}_2\cong Z_2\times G_1$ with
% \[
% 	(g_1,g_2): (g,z)\mapsto (g_1 g c_2(g_2,z)^{-1},\, g_2.z).
% \]
Similarly, $\Omega_1\times_H\dual{\Omega}_2\cong W\times G_2$, for $W\cong X_1\times Y_2$.
% subsubsection composition_of_couplings (end)

\subsubsection{\textbf{Morphisms}} % (fold)
\label{ssub:morphisms}
Let $(\Omega_i,m_i)$, $i\in\{1,2\}$, be two $(G,H)$-couplings.
Let $F:\Omega_1\to\Omega_2$ be a measurable map such that for $m_1$-a.e. 
$\omega\in\Omega_1$ and every $g\in G$ and every $h\in H$
\[
	F((g,h)\omega)=(g,h)F(\omega).
\]
Such maps are called \emph{quotient maps} or \emph{morphisms}.

% We shall need the following properties for such maps.
%
% Recall that $G$ acts on the finite measure space $(X_i,\mu_i)$ of $H$-orbits in $(\Omega_i,m_i)$,
% and $H$ acts on $(Y_i,\nu_i)\cong (\Omega_i,m_i)/G$.
% We claim that $F:\Omega_1\to\Omega_2$ defines equivariant quotient maps
% \begin{align*}
% 	&p:X_1\to X_2,\qquad &p_*\mu_1=\mu_2,\qquad &p(g.x)=g.p(x)\qquad &(g\in G),\\
% 	&q:Y_1\to Y_2,\qquad &q_*\nu_1=\nu_2,\qquad &q(h.y)=h.q(y)\qquad &(h\in H).
% \end{align*}
% Indeed, let
% \begin{align*}
% 	j_i&:(G,m_G)\times (Y_i,\nu_i)\to (\Omega_i,m_i)\qquad &(i=1,2),\\
% 	k_i&:(H,m_H)\times (X_i,\mu_i)\to (\Omega_i,m_i)\qquad &(i=1,2)
% \end{align*}
% be measure space isomorphisms as in (\ref{e:ij-fd}).
% Indeed, note that the measurable map $Q:G\times Y_1\to G\times Y_2$ given by
% \[
% 	i_2\circ Q(g,y_1)=(g^{-1},e) F\circ i_1(g,y_1)
% \]
% is essentially independent of $g\in G$, and therefore $Q$ has the form
% $Q(g,y_1)=(g_{y_1}, q(y_1))$ where $q:Y_1\to Y_2$ and $Y_1\to G$, $y_1\mapsto g_{y_1}$
% ares some measurable maps.
% Define $i_1':G\times Y_1\to \Omega$ by $i_1'(g,y_1)=i_1(gg_{y_1},y_1)$.
% Then
% \[
% 	F\circ i_1'(g,y_1)=i_2(g,q(y_1))
% \]
% and therefore $q:Y_1\to Y_2$ is an $H$-equivariant quotient map as claimed.
% Similarly, one adjusts $j_1$ and obtains a $G$-equivariant map $p:X_1\to X_2$.
% % subsubsection morphisms (end)

\subsubsection{\textbf{Compact kernels}} % (fold)
\label{ssub:compact kernels}
Let $(\Omega,m)$ be a $(G,H)$-coupling, and let
\[\{1\}\to K\to G\to \bar{G}\to\{1\}\]
be a short exact sequence where $K$ is compact. Then the
natural quotient space $(\bar\Omega,\bar{m})=(\Omega,m)/K$
is a $(\bar{G},H)$-coupling, and the natural map $F:\Omega\to\bar{\Omega}$,
$F:\omega\mapsto K\omega$, is equivariant in the sense of $F((g,h)\omega)=(\bar{g},h) F(\omega)$.
This may be considered as an \emph{isomorphism of couplings up to compact kernel}.
% subsubsection compact kernels (end)

\subsubsection{\textbf{Passage to lattices}} % (fold)
\label{ssub:passage to lattices}
Let $(\Omega,m)$ be a $(G,H)$-coupling, and let $\Gamma<G$ be a lattice.
By restricting the $G\times H$-action on $(\Omega,m)$ to $\Gamma\times H$
we obtain a $(\Gamma,H)$-coupling.
Formally, this follows by considering $(G,m_G)$ as a $\Gamma\times G$-coupling
and considering the composition $G\times_G \Omega$ as $\Omega$ with the
$\Gamma\times H$-action.
% subsubsection passage to lattices (end)

\subsection{$\rmL^p$-integrability conditions}
\label{sub:Lp-integrability} % (fold)
Let $G$ and $H$ be compactly generated unimodular lcsc groups
equipped with proper norms $|\cdot |_G$ and $|\cdot |_H$.
Let $c:G\times X\to H$ be a measurable cocycle,
and fix some $p\in [1,\infty)$.
For $g\in G$ we define
\[
	\|g\|_{c,p}=\Bigl(\int_X |c(g,x)|^p_H\,d\mu(x)\Bigr)^{1/p}.
\]
For $p=\infty$ we use the essential supremum.
Assume that $\|g\|_{c,p}<\infty$ for a.e.~$g\in G$.
We claim that there are constants $a,A>0$ so that for every $g\in G$
\begin{equation}\label{e:sublinear}
	\|g\|_{c,p}\le A\cdot |g|_G+a.
\end{equation}
Hence $c$ is $\rmL^p$-integrable in the sense of Definition~\ref{D:Lp-ME}.
The key observation here is that $\|-\|_{c,p}$ is subadditive. Indeed, by the cocycle identity,
subadditivity of the norm $|-|_H$, and the Minkowski inequality, for any $g_1,g_2\in G$ we get
\begin{align*}
	\|g_2 g_1\|_{c,p}&\le\Bigl(\int_X
		\bigl(|c(g_2,g_1. x)|_H+
		|c(g_1, x)|_H\bigr)^p\,d\mu(x)
		\Bigr)^{1/p}\\
	&\le \Bigl(\int_X |c(g_2,-)|^p_H\,d\mu\Bigr)^{1/p}
	+\Bigl(\int_X |c(g_1,-)|^p_H\,d\mu\Bigr)^{1/p}\\
	&= \|g_2\|_{c,p}+\|g_1\|_{c,p}.
\end{align*}
For $t>0$ denote $E_t=\{g\in G : \|g\|_{c,p}<t\}$.
We have $E_t\cdot E_s\subseteq E_{s+t}$ for any $t,s>0$.
Fix $t$ large enough so that $m_G(E_t)>0$.
By~\cite{doran+fell}*{Corollary~12.4 on p.~235},
$E_{2t}\supseteq E_t\cdot E_t$ has a non-empty interior. Hence any compact subset of $G$
can be covered by finitely many translates of $E_{2t}$.
The subadditivity implies that $\|g\|_{c,p}$ is bounded on compact sets.
This gives (\ref{e:sublinear}).

\begin{lemma}\label{L:composition-of-lp-coc}
	Let $G$,$H$,$L$ be compactly generated groups,
	$G\acts (X,\mu)$, $H\acts (Y,\nu)$ be finite measure-preserving actions,
	and $\alpha:G\times X\to H$ and $\beta:H\times Y\to L$ be $\rmL^p$-integrable
	cocycles for some $1\le p\le \infty$.
	Consider $Z=X\times Y$ and $G\acts Z$ by $g:(x,y)\mapsto (g.x, \alpha(g,x).y)$.
	Then the cocycle $\gamma:G\times Z\to L$ given by
	\[
		\gamma(g,(x,y))=\beta(\alpha(g,x),y).
	\]
	is $\rmL^p$-integrable.
\end{lemma}
\begin{proof}
For $p=\infty$ the claim is obvious. Assume $p<\infty$.
Let $A,a,B,b$ be constants such that $\|h\|_{\beta,p}\le B\cdot |h|_H+b$ and
$\|g\|_{\alpha,p}\le A\cdot |g|_G+a$. Then
\begin{align*}	
	\|g\|_{\gamma,p}^p &= \int_{X\times Y} |\beta(\alpha(g,x),y)|_\rmL^p\,d\mu(x)\,d\nu(y)\\
		&\le  \int_X (B\cdot |\alpha(g,x)|_H+b)^p\,d\mu(x)\\
		&\le \max(B,b)^p\cdot \|g\|_{\alpha,p}^p\le (C\cdot |g|_G+c)^p
\end{align*}
for appropriate constants $c>0$ and $C>0$.
\end{proof}

\medskip

\begin{lemma}\label{lem:composition of lp couplings}
Let $G_1,H,G_2$ be compactly generated unimodular lcsc groups. For $i\in\{1,2\}$ let
$(\Omega_i,m_i)$ be an $\rmL^p$-integrable
$(G_i,H)$-coupling. Then $\Omega_1\times_H \dual{\Omega}_2$ is an $\rmL^p$-integrable
$(G_1,G_2)$-coupling.	
\end{lemma}
\begin{proof}
This follows from Lemma~\ref{L:composition-of-lp-coc} using the explicit description
(\ref{e:composition-coc}) of the cocycles for $\Omega_1\times_H \dual{\Omega}_2$.
\end{proof}

We conclude that for each $1\le p\le \infty$, $\rmL^p$-measure equivalence is an equivalence relation
between compactly generated unimodular lcsc groups.

\medskip

\subsection{Tautening maps} % (fold)
\label{sub:equivariant_tautening_maps}

% Retain the setting of Definition~\ref{D:tautrel}.
% Often one is able to construct a measurable map $\Phi:\Omega\to\calG$ for which
% equivariance holds only for $m\times m_G\times m_G$-a.e. $\omega,g_1,g_2$.
% Let us call such maps \emph{nearly tautening} relative to $\pi:G\to\calG$.
%
% \begin{lemma}
% If $\Phi:\Omega\to\calG$ be a nearly tautening map relative to $\pi:G\to\calG$,
% then there exists a tautening map $\Phi_0:\Omega\to \calG$ so that
% $m\{ \omega\mid \Phi(\omega)\ne\Phi_0(\omega)\}=0$.
% \end{lemma}
% This follows from \cite{zimmer-book}*{Appendix B}.
%
% \medskip

\begin{lemma}\label{L:taut-MR}
Let $G$ be a lcsc group, $\Gamma$ a countable group and
$j_1,j_2:\Gamma\to G$ be homomorphisms with $\Gamma_i=j_i(\Gamma)$ being lattices
in $G$.
Assume that $G$ is taut (resp. $p$-taut and $\Gamma_i$ are $\rmL^p$-integrable).
Then there exists $g\in G$ so that
\[
	j_2(\gamma)=g\,j_1(\gamma)\,g^{-1}\qquad (\gamma\in\Gamma).
\]
If $\pi:G\to\calG$ is a continuous homomorphism into a Polish group
and $G$ is taut relative to $\pi:G\to\calG$ (resp. $G$ is $p$-taut relative to $\pi:G\to\calG$
and $\Gamma_i$ are $\rmL^p$-integrable) then there exists $y\in \calG$ with
\[
	\pi(j_2(\gamma))=y \pi(j_1(\gamma))y^{-1}\qquad (\gamma\in\Gamma).
\]
\end{lemma}
\begin{proof}
We prove the more general second statement. The group
\[
	\Delta=\{(j_1(\gamma),j_2(\gamma))\in G\times G \mid \gamma\in \Gamma\}
\]
is a closed discrete subgroup in $G\times G$.
The $G\times G$-space $\Omega=(G\times G)/\Delta$ equipped
with the $G\times G$-invariant measure is easily seen to be a $(G,G)$-coupling.
It will be $\rmL^p$-integrable if $\Gamma_1$ and $\Gamma_2$ are $\rmL^p$-integrable
lattices.
Let $\Phi:\Omega\to \calG$ be the tautening map.
There are $a,b\in G$ and $x\in\calG$ such that for all $g_1,g_2\in G$
\[
	\Phi((g_1a,g_2b)\Delta_f)=\pi(g_1) x \pi(g_2)^{-1}.
\]
Since $(a,b)$ and $(j_1(\gamma)^a \,a,\,j_2(\gamma)^b\,b)$ are in the same
$\Delta$-coset, where $g^h=hgh^{-1}$, we get for all $g_1,g_2\in G$ and every $\gamma\in \Gamma$
\[
	\pi(g_1) x \pi(g_2)^{-1}=\pi(g_1) \pi(j_1(\gamma)^a) x\pi(j_2(\gamma)^b)^{-1} \pi(g_2)^{-1}.
\]
This implies that $j_1$ and $j_2$ are conjugate homomorphisms.
\end{proof}

The following lemma relates tautening maps $\Phi:\Omega\to G$
and cocycle rigidity for ME-cocycles.

\begin{lemma}\label{L:coc-taut}
	Let $G$ be a unimodular lcsc group, $\calG$ be a Polish group, $\pi:G\to\calG$ a
	continuous homomorphism.
	Let $(\Omega,m)$ be a $(G,G)$-coupling and $\alpha:G\times X\to G$, $\beta:G\times Y\to G$
	be the corresponding ME-cocycles.
Then there is a tautening map $\Omega\to \calG$
	iff the $\calG$-valued cocycle $\pi\circ\alpha$ is conjugate to $\pi$, that is,
	\[
		\pi\circ\alpha(g,x)=f(g.x)^{-1} \pi(g) f(x).
	\]
Moreover, $\Omega$ is taut relative to $\pi$
	if such a measurable map $f:X\to \calG$ is unique.
	This is also equivalent to $\pi\circ\beta$ being uniquely conjugate to $\pi:G\to\calG$.
\end{lemma}

\begin{proof}
Let $\alpha:G\times X\to G$
be the ME-cocycle associated to a measure space isomorphism
$i:(G,m_G)\times (X,\mu)\to (\Omega,m)$
as in (\ref{e:ij-fd}). In particular,
\[
	(g_1,g_2): i(g,x)\mapsto i(g_2 g\alpha(g_1,x)^{-1},\,g_1.x).
\]
We shall now establish a $1$-to-$1$ correspondence between Borel maps $f:X\to \calG$
with
\[
	\pi\circ\alpha(g,x)=f(g.x)^{-1}\pi(g) f(x)
\]
and tautening maps $\Phi:\Omega\to\calG$. Given $f$ as above one verifies that
\[
	\Phi:\Omega\to \calG,\qquad \Phi(i(g,x))=f(x)\pi(g)^{-1}
\]
is $G\times G$-equivariant.

For the converse direction, suppose $\Phi:\Omega\to G$ is a tautening map. Thus,
\[
	g_1 \Phi(g_0,x) g_2^{-1}=\Phi((g_1,g_2)(g_0,x))=\Phi(g_2 g_0 \alpha(g_1,x)^{-1},\, g_1.x).
\]
For $\mu$-a.e.~$x\in X$ and a.e.~$g\in G$
the value of $\Phi(g,x)g$ is constant $f(x)$. If we substitute $g_0=g_1=g$ and $g_2=g\alpha(g,x)g^{-1}$ in the above identity, then we obtain $\alpha(g,x)=f(g.x)^{-1}g f(x)$.
\end{proof}

% Let us formally state a fact that is used in Section~\ref{sec:measure_equivalence}.
%
% \begin{lemma}\label{lem:formal lemma}
% Let $G$ be a unimodular lcsc group, $\pi:G\to\calG$ a continuous homomorphism
% into a Polish group with compact kernel $K$ and closed image $\bar{G}=\pi(G)$.
% Let $(\Omega,m)$ be a taut $(G,G)$-coupling with $\Phi:\Omega\to \calG$.
% Then there exist a Borel measure $\bar{m}$ on $\calG$ so that
% $(\calG,\bar{m})$ is a $(\bar{G},\bar{G})$-coupling
% and
% \[
% 	\Phi:(\Omega,m)\ \overto{}\ (\Omega,m)/K\times K\ \overto{F}\ (\calG,\bar{m})
% \]
% is a quotient by a compact kernel, followed by a morphism $F$
% of $(\bar{G},\bar{G})$-couplings.
% \end{lemma}

\begin{lemma} \label{lem:ergodic dec}
	Let $G$ be a unimodular lcsc group, $\calG$ be a Polish group, $\pi:G\to\calG$ a
	continuous homomorphism.
Then $G$ is ($p$-)taut, that is every ($p$-integrable) $(G,G)$-coupling is taut relative to $\pi$, iff
every ergodic ($p$-integrable) $(G,G)$-coupling is taut relative to $\pi$.
\end{lemma}

\begin{proof}
We give the proof in the $p$-integrable case, the case without the integrability condition being simpler.
We assume that
every ergodic $p$-integrable $(G,G)$-coupling is taut relative to $\pi$
and let $(\Omega,m)$ be an arbitrary $p$-integrable coupling.
We fix a fundamental domain $(X,\mu)$ for the second $G$ action such that the associated
cocycle $\alpha:G\times X\to G$ is $p$-integrable.

Let $\mu=\int\mu_t d\eta(t)$ be the $G$-ergodic decomposition of $(X,\mu)$.
By
\cite{FurmanME}*{Lemma~2.2}
%the $G$-ergodic decomposition of $(X,\mu)$ corresponds
it corresponds to the $(G\times G)$-ergodic
decomposition of $(\Omega,m)$ into ergodic couplings $(\Omega,m_t)$,
so that
% $(\Omega,m_t)$
%where
$m=\int m_td\eta(t)$.
%for some probability measure $\eta$.

%prove~\eqref{eq:conjugation} on a.e.~ergodic component of $X$, each of which corresponds
%to an ergodic component in the ergodic

%Let
%$\alpha\colon\Gamma_r\times X\to\Gamma_l$ be an integrable ME-cocycle associated to
%a $\Gamma_l$-fundamental domain $X\subset\Omega$.
Let $\abs{\_}\colon G\to\bbN$
be the length function associated to some word-metric on $G$. The
$p$-integrability of $\alpha$ means that 
\[
	\int\int_X\abs{\alpha(g,x)}^pd\mu_t(x)d\eta(t)=\int_X\abs{\alpha(g,x)}^pd\mu(x)<\infty
\]
for every $g\in G$. 
By Fubini's theorem $\int_X\abs{\alpha(g,x)}^pd\mu_t(x)<\infty$ for $\eta$-a.e.~$t$.
Hence $(\Omega,m_t)$ is $p$-integrable for $\eta$-a.e.~$m_t$, and in particular it is taut
relative to $\pi$, by our assumption.
It follows by Lemma~\ref{L:coc-taut} that
% there
%is an essentially unique measurable map $f\colon X\to\calG$ such that a.e.
%\begin{equation}\label{eq:conjugation}
%	\pi\circ\alpha(g,x)=f(g.x)\pi(g) f(x)^{-1},
%\end{equation}
%that is,
the cocycle $\pi\circ\alpha$ is conjugate to the constant cocycle $\pi$ over $\eta$-a.e $(X,\mu_t)$.
Then by~\cite{fisher-nonergodic}*{Corollary~3.6}\footnote{The target $\calG$ is assumed to be locally compact in this reference but the proof therein works the same for a Polish group $\calG$.}
$\pi\circ\alpha$ is conjugate to $\pi$ over $(X,\mu)$.
Again, by Lemma~\ref{L:coc-taut} we conclude that $(\Omega,m)$ is taut relative to $\pi$.
\end{proof}

\subsection{Strong ICC property} % (fold)
\label{sub:strong_icc_property}

\begin{lemma} \label{lem:ICC-lattices}
	Let $G$ be a unimodular lcsc group, $\calG$ a Polish group, $\pi:G\to\calG$ a continuous homomorphism.
Let $\Gamma<G$ be a lattice. Then
$\calG$
	is strongly ICC relative to $\pi(G)$
if and only if it
	is strongly ICC relative to $\pi(\Gamma)$.
\end{lemma}

\begin{proof}
Clearly if $\calG$
	is strongly ICC relative to $\pi(\Gamma)$ then it is also
strongly ICC relative to $\pi(G)$.
The reverse implication follows by averaging a $\pi(\Gamma)$ invariant measure over $G/\Gamma$.
\end{proof}

\begin{lemma}\label{L:uniq2sICC}
Let $G$ be a unimodular lcsc group, $\calG$ a Polish group, $\pi:G\to\calG$ a continuous homomorphism.
Suppose that $\calG$ is not strongly ICC relative to $\pi(G)$.
Then there is a $(G,G)$-coupling $(\Omega,m)$ with two distinct
tautening maps to $\calG$.
\end{lemma}
\begin{proof}
	Let $\mu$ be a Borel probability measure on $\calG$
	invariant under conjugations by $\pi(G)$.
	Consider $\Omega=G\times \calG$ with the measure $m=m_G\times\mu$ where $m_G$
	denotes the Haar measure, and measure-preserving $G\times G$-action
	\[
		(g_1,g_2): (g,x)\mapsto (g_1 g g_2^{-1},\, \pi(g_2)\, x\, \pi(g_2)^{-1}).
	\]
	This is clearly a $(G,G)$-coupling and the following measurable maps
	$\Phi_i:\Omega\to \calG$, $i\in \{1,2\}$, are $G\times G$-equivariant:
	$\Phi_1(g,x)=\pi(g)$ and $\Phi_2(g,x)=\pi(g)\cdot x$.
	Note that $\Phi_1=\Phi_2$ on a conull set iff $\mu=\delta_e$.
\end{proof}

\begin{lemma}\label{L:sICC2uniq}	
Let $G$ be a unimodular lcsc group and $\calG$ a Polish group. Assume that
$\calG$ is strongly ICC relative to $\pi(G)$. Let $(\Omega,m)$ be a $(G,G)$-coupling. Then:
\begin{enumerate}
\item \label{i:unique}
	There is at most one tautening map $\Phi\colon\Omega\to \calG$.
\item \label{i:descent}
	Let $F:(\Omega,m)\to (\Omega_0,m_0)$ be a morphism of $(G,G)$-couplings
	and suppose that there exists a tautening map $\Phi:\Omega\to\calG$.
	Then it descends to $\Omega_0$, i.e., $\Phi=\Phi_0\circ F$
	for a unique tautening map $\Phi_0:\Omega_0\to\calG$.
\item \label{i:2lattice}
	If $\Gamma_1,\Gamma_2<G$ are lattices, then $\Phi:\Omega\to \calG$ is unique as a
	$\Gamma_1\times\Gamma_2$-equivariant map.
\item \label{i:extn}
	If $\Gamma_1,\Gamma_2<G$ are lattices, and $(\Omega,m)$ admits
	a $\Gamma_1\times\Gamma_2$-equivariant map $\Phi:\Omega\to\calG$, then $\Phi$ is $G\times G$-equivariant.
\item \label{i:Dirac}
	If $\eta:\Omega\to \Prob(\calG)$,
	$\omega\mapsto \eta_\omega$, is a measurable $G\times G$-equivariant map to the space
	of Borel probability measures on $\calG$ endowed with the weak topology,
	then it takes values in Dirac measures: We have $\eta_\omega=\delta_{\Phi(\omega)}$,
	where $\Phi:\Omega\to \calG$ is the unique tautening map.
\end{enumerate}
\end{lemma}
\begin{proof}
We start from the last claim and deduce the other ones from it.

\noindent{(\ref{i:Dirac})}.
Given an equivariant map $\eta:\Omega\to \Prob(\calG)$	consider the convolution
\[
	\nu_\omega=\check\eta_\omega * \eta_\omega,
\]
namely the image of $\eta_\omega\times\eta_\omega$ under the map $(a,b)\mapsto a^{-1}\cdot b$.
Then
\[
	\nu_{(g,h)\omega}=\nu_\omega^{\pi(g)}\qquad (g,h\in G),
\]
where the latter denotes the push-forward of $\nu_\omega$ under the conjugation
\[
	a\mapsto a^{\pi(g)}=\pi(g)^{-1}a \pi(g).
\]
In particular, the map $\omega\mapsto \nu_\omega$ is invariant under the action of the second $G$-factor.
Therefore $\nu_\omega$ descends to a measurable map $\tilde\nu:\Omega/G\to \Prob(\calG)$, satisfying
\[
	\tilde\nu_{g.x}=\tilde\nu_x^{\pi(g)}\qquad (x\in X=\Omega/G,\ g\in G).
\]
Here we identify $\Omega/G$ with a finite measure space $(X,\mu)$ as in~\eqref{e:ij-fd}.
Consider the center of mass
\[
	\bar\nu=\frac{1}{\mu(X)}\int_{X}\tilde\nu_x \,d\mu(x).
\]
It is a probability measure on $\calG$, which is invariant under conjugations.
By the strong ICC property relative to $\pi(G)$ we get $\bar\nu=\delta_e$.
Since $\delta_e$ is an extremal point of $\Prob(\calG)$,
it follows that $m$-a.e. $\nu(\omega)=\delta_e$.
This implies that $\eta_\omega=\delta_{\Phi(\omega)}$ for
some measurable $\Phi:\Omega\to \calG$.
The latter is automatically $G\times G$-equivariant.
\smallskip

\noindent{(\ref{i:unique})}. If $\Phi_1,\Phi_2:\Omega\to \calG$ are tautening maps,
then $\eta_\omega=\frac{1}{2}(\delta_{\Phi_1(\omega)}+\delta_{\Phi_2(\omega)})$
is an equivariant map $\Omega\to\Prob(\calG)$.
By (\ref{i:Dirac}) it takes values in Dirac measures,
which is equivalent to the $m$-a.e. equality $\Phi_1=\Phi_2$.

\smallskip

\noindent{(\ref{i:descent})}.
Disintegration of $m$ with respect to $m_0$ gives a $G\times G$-equivariant measurable
map $\Omega_0\to\mathcal{M}(\Omega)$, $\omega\mapsto m_{\omega_0}$, to the space
of finite measures on $\Omega$.
Then the map $\eta:\Omega_0\to\Prob(\calG)$, given by
\[
	\eta_{\omega_0}=\|m_{\omega_0}\|^{-1}\cdot \Phi_* (m_{\omega_0}).
\]
is $G\times G$-equivariant. Hence by (\ref{i:Dirac}), $\eta_{\omega_0}=\delta_{\Phi_0(\omega_0)}$
for the unique tautening map $\Phi_0:\Omega_0\to\calG$.
The relation $\Phi=\Phi_0\circ F$ follows from the fact that Dirac measures are extremal.

\smallskip

\noindent{(\ref{i:2lattice})} follows from (\ref{i:extn}) and (\ref{i:unique}).

\smallskip

\noindent{(\ref{i:extn})}.
The claim is equivalent to:
For $m$-a.e. $\omega\in\Omega$ the map $F_\omega\colon G\times G\to \calG$ with
\[
	F_\omega(g_1,g_2)=\pi(g_1)^{-1}\,\Phi((g_1,g_2)\omega)\,\pi(g_2)
\]
is $m_G\times m_G$-a.e. constant $\Phi_0(\omega)$.
Note that the family $\{F_\omega\}$ has the following equivariance property:
For $g_1,g_2,h_1,h_2\in G$ we have
\begin{eqnarray*}
	F_{(h_1,h_2)\omega}(g_1,g_2)&=&\pi(g_1)^{-1}\Phi((g_1h_1,g_2h_2)\omega)\pi(g_2)\\
	&=&\pi(h)_1^{-1} F_\omega(g_1h_1, g_2h_2) \pi(h_2).
\end{eqnarray*}
Since $\Phi$ is $\Gamma_1\times\Gamma_2$-equivariant, for $m$-a.e. $\omega\in\Omega$ the map
$F_\omega$ descends to $G/\Gamma_1\times G/\Gamma_2$.
Let $\eta_\omega\in\Prob(\calG)$ denote the distribution of $F_\omega(\cdot,\cdot)$
over the probability space $G/\Gamma_1\times G/\Gamma_2$, that is, for
a Borel subset $E\subset\calG$
\[
	\eta_\omega(E)=m_{G/\Gamma_1}\times m_{G/\Gamma_2}\{ (g_1,g_2) \mid F_\omega(g_1,g_2)\in E\}.
\]
Since this measure is invariant under translations by $G\times G$, it follows
that $\eta_\omega$ is a $G\times G$-equivariant maps $\Omega\to \Prob(\calG)$.
By (\ref{i:Dirac}) one has $\eta_\omega=\delta_{f(\omega)}$ for some measurable $G\times G$-equivariant map
$f:\Omega\to \calG$.
Hence $F_\omega(g_1,g_2)=f(\omega)$ for a.e. $g_1,g_2\in G$; it follows that
\begin{equation}\label{e:phi-g1g2}
	\Phi((g_1,g_2)\omega)=\pi(g_1) \Phi(\omega) \pi(g_2)^{-1}
\end{equation}
holds for $m_G\times m_G\times m$-a.e. $(g_1,g_2,\omega)$.
\end{proof}

\begin{corollary}
	Let $\pi:G\to\calG$ be as above and assume that $\calG$
	is strongly ICC relative to $\pi(G)$. Then the collection of all
	$(G,G)$-couplings which are taut relative to $\pi:G\to\calG$
	is closed under the operations of taking the dual, compositions, quotients and extensions.
\end{corollary}
\begin{proof}
The uniqueness of tautening maps follow from the relative strong ICC property (Lemma~\ref{L:sICC2uniq}.(\ref{i:unique})).
Hence we focus on the existence of such maps.
 	
Let $\Phi:\Omega\to\calG$ be a tautening map. Then $\Psi(\dual{\omega})=\Phi(\omega)^{-1}$
is a tautening map 	$\dual{\Omega}\to\calG$.

Let $\Phi_i:\Omega_i\to\calG$, $i=1,2$, be tautening maps. Then $\Psi([\omega_1,\omega_2])=\Phi(\omega_1)\cdot\Phi(\omega_2)$
is a tautening map 	$\Omega_1\times_G\Omega_2\to\calG$.

If $F:(\Omega_1,m_1)\to(\Omega_2,m_2)$ is a quotient map and $\Phi_1:\Omega_1\to\calG$ is a tautening map,
then, by Lemma~\ref{L:sICC2uniq}.(\ref{i:descent}), $\Phi_1$ factors as $\Phi_1=\Phi_2\circ F$ for a tautening
map $\Phi_2:\Omega_2\to\calG$.
On the other hand, given a tautening map $\Phi_2:\Omega_2\to\calG$, the map $\Phi_1=\Phi_2\circ F$
is tautening for $\Omega_1$.
\end{proof}

% subsection strong_icc_property (end)

% section appendix_measure_equivalence (end){Generalities on Measure Equivalence}
\section{Bounded cohomology} % (fold)
\label{sec:cohomological tools}

Our background references for bounded cohomology, especially for the
functorial approach to it, are~\cites{burger+monod, monod-book}.
We summarize what we need from Burger-Monod's theory of bounded
cohomology. Since we restrict to discrete groups, some results we
quote from this theory already go back to Ivanov~\cite{ivanov}.

\subsection{Banach modules} % (fold)
\label{sub:banach_modules}

All Banach spaces are over the field $\bbR$ of real numbers. By the
dual of a Banach space we understand the normed topological dual. The
dual of a Banach space $E$ is denoted by $E^\ast$.
Let $\Gamma$ be a discrete and countable group. A \emph{Banach
  $\Gamma$-module} is a Banach space $E$ endowed with a group
homomorphism $\pi$ from $\Gamma$ into the group of isometric linear
automorphisms of $E$.  We use the module notation $\gamma\cdot
e=\pi(\gamma)(e)$ or just $\gamma e=\pi(\gamma)(e)$ for
$\gamma\in\Gamma$ and $e\in E$ whenever the action is clear from the
context.  The submodule of $\Gamma$-invariant elements is denoted by
$E^\Gamma$.  Note that $E^\Gamma\subset E$ is closed.

If $E$ and $F$ are Banach $\Gamma$-modules, a \emph{$\Gamma$-morphism}
$E\rightarrow F$ is a $\Gamma$-equivariant continuous linear map. The
space $\calB(E,F)$ of continuous, linear maps $E\rightarrow F$ is
endowed with a natural Banach $\Gamma$-module structure via
\begin{equation}\label{eq:G action on space of bounded maps}
  (\gamma\cdot f)(e)=\gamma f(\gamma^{-1}e).
\end{equation}
The \emph{contragredient} Banach $\Gamma$-module structure $E^\sharp$
associated to $E$ is by definition $\calB(E,\bbR)=E^\ast$ with the
$\Gamma$-action~\eqref{eq:G action on space of bounded maps}.  A
\emph{coefficient $\Gamma$-module} is a Banach $\Gamma$-module $E$
contragredient to some separable continuous Banach $\Gamma$-module
denoted by $E^\flat$. The choice of $E^\flat$ is part of the data. The
specific choice of $E^\flat$ defines a \mbox{weak-$\ast$} topology on
$E$. The only examples that appear in this paper are
$E=\rmL^\infty(X,\mu)$ with $E^\flat=L^1(X,\mu)$ and $E=E^\flat=\bbR$.

For a coefficient $\Gamma$-module $E$ let $\rmCb^k(\Gamma,E)$ be the
Banach $\Gamma$-module $\rmL^\infty(\Gamma^{k+1},E)$ consisting of
bounded maps from $\Gamma^{k+1}$ to $E$ endowed with the supremum norm
and the $\Gamma$-action:
\begin{equation}\label{eq:action on cochains}
  (\gamma\cdot f)(\gamma_0,\dots,\gamma_k)=\gamma\cdot f(\gamma^{-1}\gamma_0,\dots,\gamma^{-1}\gamma_k).
\end{equation}

For a coefficient $\Gamma$-module $E$ and a standard Borel $\Gamma$-space $S$
with quasi-invariant measure let
$\rmLweak^\infty(S, E)$ be the space of \mbox{weak-$\ast$}-measurable
essentially bounded maps from $S$ to $E$, where maps are identified if
they only differ on a null set.  The space $\rmLweak^\infty(S, E)$ is
endowed with the essential supremum norm and the
$\Gamma$-action~\eqref{eq:action on cochains}. For a measurable space
$X$ the Banach space $\calB^\infty(X, E)$ is the space of
\mbox{weak-$\ast$}-measurable bounded maps from $X$ to $E$ endowed
with supremum norm~\cite{burger+iozzi}*{Section~2} and the
$\Gamma$-action~\eqref{eq:action on cochains}.

% subsection Banach modules (end)

\subsection{Injective resolutions} % (fold)
\label{sub:bounded_cohomology}

Let $\Gamma$ be a discrete group and $E$ be a Banach $\Gamma$-module.
The sequence of Banach $\Gamma$-modules $\rmCb^k(\Gamma, E)$, $k\ge
0$, becomes a chain complex of Banach $\Gamma$-modules via the
standard homogeneous coboundary operator
\begin{equation}\label{eq:homogeneous differential}
  d(f)(\gamma_0,\dots,\gamma_k)=\sum_{i\ge 0}^k(-1)^i
  f(\gamma_0,\dots, \hat{\gamma_i},\dots,\gamma_k).
\end{equation}

The \emph{bounded cohomology} $\rmHb^\bullet(\Gamma,E)$ of $\Gamma$ with
coefficients $E$ is the cohomology of the complex of
$\Gamma$-invariants $\rmCb^\bullet(\Gamma, E)^\Gamma$.  The bounded cohomology
$\rmHb^\bullet(\Gamma,E)$ inherits a semi-norm from $\rmCb^\bullet(\Gamma,
E)$: The \emph{(semi-)norm} of an element $x\in \rmHb^k(\Gamma, E)$ is
the infimum of the norms of all cocycles in the cohomology class $x$.

Next we briefly recall the functorial approach to bounded cohomology
as introduced by Ivanov~\cite{ivanov} for discrete groups and further
developed by Burger-Monod~\cite{burger+monod, monod-book}. We refer
for the definition of \emph{relative injectivity} of a Banach
$\Gamma$-module to~\cite{monod-book}*{Definition~4.1.2 on p.~32}. A
\emph{strong resolution} $E^\bullet$ of $E$ is a resolution, i.e.~an
acyclic complex,
\[
0\to E\to E^0\to E^1\to E^2\to\dots
\]
of Banach $\Gamma$-modules that has a chain contraction which is
contracting with respect to the Banach norms. The key to the
functorial definition of bounded cohomology are the following two
theorems:

\begin{theorem}[\cite{burger+monod}*{Proposition~1.5.2}]\label{thm:main homological theorem}
  Let $E$ and $F$ be Banach $\Gamma$-modules.  Let $E^\bullet$ be a
  strong resolution of $E$.  Let $F^\bullet$ be a resolution $F$ by
  relatively injective Banach $\Gamma$-modules. Then any
  $\Gamma$-morphism $E\to F$ extends to a $\Gamma$-morphism of
  resolutions $E^\bullet\to F^\bullet$ which is unique up to
  $\Gamma$-homotopy. Hence $E\to F$ induces functorially continuous
  linear maps $\rmH^\bullet({E^\bullet}^\Gamma)\to\rmH^\bullet({F^\bullet}^\Gamma)$.
\end{theorem}

\begin{theorem}[\cite{monod-book}*{Corollary~7.4.7 on p.~80}]\label{thm:the definining complex is strong injective resolution}
  Let $E$ be a Banach $\Gamma$-module.  The complex
  $E\to\rmCb^\bullet(\Gamma,E)$ with $E\to\rmCb^0(\Gamma,E)$ being the
  inclusion of constant functions is a strong, relatively injective
  resolution.
\end{theorem}

For a coefficient $\Gamma$-module, a measurable space $X$ with
measurable $\Gamma$-action, and a standard Borel $\Gamma$-space $S$
with quasi-invariant measure
we obtain
chain complexes $\calB^\infty(X^{\bullet+1},E)$ and $\rmLweak(S^{\bullet+1},
E)$ of Banach $\Gamma$-modules via the standard homogeneous coboundary
operators (similar as in~\eqref{eq:homogeneous differential}).

The following result is important for expressing induced maps in
bounded cohomology in terms of boundary maps~\cite{burger+iozzi}.

\begin{proposition}[\cite{burger+iozzi}*{Proposition~2.1}]\label{prop:B-complex strong resolution}
  Let $E$ be a coefficient $\Gamma$-module. Let $X$ be a measurable
  space with measurable $\Gamma$-action. The complex
  $E\to\calB^\infty(X^{\bullet+1},E)$ with $E\to\calB^\infty(X,E)$ being
  the inclusion of constant functions is a strong resolution of $E$.
\end{proposition}

The next theorem is one of the main results of the functorial approach to
bounded cohomology by Burger-Monod:

\begin{theorem}[\citelist{\cite{burger+monod}*{Corollary~2.3.2}\cite{monod-book}*{Theorem~7.5.3 on p.~83}}]\label{thm:boundary resolutions by Burger-Monod}
  Let $S$ be a regular $\Gamma$-space and be $E$ a coefficient
  $\Gamma$-module. Then $E\to\rmLweak(S^{\bullet+1}, E)$ with
  $E\to\rmLweak(S^{\bullet+1},E)$ being the inclusion of constant
  functions is a strong resolution. If, in addition, $S$ is amenable
  in the sense of Zimmer~\cite{monod-book}*{Definition~5.3.1}, then
  each $\rmLweak(S^{k+1},E)$ is relatively injective, and according to
  Theorem~\ref{thm:main homological theorem} the cohomology groups
  $\rmH^\bullet(\rmLweak(S^{\bullet+1},E)^\Gamma)$ are canonically isomorphic
  to $\rmHb^\bullet(\Gamma,E)$.
\end{theorem}

\begin{definition}\label{def:poisson transform}
  Let $S$ be a standard Borel $\Gamma$-space with a quasi-invariant
  probability measure $\mu$. Let $E$ be a coefficient $\Gamma$-module.
  The \emph{Poisson transform} $\rmPT^\bullet:
  \rmLweak(S^{\bullet+1},E)\to\rmCb^\bullet(\Gamma, E)$ is the
  $\Gamma$-morphism of chain complexes defined by
  \[
  \rmPT^k(f)(\gamma_0,\dots,\gamma_k)= \int_{S^{k+1}}
  f(\gamma_0 s_0,\dots,\gamma_k s_k)d\mu(s_0)\dots d\mu(s_k).
  \]
\end{definition}

  If $S$ is amenable, then the Poisson transform induces a canonical isomorphism in
cohomology (Theorem~\ref{thm:boundary resolutions by Burger-Monod}). By the same theorem this
isomorphism does not depend on the choice of $\mu$ within a given measure class.
 \end{appendix}


\begin{bibdiv}
\begin{biblist}

	% \bib{adams+elliot}{article}{
	%    author={Adams, Scot},
	%    author={Elliott, George A.},
	%    author={Giordano, Thierry},
	%    title={Amenable actions of groups},
	%    journal={Trans. Amer. Math. Soc.},
	%    volume={344},
	%    date={1994},
	%    number={2},
	%    pages={803--822},
	%    %issn={0002-9947},
	%    %review={\MR{1250814 (94k:22010)}},
	%    %doi={10.2307/2154508},
	% }


\bib{sobolev}{article}{
	   author={Bader, U.},
	   author={Furman, A.},
	   author={Sauer, R.},	
	   title={Efficient subdivision in hyperbolic groups and applications},
	 journal={Groups Geom. Dyn.}
	   date={2013},
	   pages={to appear},
	   eprint={arxiv:math/1003.1562},
	}

\bib{bridson}{article}{
	   author={Bridson, M. R.},
	   author={Tweedale, M.},
	   author={Wilton, H.},
	   title={Limit groups, positive-genus towers and measure-equivalence},
	   journal={Ergodic Theory Dynam. Systems},
	   volume={27},
	   date={2007},
	   number={3},
	   pages={703--712},
	   %doi={10.1017/S0143385706001039},
	}
\bib{burger+bucher+iozzi}{article}{
	   author = {Bucher, M.},
	   author = {Burger, M.},
	   author = {Iozzi, A.},
	    title = {A Dual Interpretation of the Gromov--Thurston Proof of Mostow Rigidity and Volume Rigidity for Representations of Hyperbolic Lattices},
	   note = {arXiv: 1205.1018},
	     year = {2012},
}

\bib{burger+iozzi}{article}{
   author={Burger, M.},
   author={Iozzi, A.},
   title={Boundary maps in bounded cohomology. Appendix to: "Continuous
   bounded cohomology and applications to rigidity theory" [Geom.\ Funct.\
   Anal.\ {\bf 12} (2002), no.\ 2, 219--280] by
   Burger and N. Monod},
   journal={Geom. Funct. Anal.},
   volume={12},
   date={2002},
   number={2},
   pages={281--292},
   %issn={1016-443X},
   %review={\MR{1911668 (2003d:53065b)}},
}

\bib{burger+iozzi2}{article}{
   author={Burger, M.},
   author={Iozzi, A.},
   title={A useful formula in bounded cohomology},
   status={to appear in Seminaires et Congres, nr.~18},
   eprint={http://www.math.ethz.ch/∼iozzi/grenoble.ps},
}


% \bib{burger+iozzi+wienhard}{article}{
%    author={Burger, M.},
%    author={Iozzi, A.},
%    author={Wienhard, A.}
%    title={Surface group representations with maximal Toledo invariant},
%    status={to appear in Annals of Math.},
%    note={arXiv:math/0605656v3},
% }

\bib{burger+monod}{article}{
   author={Burger, M.},
   author={Monod, N.},
   title={Continuous bounded cohomology and applications to rigidity theory},
   journal={Geom. Funct. Anal.},
   volume={12},
   date={2002},
   number={2},
   pages={219--280},
}
\bib{burger-mozes}{article}{
   author={Burger, M.},
   author={Mozes, S.},
   title={${\rm CAT}$(-$1$)-spaces, divergence groups and their
   commensurators},
   journal={J. Amer. Math. Soc.},
   volume={9},
   date={1996},
   number={1},
   pages={57--93},
}
\bib{Casson}{article}{
   author={Casson, A.},
   author={Jungreis, D.},
   title={Convergence groups and Seifert fibered $3$-manifolds},
   journal={Invent. Math.},
   volume={118},
   date={1994},
   number={3},
   pages={441--456},
   %doi={10.1007/BF01231540},
}

% \bib{cohn}{book}{
%    author={Cohn, Donald L.},
%    title={Measure theory},
%    publisher={Birkh\"auser Boston},
%    place={Mass.},
%    date={1980},
%    %pages={ix+373},
%    %isbn={3-7643-3003-1},
%    %review={\MR{578344 (81k:28001)}},
% }

\bib{corlette}{article}{
   author={Corlette, K.},
   title={Archimedean superrigidity and hyperbolic geometry},
   journal={Ann. of Math. (2)},
   volume={135},
   date={1992},
   number={1},
   pages={165--182},
   %issn={0003-486X},
   %review={\MR{1147961 (92m:57048)}},
}

\bib{corlette+zimmer}{article}{
   author={Corlette, K.},
   author={Zimmer, R. J.},
   title={Superrigidity for cocycles and hyperbolic geometry},
   journal={Internat. J. Math.},
   volume={5},
   date={1994},
   number={3},
   pages={273--290},
   %issn={0129-167X},
   %review={\MR{1274120 (95g:58055)}},
}
\bib{doran+fell}{book}{
   author={Fell, J. M. G.},
   author={Doran, R. S.},
   title={Representations of $^*$-algebras, locally compact groups, and
   Banach $^*$-algebraic bundles. Vol. 1},
   series={Pure and Applied Mathematics},
   volume={125},
   note={Basic representation theory of groups and algebras},
   publisher={Academic Press Inc.},
   %place={Boston, MA},
   date={1988},
   pages={xviii+746},
   %isbn={0-12-252721-6},
   %review={\MR{936628 (90c:46001)}},
}

\bib{fisher}{article}{
   author={Fisher, D.},
   author={Hitchman, T.},
   title={Cocycle superrigidity and harmonic maps with infinite-dimensional
   targets},
   journal={Int. Math. Res. Not.},
   date={2006},
}
\bib{fisher-nonergodic}{article}{
   author={Fisher, David},
   author={Morris, Dave Witte},
   author={Whyte, Kevin},
   title={Nonergodic actions, cocycles and superrigidity},
   journal={New York J. Math.},
   volume={10},
   date={2004},
   pages={249--269 (electronic)},
   %issn={1076-9803},
   %review={\MR{2114789 (2005i:37003)}},
}

\bib{FurmanOEw}{article}{
   author={Furman, A.},
   title={Orbit equivalence rigidity},
   journal={Ann. of Math. (2)},
   volume={150},
   date={1999},
   number={3},
   pages={1083--1108},
   %issn={0003-486X},
   %review={\MR{1740985 (2001a:22018)}},
}

\bib{FurmanME}{article}{
   author={Furman, A.},
   title={Gromov's measure equivalence and rigidity of higher rank lattices},
   journal={Ann. of Math. (2)},
   volume={150},
   date={1999},
   number={3},
   pages={1059--1081},
   %issn={0003-486X},
   %review={\MR{1740986 (2001a:22017)}},
}
\bib{locally-compact}{article}{
   author={Furman, A.},
   title={Mostow-Margulis rigidity with locally compact targets},
   journal={Geom. Funct. Anal.},
   volume={11},
   date={2001},
   number={1},
   pages={30--59},
   %issn={1016-443X},
   %review={\MR{1829641 (2002d:22015)}},
}

\bib{Furman:MGT}{article}{
   author={Furman, A.},
   title={A survey of measured group theory},
   conference={
      title={Geometry, rigidity, and group actions},
   },
   book={
      series={Chicago Lectures in Math.},
      publisher={Univ. Chicago Press},
      place={Chicago, IL},
   },
   date={2011},
   pages={296--374},
   % review={\MR{2807836 (2012e:37007)}},
} 

\bib{Furstenberg}{article}{
   author={Furstenberg, H.},
   title={Boundary theory and stochastic processes on homogeneous spaces},
   conference={
      title={Harmonic analysis on homogeneous spaces (Proc. Sympos. Pure
      Math., Vol. XXVI, Williams Coll., Williamstown, Mass., 1972)},
   },
   book={
      publisher={Amer. Math. Soc.},
      place={Providence, R.I.},
   },
   date={1973},
   pages={193--229},
}
\bib{furstenberg-bourbaki}{article}{
   author={Furstenberg, H.},
   title={Rigidity and cocycles for ergodic actions of semisimple Lie groups
   (after G. A. Margulis and R. Zimmer)},
   conference={
      title={Bourbaki Seminar, Vol. 1979/80},
   },
   book={
      series={Lecture Notes in Math.},
      volume={842},
      publisher={Springer},
      place={Berlin},
   },
   date={1981},
   pages={273--292},
}
\bib{Gabai}{article}{
   author={Gabai, D.},
   title={Convergence groups are Fuchsian groups},
   journal={Ann. of Math. (2)},
   volume={136},
   date={1992},
   number={3},
   pages={447--510},
   %doi={10.2307/2946597},
}

\bib{Gaboriau-cost}{article}{
   author={Gaboriau, D.},
   title={Co\^ut des relations d'\'equivalence et des groupes},
   % language={French, with English summary},
   journal={Invent. Math.},
   volume={139},
   date={2000},
   number={1},
   pages={41--98},
}
\bib{gaboriau-l2}{article}{
   author={Gaboriau, D.},
   title={Invariants $l\sp 2$ de relations d'\'equivalence et de groupes},
   language={French},
   journal={Publ. Math. Inst. Hautes \'Etudes Sci.},
   number={95},
   date={2002},
   pages={93--150},
   %issn={0073-8301},
   %review={\MR{1953191 (2004b:22009)}},
}

\bib{Gaboriau:2005exmps}{article}{
   author={Gaboriau, D.},
   title={Examples of groups that are measure equivalent to the free group},
   journal={Ergodic Theory Dynam. Systems},
   volume={25},
   date={2005},
   number={6},
   pages={1809\ndash 1827},
   % issn={0143-3857},
   % review={\MR{2183295 (2006i:22024)}},
}
% \bib{GhH}{collection}{
%    title={Sur les groupes hyperboliques d'apr\`es Mikhael Gromov},
%    %language={French},
%    series={Progress in Mathematics},
%    volume={83},
%    editor={Ghys, {\'E}.},
%    editor={de la Harpe, Pierre},
%    note={Papers from the Swiss Seminar on Hyperbolic Groups held in Bern,
%    1988},
%    publisher={Birkh\"auser Boston Inc.},
%    place={Boston, MA},
%    date={1990},
%    %pages={xii+285},
%    %isbn={0-8176-3508-4},
%    %review={\MR{1086648 (92f:53050)}},
% }
\bib{ghys}{article}{
   author={Ghys, {\'E}.},
   title={Groups acting on the circle},
   journal={Enseign. Math. (2)},
   volume={47},
   date={2001},
   number={3-4},
   pages={329--407},
}
% \bib{gromov}{article}{
%    author={Gromov, M.},
%    title={Volume and bounded cohomology},
%    journal={Inst. Hautes \'Etudes Sci. Publ. Math.},
%    number={56},
%    date={1982},
%    pages={5--99 (1983)},
%    %issn={0073-8301},
%    %review={\MR{686042 (84h:53053)}}}
% }
\bib{gromov-invariants}{article}{
   author={Gromov, M.},
   title={Asymptotic invariants of infinite groups},
   conference={
      title={Geometric group theory, Vol.\ 2},
      address={Sussex},
      date={1991},
   },
   book={
      series={London Math. Soc. Lecture Note Ser.},
      volume={182},
      publisher={Cambridge Univ. Press},
      place={Cambridge},
   },
   date={1993},
   pages={1--295},
   %review={\MR{1253544 (95m:20041)}},
}

\bib{hjorth}{article}{
   author={Hjorth, G.},
   title={A converse to Dye's theorem},
   journal={Trans. Amer. Math. Soc.},
   volume={357},
   date={2005},
   number={8},
   pages={3083--3103 (electronic)},
   %issn={0002-9947},
   %review={\MR{2135736 (2005m:03093)}},
}

% \bib{Hjorth+Kechris}{article}{
%       author={Hjorth, G.},
%       author={Kechris, A.~S.},
%        title={Rigidity theorems for actions of product groups and countable
%   {B}orel equivalence relations},
%         date={2005},
%      journal={Mem. Amer. Math. Soc.},
%       volume={177},
%       number={833},
%        pages={viii+109},
%         % ISSN={0065-9266},
%    %   review={\MR{MR2155451}},
% }

\bib{haagerup}{article}{
   author={Haagerup, U.},
   author={Munkholm, H. J.},
   title={Simplices of maximal volume in hyperbolic $n$-space},
   journal={Acta Math.},
   volume={147},
   date={1981},
   number={1-2},
   pages={1--11},
   % issn={0001-5962},
   % review={\MR{631085 (82j:53116)}},
}
\bib{Ioana+Peterson+Popa}{article}{
   author={Ioana, A.},
   author={Peterson, J.},
   author={Popa, S.},
   title={Amalgamated free products of weakly rigid factors and calculation
   of their symmetry groups},
   journal={Acta Math.},
   volume={200},
   date={2008},
   number={1},
   pages={85--153},
}
\bib{ioana}{article}{
   author={Ioana, Adrian},
   title={Cocycle superrigidity for profinite actions of property (T)
   groups},
   journal={Duke Math. J.},
   volume={157},
   date={2011},
   number={2},
   pages={337--367},
   % issn={0012-7094},
   % review={\MR{2783933 (2012c:37006)}},
   % doi={10.1215/00127094-2011-008},
}

\bib{iozzi}{article}{
   author={Iozzi, Alessandra},
   title={Bounded cohomology, boundary maps, and rigidity of representations
   into ${\rm Homeo}_+(\bold S^1)$ and ${\rm SU}(1,n)$},
   conference={
      title={Rigidity in dynamics and geometry},
      address={Cambridge},
      date={2000},
   },
   book={
      publisher={Springer},
      %place={Berlin},
   },
   date={2002},
   pages={237--260},
   %review={\MR{1919404 (2003g:22008)}},
}

\bib{ivanov}{article}{
   author={Ivanov, N. V.},
   title={Foundations of the theory of bounded cohomology},
   language={Russian, with English summary},
   %note={Studies in topology, V},
   journal={Zap. Nauchn. Sem. Leningrad. Otdel. Mat. Inst. Steklov. (LOMI)},
   volume={143},
   date={1985},
   pages={69--109, 177--178},
   %issn={0373-2703},
   %review={\MR{806562 (87b:53070)}},
}
\bib{kechris}{book}{
   author={Kechris, Alexander S.},
   title={Classical descriptive set theory},
   series={Graduate Texts in Mathematics},
   volume={156},
   publisher={Springer-Verlag},
   %place={New York},
   date={1995},
   %pages={xviii+402},
   %isbn={0-387-94374-9},
   %review={\MR{1321597 (96e:03057)}},
}

\bib{Kida:thesis}{article}{
   author={Kida, Y.},
   title={The mapping class group from the viewpoint of measure equivalence theory},
   journal={Mem. Amer. Math. Soc.},
   volume={196},
   date={2008},
   number={916},
   pages={viii+190},
}
\bib{Kida:ME}{article}{
   author={Kida, Yoshikata},
   title={Measure equivalence rigidity of the mapping class group},
   journal={Ann. of Math. (2)},
   volume={171},
   date={2010},
   number={3},
   pages={1851--1901},
   % issn={0003-486X},
   % review={\MR{2680399 (2011e:37009)}},
   % doi={10.4007/annals.2010.171.1851},
}
\bib{Kida:OE}{article}{
   author={Kida, Yoshikata},
   title={Orbit equivalence rigidity for ergodic actions of the mapping
   class group},
   journal={Geom. Dedicata},
   volume={131},
   date={2008},
   pages={99--109},
   %issn={0046-5755},
   %review={\MR{2369194 (2008k:37011)}},
   %doi={10.1007/s10711-007-9219-8},
}

\bib{Kida:amalgamated}{article}{
   author={Kida, Yoshikata},
   title={Rigidity of amalgamated free products in measure equivalence},
   journal={J. Topol.},
   volume={4},
   date={2011},
   number={3},
   pages={687--735},
   % issn={1753-8416},
   % review={\MR{2832574 (2012i:20035)}},
   % doi={10.1112/jtopol/jtr012},
}
%
% \bib{lanner}{article}{
%    author={Lann{\'e}r, F.},
%    title={On complexes with transitive groups of automorphisms},
%    journal={Comm. S\'em., Math. Univ. Lund [Medd. Lunds Univ. Mat. Sem.]},
%    volume={11},
%    date={1950},
%    pages={71},
%    % review={\MR{0042129 (13,58c)}},
% }
\bib{loeh}{article}{
   author={L{\"o}h, C.},
   title={Measure homology and singular homology are isometrically
   isomorphic},
   journal={Math.~Z.},
   volume={253},
   date={2006},
   number={1},
   pages={197--218},
   %issn={0025-5874},
   %review={\MR{2206643 (2006m:55021)}},
}
%
% \bib{loeh-iso}{article}{
%    author={L{\"o}h, C.},
%    title={Isomorphisms in $l\sp 1$-homology},
%    journal={M\"unster J. Math.},
%    volume={1},
%    date={2008},
%    number={1},
%    pages={237--265},
%    %issn={1867-5778},
%    %review={\MR{2502500}},
% }

\bib{Margulis-Mostow}{article}{
   author={Margulis, G. A.},
   title={Non-uniform lattices in semisimple algebraic groups},
   conference={
      title={Lie groups and their representations (Proc. Summer School on
      Group Representations of the Bolyai J\'anos Math. Soc., Budapest,
      1971)},
   },
   book={
      publisher={Halsted, New York},
   },
   date={1975},
   pages={371--553},
   %review={\MR{0422499 (54 \#10486)}},
}
\bib{Margulis:1974:ICM}{article}{
   author={Margulis, G. A.},
   title={Discrete groups of motions of manifolds of nonpositive curvature},
   language={Russian},
   conference={
      title={Proceedings of the International Congress of Mathematicians
      (Vancouver, B.C., 1974), Vol. 2},
   },
   book={
      publisher={Canad. Math. Congress, Montreal, Que.},
   },
   date={1975},
   pages={21--34},
   % review={\MR{0492072 (58 \#11226)}},
}

\bib{Margulis:book}{book}{
      author={Margulis, G. A.},
       title={Discrete subgroups of semisimple {L}ie groups},
      series={Ergebnisse der Mathematik und ihrer Grenzgebiete (3) [Results in
  Mathematics and Related Areas (3)]},
   publisher={Springer-Verlag},
     address={Berlin},
        date={1991},
      volume={17},
      %   ISBN={3-540-12179-X},
      % review={\MR{MR1090825 (92h:22021)}},
}

\bib{monod-book}{book}{
   author={Monod, N.},
   title={Continuous bounded cohomology of locally compact groups},
   series={Lecture Notes in Mathematics},
   volume={1758},
   publisher={Springer-Verlag},
   place={Berlin},
   date={2001},
   pages={x+214},
   isbn={3-540-42054-1},
   %review={\MR{1840942 (2002h:46121)}},
}

\bib{MonodShalom}{article}{
   author={Monod, N.},
   author={Shalom, Y.},
   title={Orbit equivalence rigidity and bounded cohomology},
   journal={Ann. of Math. (2)},
   volume={164},
   date={2006},
   number={3},
   pages={825--878},
   %issn={0003-486X},
   %review={\MR{2259246 (2007k:37007)}},
}
\bib{MonodShalom-cocycle}{article}{
   author={Monod, Nicolas},
   author={Shalom, Yehuda},
   title={Cocycle superrigidity and bounded cohomology for negatively curved
   spaces},
   journal={J. Differential Geom.},
   volume={67},
   date={2004},
   number={3},
   pages={395--455},
   %issn={0022-040X},
   %review={\MR{2153026 (2006g:53051)}},
}


\bib{Mostow:1973book}{book}{
   author={Mostow, G. D.},
   title={Strong rigidity of locally symmetric spaces},
   note={Annals of Mathematics Studies, No. 78},
   publisher={Princeton University Press},
   place={Princeton, N.J.},
   date={1973},
   pages={v+195},
   % review={\MR{0385004 (52 \#5874)}},
}
\bib{pansu+zimmer}{article}{
   author={Pansu, P.},
   author={Zimmer, R. J.},
   title={Rigidity of locally homogeneous metrics of negative curvature on
   the leaves of a foliation},
   journal={Israel J. Math.},
   volume={68},
   date={1989},
   number={1},
   pages={56--62},
   % issn={0021-2172},
   % review={\MR{1035880 (91e:58144)}},
   % doi={10.1007/BF02764968},
}

\bib{popa-betti}{article}{
   author={Popa, S.},
   title={On a class of type ${\rm II}\sb 1$ factors with Betti numbers
   invariants},
   journal={Ann. of Math. (2)},
   volume={163},
   date={2006},
   number={3},
   pages={809--899},
   %issn={0003-486X},
   %review={\MR{2215135 (2006k:46097)}},
}
\bib{popa:wI}{article}{
      author={Popa, S.},
       title={Strong rigidity of {$\rm II\sb 1$} factors arising from malleable
  actions of {$w$}-rigid groups. {I}},
        date={2006},
     journal={Invent. Math.},
      volume={165},
      number={2},
       pages={369\ndash 408},
}
\bib{popa:wII}{article}{
      author={Popa, S.},
       title={Strong rigidity of {$\rm II\sb 1$} factors arising from malleable
  actions of {$w$}-rigid groups. {II}},
        date={2006},
     journal={Invent. Math.},
      volume={165},
      number={2},
       pages={409\ndash 451},
}
\bib{popa-cocycle}{article}{
   author={Popa, S.},
   title={Cocycle and orbit equivalence superrigidity for malleable actions
   of $w$-rigid groups},
   journal={Invent. Math.},
   volume={170},
   date={2007},
   number={2},
   pages={243--295},
   %issn={0020-9910},
   %review={\MR{2342637 (2008f:37010)}},
}
\bib{Popa:ICM}{article}{
   author={Popa, S.},
   title={Deformation and rigidity for group actions and von Neumann
   algebras},
   conference={
      title={International Congress of Mathematicians. Vol. I},
   },
   book={
      publisher={Eur. Math. Soc., Z\"urich},
   },
   date={2007},
   pages={445--477},
}
\bib{popa-spectralgap}{article}{
   author={Popa, S.},
   title={On the superrigidity of malleable actions with spectral gap},
   journal={J. Amer. Math. Soc.},
   volume={21},
   date={2008},
   number={4},
   pages={981--1000},
   %issn={0894-0347},
   %review={\MR{2425177 (2009e:46056)}},
}
\bib{prasad}{article}{
   author={Prasad, G.},
   title={Strong rigidity of ${\bf Q}$-rank $1$ lattices},
   journal={Invent. Math.},
   volume={21},
   date={1973},
   pages={255--286},
}
\bib{ratcliffe}{book}{
   author={Ratcliffe, J. G.},
   title={Foundations of hyperbolic manifolds},
   series={Graduate Texts in Mathematics},
   volume={149},
   publisher={Springer-Verlag},
   place={New York},
   date={1994},
   %pages={xii+747},
   %isbn={0-387-94348-X},
   %review={\MR{1299730 (95j:57011)}},
}
% \bib{ratner}{article}{
%    author={Ratner, M.},
%    title={Raghunathan's conjectures for Cartesian products of real and
%    $p$-adic Lie groups},
%    journal={Duke Math. J.},
%    volume={77},
%    date={1995},
%    number={2},
%    pages={275--382},
%    % issn={0012-7094},
%    % review={\MR{1321062 (96d:22015)}},
% }

\bib{ratnerICM}{article}{
   author={Ratner, M.},
   title={Interactions between ergodic theory, Lie groups, and number
   theory},
   conference={
      title={ 2},
      address={Z\"urich},
      date={1994},
   },
   book={
      publisher={Birkh\"auser},
      place={Basel},
   },
   date={1995},
   pages={157--182},
   % review={\MR{1403920 (98k:22046)}},
}

\bib{shalom}{article}{
   author={Shalom, Y.},
   title={Rigidity, unitary representations of semisimple groups, and
   fundamental groups of manifolds with rank one transformation group},
   journal={Ann. of Math. (2)},
   volume={152},
   date={2000},
   number={1},
   pages={113--182},
   %issn={0003-486X},
   %review={\MR{1792293 (2001m:22022)}},
}

\bib{Shalom:2005ECM}{article}{
   author={Shalom, Y.},
   title={Measurable group theory},
   conference={
      title={European Congress of Mathematics},
   },
   book={
      publisher={Eur. Math. Soc., Z\"urich},
   },
   date={2005},
   pages={391--423},
   % review={\MR{2185757 (2006k:37007)}},
}

\bib{stroppel}{book}{
   author={Stroppel, Markus},
   title={Locally compact groups},
   series={EMS Textbooks in Mathematics},
   publisher={European Mathematical Society (EMS), Z\"urich},
   date={2006},
   pages={x+302},
   %isbn={3-03719-016-7},
   %review={\MR{2226087 (2007d:22001)}},
   %doi={10.4171/016},
}

\bib{tits-trees}{article}{
   author={Tits, Jacques},
   title={Sur le groupe des automorphismes d'un arbre},
   %language={French},
   conference={
      title={Essays on topology and related topics (M\'emoires d\'edi\'es
      \`a Georges de Rham)},
   },
   book={
      publisher={Springer},
      %place={New York},
   },
   date={1970},
   %pages={188--211},
   %review={\MR{0299534 (45 \#8582)}},
}

\bib{thurston}{article}{
  author={Thurston, W. P.},
  title={The Geometry and Topology of Three-Manifolds},
  date={1978},
  eprint={http://www.msri.org/publications/books/gt3m},
}
%
% \bib{witte}{article}{
%    author={Witte, D.},
%    title={Measurable quotients of unipotent translations on homogeneous
%    spaces},
%    journal={Trans. Amer. Math. Soc.},
%    volume={345},
%    date={1994},
%    number={2},
%    pages={577--594},
%    % issn={0002-9947},
%    % review={\MR{1181187 (95a:22005)}},
% }

\bib{zastrow}{article}{
   author={Zastrow, A.},
   title={On the (non)-coincidence of Milnor-Thurston homology theory with
   singular homology theory},
   journal={Pacific J. Math.},
   volume={186},
   date={1998},
   number={2},
   pages={369--396},
   %issn={0030-8730},
   %review={\MR{1663826 (2000a:55008)}},
}

\bib{zimmer-csr}{article}{
   author={Zimmer, R. J.},
   title={Strong rigidity for ergodic actions of semisimple Lie groups},
   journal={Ann. of Math. (2)},
   volume={112},
   date={1980},
   number={3},
   pages={511--529},
   %issn={0003-486X},
   %review={\MR{595205 (82i:22011)}},
}

\bib{zimmer-foliatedmostow}{article}{
   author={Zimmer, R. J.},
   title={On the Mostow rigidity theorem and measurable foliations by
   hyperbolic space},
   journal={Israel J. Math.},
   volume={43},
   date={1982},
   number={4},
   pages={281--290},
   % issn={0021-2172},
   % review={\MR{693350 (85b:22023)}},
   % doi={10.1007/BF02761234},
}
% \bib{zimmer-analogue}{article}{
%    author={Zimmer, R. J.},
%    title={An analogue of the Mostow-Margulis rigidity theorems for ergodic
%    actions of semisimple Lie groups},
%    journal={Bull. Amer. Math. Soc. (N.S.)},
%    volume={2},
%    date={1980},
%    number={1},
%    pages={168--170},
%    %issn={0273-0979},
%    %review={\MR{551755 (80m:22016)}},
% }

\bib{zimmer-book}{book}{
   author={Zimmer, R. J.},
   title={Ergodic theory and semisimple groups},
   series={Monographs in Mathematics},
   volume={81},
   publisher={Birkh\"auser Verlag},
   place={Basel},
   date={1984},
   pages={x+209},
   %isbn={3-7643-3184-4},
   %review={\MR{776417 (86j:22014)}},
}

% \bib{Zimmer-transv}{article}{
%    author={Zimmer, R.~J.},
%    title={Groups generating transversals to semisimple Lie group actions},
%    journal={Israel J. Math.},
%    volume={73},
%    date={1991},
%    number={2},
%    pages={151--159},
%    % issn={0021-2172},
%    % review={\MR{1135209 (93e:22012)}},
% }

\end{biblist}
\end{bibdiv}
\end{document}